\documentclass[10pt]{amsart}

\usepackage{amsmath, amssymb, amsthm, amsfonts}

\input xy
\xyoption{all}

\usepackage{tikz-cd}
\tikzset{
  symbol/.style={
    draw=none,
    every to/.append style={
      edge node={node [sloped, allow upside down, auto=false]{$#1$}}}
  }
}

\usetikzlibrary{intersections}

\usepackage{varwidth}

\usepackage{hyperref}

\usepackage{color}

\usepackage{enumitem}

\swapnumbers
\numberwithin{equation}{section}

\theoremstyle{plain}
\newtheorem{theorem}[subsubsection]{Theorem}
\newtheorem{lemma}[subsubsection]{Lemma}
\newtheorem{prop}[subsubsection]{Proposition}
\newtheorem{cor}[subsubsection]{Corollary}

\theoremstyle{definition}
\newtheorem{defn}[subsubsection]{Definition}

\newtheorem{remark}[subsubsection]{Remark}
\newtheorem{exam}[subsubsection]{Example}

\def\AA{\mathbb{A}}
\def\BB{\mathbb{B}}
\def\CC{\mathbb{C}}
\def\DD{\mathbb{D}}

\def\FF{\mathbb{F}}
\def\GG{\mathbb{G}}
\def\HH{\mathbb{H}}

\def\NN{\mathbb{N}}

\def\PP{\mathbb{P}}
\def\QQ{\mathbb{Q}}
\def\RR{\mathbb{R}}

\def\TT{\mathbb{T}}

\def\WW{\mathbb{W}}

\def\ZZ{\mathbb{Z}}

\def\calA{\mathcal{A}}

\def\calD{\mathcal{D}}
\def\calE{\mathcal{E}}
\def\calF{\mathcal{F}}

\def\calK{\mathcal{K}}

\def\calO{\mathcal{O}}
\def\calP{\mathcal{P}}

\def\calV{\mathcal{V}}

\newcommand\cA{\mathcal{A}}
\newcommand\cB{\mathcal{B}}
\newcommand\cC{\mathcal{C}}
\newcommand\cD{\mathcal{D}}
\newcommand\cE{\mathcal{E}}
\newcommand\cF{\mathcal{F}}

\newcommand\cH{\mathcal{H}}
\newcommand\cI{\mathcal{I}}
\newcommand\cJ{\mathcal{J}}
\newcommand\cK{\mathcal{K}}

\newcommand\cO{\mathcal{O}}
\newcommand\cP{\mathcal{P}}

\newcommand\cR{\mathcal{R}}
\newcommand\cS{\mathcal{S}}

\newcommand\cU{\mathcal{U}}
\newcommand\cV{\mathcal{V}}

\newcommand\cZ{\mathcal{Z}}

\def\bP{\mathbf{P}}
\def\bQ{\mathbf{Q}}

\def\bT{\mathbf{T}}

\newcommand\frA{\mathfrak{A}}

\newcommand\frF{\mathfrak{F}}

\newcommand\frg{\mathfrak{g}}
\newcommand\frh{\mathfrak{h}}

\newcommand\frakm{\mathfrak{m}}

\newcommand\frp{\mathfrak{p}}
\newcommand\frq{\mathfrak{q}}

\newcommand\frt{\mathfrak{t}}

\newcommand\tilW{\widetilde{W}}

\newcommand\dG{\widehat{G}}
\newcommand\dT{\widehat{T}}
\newcommand\dH{\widehat{H}}

\newcommand\dS{\widehat{S}}

\newcommand\dg{\widehat{\mathfrak{g}}}
\newcommand\dt{\widehat{\mathfrak{t}}}

\newcommand\LG{\leftexp{L}{G}}

\newcommand\aff{\textup{aff}}

\newcommand\AS{\textup{AS}}

\newcommand{\Bun}{\textup{Bun}}

\newcommand{\ch}{\textup{char}}

\newcommand{\coker}{\textup{coker}}

\newcommand\ev{\textup{ev}}

\newcommand{\Fl}{\textup{Fl}}

\newcommand\Fr{\textup{Fr}}

\newcommand\Fun{\textup{Fun}}
\newcommand\Gal{\textup{Gal}}

\newcommand{\Gr}{\textup{Gr}}

\newcommand{\Hk}{\textup{Hk}}

\newcommand\IC{\textup{IC}}
\newcommand\id{\textup{id}}
\renewcommand{\Im}{\textup{Im}}
\newcommand{\Ind}{\textup{Ind}}
\newcommand{\ind}{\textup{ind}}

\newcommand\Irr{\textup{Irr}}

\newcommand\Lie{\textup{Lie }}
\newcommand\Loc{\textup{Loc}}

\newcommand{\mult}{\textup{mult}}

\newcommand\Out{\textup{Out}}

\newcommand\Perv{\textup{Perv}}

\newcommand\pt{\textup{pt}}

\newcommand{\re}{\textup{re}}
\newcommand{\red}{\textup{red}}

\newcommand\Rep{\textup{Rep}}
\newcommand{\Res}{\textup{Res}}

\newcommand\rk{\textup{rk}}

\newcommand\Spec{\textup{Spec}\ }
\newcommand\St{\textup{St}}

\newcommand\Stab{\textup{Stab}}

\newcommand\Swan{\textup{Swan}}
\newcommand\Sym{\textup{Sym}}
\newcommand{\tame}{\textup{tame}}

\newcommand{\Vect}{\textup{Vect}}

\newcommand\Aut{\textup{Aut}}
\newcommand\Hom{\textup{Hom}}
\newcommand\End{\textup{End}}

\newcommand\uHom{\underline{\Hom}}
\newcommand\uEnd{\underline{\End}}

\newcommand\GL{\textup{GL}}

\newcommand\PGL{\textup{PGL}}

\newcommand\SL{\textup{SL}}

\newcommand\SO{\textup{SO}}

\newcommand\PSO{\textup{PSO}}

\newcommand\Sp{\textup{Sp}}

\newcommand\Spin{\textup{Spin}}

\newcommand{\Gm}{\GG_m}
\newcommand{\Gmg}{\GG_{m,\ov k}}
\def\Ga{\GG_a}

\newcommand{\Ad}{\textup{Ad}}

\newcommand\xch{\mathbb{X}^*}
\newcommand\xcoch{\mathbb{X}_*}

\newcommand{\incl}{\hookrightarrow}
\newcommand{\isom}{\stackrel{\sim}{\to}}

\newcommand{\surj}{\twoheadrightarrow}
\newcommand{\leftexp}[2]{{\vphantom{#2}}^{#1}{#2}}
\newcommand{\pH}{\leftexp{p}{\textup{H}}}

\newcommand{\Qlbar}{\overline{\QQ}_\ell}

\renewcommand{\j}[1]{\langle{#1}\rangle}
\newcommand{\wt}[1]{\widetilde{#1}}

\newcommand\quash[1]{}

\newcommand{\bu}{\bullet}
\newcommand{\ov}{\overline}
\newcommand{\bs}{\backslash}
\newcommand{\tl}[1]{[\![#1]\!]}
\newcommand{\lr}[1]{(\!(#1)\!)}
\newcommand\sss{\subsubsection}
\newcommand\xr{\xrightarrow}
\newcommand\op{\oplus}
\newcommand\ot{\otimes}
\newcommand\bt{\boxtimes}
\newcommand\one{\mathbf{1}}

\newcommand{\vn}{\varnothing}

\newcommand{\cohog}[2]{\textup{H}^{#1}({#2})}     
\newcommand{\cohoc}[2]{\textup{H}_{c}^{#1}({#2})}     

\newcommand\upH{\textup{H}}

\newcommand{\oll}[1]{\overleftarrow{#1}}
\newcommand{\orr}[1]{\overrightarrow{#1}}

\renewcommand\a\alpha
\renewcommand\b\beta
\newcommand\g\gamma
\newcommand\G\Gamma
\renewcommand\d\delta
\newcommand{\D}{\Delta}

\renewcommand{\th}{\theta}

\newcommand{\ph}{\varphi}
\renewcommand{\r}{\rho}
\newcommand{\s}{\sigma}

\newcommand{\y}{\eta}
\newcommand{\z}{\zeta}

\renewcommand{\l}{\lambda}

\newcommand{\om}{\omega}
\newcommand{\Om}{\Omega}

\newcommand{\kbar}{\overline{k}}

\newcommand{\Zm}{\ZZ/m\ZZ}

\newcommand{\ls}[1]{k(\!({#1})\!)}

\newcommand{\abar}{\overline{\alpha}}

\title{Euphotic representations and rigid automorphic data}

\dedicatory{}
\author{Konstantin Jakob}
\thanks{K.J. is supported by the DFG Research Fellowship JA 2967/1-1. Z.Y. is partially supported by the Packard Foundation and the Simons Foundation.}
\address{(K.J.) Department of Mathematics, Massachusetts Institute of Technology, 77 Massachusetts Ave, Cambridge, MA 02139 }
\email{kjakob@mit.edu}
\author{Zhiwei Yun}
\address{(Z.Y.) Department of Mathematics, Massachusetts Institute of Technology, 77 Massachusetts Ave, Cambridge, MA 02139}
\email{zyun@mit.edu}
\date{}
\subjclass[2020]{Primary: 11F70, 14D24, 22E57; Secondary: 14F06}
\keywords{Langlands correspondence, rigid local systems, automorphic representations, Hessenberg varieties}

\begin{document}

\begin{abstract}
We propose a new method to construct rigid $G$-automorphic representations and rigid $\dG$-local systems for reductive groups $G$. The construction involves the notion of {\em euphotic representations}, and the proof for rigidity involves the geometry of certain Hessenberg varieties. 
\end{abstract}

\maketitle

\tableofcontents
\section{Introduction}

\subsection{Rigid local systems and automorphic representations}
Rigid local systems on a punctured curve are those that don't admit deformations that preserve the local monodromy at the punctures. Many well-known local systems in arithmetic are rigid, e.g., Kloosterman sheaves and hypergeometric sheaves. Katz (\cite{Ka88} and \cite{Ka96}) studied rigid local systems systematically, and he gave an algorithm for producing all tame rigid local systems of arbitrary rank. This algorithm has been extended by Arinkin \cite{Arinkin10} (and Deligne) to cover all rigid local systems.

For a reductive group $H$ over $\Qlbar$ an $H$-local systems on a curve $U$ is a homomorphism from $\pi_{1}(U,u)$ to $H(\Qlbar)$. There is a notion of rigidity for $H$-local systems generalizing the rigidity for rank $n$ local systems. Much less is known about rigid $H$-local systems for general $H$. 

Rigid local systems have seen application in inverse Galois theory and in the construction of motives with exceptional Galois groups, see \cite{Yun14}, \cite{DR00}, \cite{DR10}. In particular it is of interest to construct many examples of rigid $H$-local systems, especially for exceptional groups $H$. 

While \cite{DR10} use the Katz algorithm to construct and classify rigid $G_2$-local systems, this algorithm is unavailable for a general reductive group $H$. Even for rank $n$ local systems it is often computationally and technically involved. In a series of works (\cite{HNY}, \cite{Yun14} and \cite{Yun16}) a new method of constructing rigid $H$-local systems is developed, and many examples are given. The key new ingredient in that method is to use the Langlands correspondence for function fields to transport the problem of constructing rigid $H$-local systems into constructing rigid $G$-automorphic representations. Here $H$ is identified with the Langlands dual group of $G$. 

This method has several advantages. When trying construct a rigid $H$-local system using the Katz algorithm one has to construct a rank $n$ local system and impose conditions on it to force its global monodromy to lie in $H$. Not all rigid $H$-local systems can be obtained in this way. 

In the geometric Langlands approach one may directly construct $H$-local systems. In addition it turns out that rigid automorphic representations are sometimes easier to obtain, and techniques from the geometric Langlands program are crucial in passing from rigid automorphic representations to local systems.

This paper aims to expand the zoo of rigid $G$-automorphic representations and rigid $\dG$-local systems by generalizing the construction of \cite{Yun16}. We consider $F=k(t)$, the function field of $\PP^{1}$ over $k=\FF_{q}$. The $G$-automorphic representations we construct have depth zero at $0$ and positive depth $1/m$ at $\infty$. The corresponding $\dG$-local systems over $\PP^{1}\bs\{0,\infty\}$ are tamely ramified at $0$ and wildly ramified at $\infty$, and are expected to be rigid.

\subsection{Euphotic representations}
In this introduction let $G$ be a split almost simple group over the function field $F=k(t)$ of $\PP^{1}_{k}$. Let $F_{0},F_{\infty}$ be the local fields of $F$ at $0,\infty\in \PP^{1}$.

The starting point of  \cite{Yun16} is a class of supercuspidal representations of $G(F_{\infty})$  introduced by Reeder and Yu \cite{RY14} called {\em epipelagic representations}. They generalize an earlier construction of {\em simple supercuspidal representations} by Gross and Reeder \cite{GR10} that motivated the construction of Kloosterman sheaves in \cite{HNY}.  

In this paper, we define a more general class of representations of the $p$-adic group $G(F_{\infty})$ than epipelagic ones which we call {\em euphotic} representations.\footnote{In oceanography euphotic is synonymous to epipelagic, stressing the role of light. Depending on the transparency of the ocean water the euphotic zone may vary in depth - in analogy there are more possibilities for the depth of a euphotic representation than for the depth of an epipelagic representation. } The data needed to construct a euphotic representation is a triple $(\bP_{\infty}, \psi,\chi)$. Here
\begin{itemize}
\item $\bP_{\infty}$ is a parahoric subgroup of $G(F_{\infty})$;
\item $\psi$ is a linear function on the vector space $V_{\bP}=\bP^{+}_{\infty}/\bP^{++}_{\infty}$ (the first nontrivial associated graded of the pro-unipotent radical $\bP_{\infty}^{+}$ under the Moy-Prasad filtration). We require $\psi$ to be {\em semisimple} in the sense that its orbit under $L_{\bP}=\bP_{\infty}/\bP^{+}_{\infty}$ is closed.
\item Let $L_{\psi}$ be the stabilizer of $\psi$ under $L_{\bP}$, and $B_{\psi}$ be a Borel subgroup of $L_{\psi}$ with quotient Cartan $T_{\psi}$. Then $\chi$ is a character $\chi: T_{\psi}(k)\to \Qlbar^{\times}$. 
\end{itemize}
A euphotic representation $\pi$ of type $(\bP_{\infty},\psi,\chi)$ is an irreducible representation of $G(F_{\infty})$ that contains an eigenvector under $B_{\psi}(k)\bP^{+}_{\infty}$ on which $B_{\psi}(k)$ acts via $\chi$ (inflated from $T_{\psi}(k)$) and $\bP^{+}_{\infty}$ acts via $\psi_{k}\circ\psi$ (inflated from $V_{\bP}$, and $\psi_{k}:k\to \Qlbar^{\times}$ is a fixed nontrivial additive character).

Compared to the notion of epipelagic representations in \cite{RY14}, we have relaxed the condition on $\psi$: it is only required to have a closed orbit under $L_{\bP}$ and {\em not} required to have finite stabilizer under $L_{\bP}$. Functionals on $V_{\bP}$ with closed orbit that are not stable are also encountered in work of Kamgarpour and Yi on the geometric Langlands correspondence for hypergeometric sheaves \cite{KY20}.


\subsection{Euphotic automorphic data}
To construct rigid $G$-automorphic representations, we start with a triple $(\bP_{\infty}, \psi,\chi)$ as above, and choose a parahoric subgroup $\bQ_{0}$ of $G(F_{0})$. We impose several conditions on these data (see Definition \ref{d:qep data}) which are of geometric nature (i.e., they only have to do with the situation over $\ov k$). Among these conditions is the requirement that certain Hessenberg varieties coming from cyclic gradings on $\frg$ have a stabilizer property under a group action, which we call ``spectrally meager'' (see Definition \ref{d:sp meager}), a notion that we believe is of independent interest.

We prove that a euphotic automorphic datum $(\bP_{\infty}, \psi,\chi, \bQ_{0})$ is weakly rigid in the following sense: there is a small (but nonzero) number of irreducible cuspidal automorphic representations $\pi$ of $G(\AA_{F})$ such that $\pi_{\infty}$ is  euphotic of type $(\bP_{\infty},\psi,\chi)$, $\pi_{0}$ contains a nonzero $\bQ_{0}$-fixed vector, and $\pi$ is unramified otherwise; and the number of such cuspidal automorphic representations is uniformly bounded when $k$ is replaced with any finite extension.   This is proved by analyzing the space of automorphic functions cut out by the eigen-conditions at $0$ and $\infty$, and Hessenberg varieties naturally show up in this analysis.

\subsection{Hecke eigensheaves and local systems} To construct the $\dG$-local systems out of these automorphic representations, we consider automorphic sheaves instead of functions.  The automorphic datum $(\bP_{\infty}, \psi,\chi, \bQ_{0})$ gives rise to an abelian category $\cP=\cP(\psi,\chi)$ of perverse sheaves on a certain moduli stack of $G$-bundles on $\PP^{1}_{\ov k}$ with level structures given by $\bP^{++}_{\infty}$ and $\bQ_{0}$. This category has only finitely many simple objects, which is an indication of rigidity. 

Here comes a crucial difference with all previous work on rigid automorphic representations. Previously the analogous categories $\cP$ always decomposed into Hecke-stable pieces with one simple object (a Hecke eigensheaf) in each piece, and the framework of \cite{Yun14} allowed us to extract a $\dG$-local system from each Hecke eigensheaf. In the euphotic situation, the category $\cP$ sometimes has more than one simple object yet there isn't an obvious way to decompose it further.  We remark that this is a feature rather than a bug: it is likely that these more complicated categories $\cP$ give nontrivial global $L$-packets. To deal with this situation, we extend the framework of \cite{Yun14} to extract eigen local systems from a Hecke eigencategory rather than a Hecke eigensheaf.  The extra work needed is of categorical nature: we need to analyze the structure of semisimple factorizable module categories under a neutral Tannakian category. We give a self-contained treatment of this issue in Appendix A, proving a classification result in Theorem \ref{th:fact mod}.


The main results of general nature in this paper can be summarized as follows.  For simplicity we state the results in the case $G$ is split. For notations, see \S\ref{s:sh}.

\begin{theorem}[see Theorem \ref{th:eigen loc} and Prop. \ref{p:mono 0}] Assume $G$ is split. Let $(\bP_{\infty}, \psi,\chi,\bQ_{0})$ be a euphotic automorphic datum. Consider the category $\cP(\psi,\chi)$ of perverse sheaves on $\Bun_{G}(\bQ_{0}, \bP_{\infty}^{++})_{\ov k}$ that are  $(B_{\psi} \ltimes V_{\bP}, \AS_{\psi}\bt\cK_{\chi})$-equivariant. Then
\begin{enumerate}
\item $\cP(\psi,\chi)$ has finitely many simple objects up to isomorphisms, and all of them are clean extensions from an explicit locally closed substack of $\Bun_{G}(\bQ_{0}, \bP_{\infty}^{++})$.
\item There are finitely many semisimple $\dG$-local systems $\{E_{\s}\}_{\s\in \Sigma}$ over $\Gmg$ (for some index set $\Sigma$), and a decomposition of $\cP^{ss}(\psi,\chi)$ (semisimple objects in $\cP(\psi,\chi)$) 
\begin{equation*}
\cP^{ss}(\psi,\chi)=\bigoplus_{\s\in \Sigma} \cP_{\s}
\end{equation*}
into Hecke-stable subcategories, such that each $\cP_{\s}$ is an $E_{2}$-module category under $\Rep(\dG_{\s})$  for $\dG_{\s}=\Aut_{\dG}(E_{\s})$ (whose action on $\cP_{\s}$ is denoted by $\bu$), and the action of the geometric  Hecke operator $\bT_{V}$ (where $V\in \Rep(\dG)$) on $\cA\in \cP_{\s}$ is given by
\begin{equation*}
\bT_{V}(\cA)\cong \bigoplus_{E}E\boxtimes ([E_{\s}(V):E]\bu \cA)\in \Perv(\Gmg\times \Bun_{G}(\bQ_{0}, \bP_{\infty}^{++})_{\ov k}).
\end{equation*}
Here the direct sum is over all irreducible local systems $E$ over $\Gmg$, $E_{\s}(V)\in\Loc(\Gmg)$ is the (semisimple) local system on $\Gmg$ associated to $E_{\s}$ and $V$, $[E_{\s}(V):E]$ is the multiplicity space of $E$ in $E_{\s}(V)$, viewed as an object in $\Rep(\dG_{\s})$. 

\item The geometric monodromy of each $\dG$-local system $E_{\s}$ is tame and unipotent at $0$. Under Lusztig's bijection between  unipotent classes in $\dG$ and two-sided cells of the affine Weyl group $W_{\aff}$, the unipotent monodromy of $E_{\s}$ at $0$ corresponds to the two-sided cell $c_{\bQ}$ containing the longest element of $W_{\bQ_{0}}$.
\end{enumerate}

\end{theorem}

\subsection{Examples} More than half of the paper is devoted to various examples of euphotic automorphic data that are not epipelagic. 

The starting point of our work is a new rigid $G_{2}$-connection on $\PP^{1}\bs\{0,\infty\}$ found by the first-named author \cite{Ja20}. We looked for the automorphic representation corresponding to the $\ell$-adic counterpart of that $G_{2}$-connection, and arrived at the notion of euphotic automorphic data in general. This $G_{2}$ example is presented in detail in \S\ref{s:G2}.

In \S\ref{s: hyperspecial} we give a complete list of euphotic automorphic data when the parahoric subgroup $\bP_{\infty}$ is the hyperspecial parahoric $G(\cO_{\infty})$. The list in this case turns out to be closely related to the classification of double partial flag varieties $G/P_{1}\times G/P_{2}$ that are spherical as a $G$-variety. The latter problem has been solved by Stembridge \cite{St03}, and we use his results. We then study  in \S\ref{s: analyze stabs} the Hessenberg varieties that appear in these examples in more detail, in order to conclude that they are spectrally meager, thereby verifying that the list in \S\ref{s: hyperspecial} indeed gives euphotic automorphic data. 

In \S\ref{s:pot expl}  we give some potential examples of euphotic automorphic data, mostly for exceptional groups. For these groups we have only checked one of the conditions in the definition of euphotic automorphic data.

\subsection{Questions for further study}
\begin{enumerate}
\item The most complicated part in verifying a euphotic automorphic datum is to show that certain Hessenberg varieties are spectrally meager, a condition on the stabilizers of a certain solvable group action.  In \S\ref{s: analyze stabs} we deal with  Hessenberg varieties arising from the adjoint representation of $G$, and we show they are spectrally meager by relating them to Springer fibers. What is still missing is an effective criterion for Hessenberg varieties to be spectrally meager in general.

\item We make predictions on the Langlands parameters of euphotic representations in \S\ref{ss:local L param}, especially about their slopes. This prediction is closely related to another open problem of showing that the $\dG$-local systems we obtain from euphotic automorphic data are indeed cohomologically rigid (see Prop. \ref{p:rig} for evidence).

\item Give a complete list of $(\bP_{\infty}, \psi,\bQ_{0})$ for each almost simple quasi-split $G$, such that there exists $\chi$ making $(\bP_{\infty}, \psi,\chi, \bQ_{0})$ into a euphotic automorphic datum.

\end{enumerate}

\subsection{Notation and preliminaries}

\subsubsection{The curve $X$}
Fix a finite field $k$ of characteristic $p$ and let $F$ be the field of rational functions on $X=\PP^1_k$. Fix an affine coordinate $t$ on $X\bs\{\infty\}$, and we identify $F=k(t)$. The closed points $|X|$ of $X$ are in bijection with the places of $F$. For $x\in |X|$ we denote by $\calO_x$ the completed local ring of $X$ at $x$ and by $F_x$ its field of fractions. The ring of add\`eles of $F$ is the restricted product
\[\AA_F=\sideset{}{'}\prod_{x\in |X|}F_x.\]
 
\subsubsection{The split group $\GG$}\label{no:group split} 
Let $\GG$ be a split, connected semisimple group over $k$ which is almost simple over $\kbar$. Fix a maximal split torus $\TT\subset \GG$, a Borel subgroup $\BB$ containing $\TT$, and extend these choices into a pinning $\dagger=(\TT,\BB, \{x_{i}\}_{i\in I})$ of $\GG$ (where $I$ indexes the set of simple roots).  Let $\Aut^{\dagger}(\GG)$ be the finite group of pinned automorphisms of $\GG$, which is identified with the outer automorphism group of $\GG$.  Let $\frg=\Lie\GG$.

We make the following assumption:
\begin{equation*}
\mbox{There exists a non-degenerate $\Ad(\GG)$-invariant symmetric bilinear form on $\frg$.}
\end{equation*}
This assumption is satisfied when $\ch(k)$ is sufficiently large.  

\sss{The quasi-split group $G$}\label{no:group} 
Let $e\in\{1,2,3\}$. Assume $p\ne e$ and that $k$ contains all $e$-th roots of unity. Fix an injective homomorphism $\th_{\Out}: \mu_{e}(k)\to \Aut^{\dagger}(\GG)$.

Let $G$ be the quasi-split form of $\GG$ over $\Gm=X-\{0,\infty\}$ determined by $\th_{\Out}$. More precisely, let $\wt \Gm=\Spec k[t^{1/e},t^{-1/e}]\to \Gm=\Spec k[t,t^{-1}]$ be the $\mu_{e}$-torsor, and consider the Weil restriction $\Res_{\wt\Gm/\Gm}\GG$. Then $G$ is the fixed point subgroup of $\Res_{\wt\Gm/\Gm}\GG$  under the diagonal action of $\mu_{e}$ on both $\wt\Gm$ and on $\GG$ via $\th_{\Out}$.  

The base changes $G_{F_{0}}$ and $G_{F_{\infty}}$ are quasi-split forms of $\GG$ over the respective local fields determined by the same homomorphism $\th_{\Out}$, viewing $\mu_{e}(k)$ as the the quotient of $\Gal(F_{0}^{s}/F_{0})$ and $\Gal(F_{\infty}^{s}/F_{\infty})$ realized by the tamely ramified extensions $\ls{t^{1/e}}$ and $\ls{t^{-1/e}}$. 

Let $A=\TT^{\mu_{e},\circ}$. From the construction of the group scheme $G$ over $\Gm$, the constant torus $A\times\Gm$  is a maximal split torus of $G$. Similarly, $A_{F_{0}}$ and $A_{F_{\infty}}$ are maximal split tori of $G_{F_{0}}$ and $G_{F_{\infty}}$.

\sss{Coefficient field} We fix a prime $\ell\ne \ch(k)$. The representations of $p$-adic groups and ad\`elic groups will be on $\Qlbar$-vector spaces. Sheaves considered in this paper are \'etale complexes with $\Qlbar$-coefficients over $k$-stacks or $\ov k$-stacks. All sheaf-theoretic functors are understood to be derived.


\subsection*{Acknowledgement}
The authors are grateful to Masoud Kamgarpour and Lingfei Yi for discussions. KJ wishes to thank Michael Dettweiler and Stefan Reiter for teaching him rigid local systems. He especially thanks Jochen Heinloth for his continued support and lots of discussions about the geometric Langlands program.

\section{Euphotic representations}
In this section we introduce a class of representations of the $p$-adic group $G(F_{\infty})$ that generalize the {\em epipelagic representations} introduced by Reeder and Yu in \cite{RY14}.

This section concerns only the quasi-split group $G_{F_{\infty}}$ over $F_{\infty}$. We denote $G_{F_{\infty}}$ simply by $G$ in this section.  Using the affine coordinate $t$ on $X\setminus\{\infty\}$, we identify $F_\infty$ with $K=\ls{t^{-1}}$ and write $K_e=\ls{t^{-1/e}}$.

\subsection{Parahorics and gradings} 

\subsubsection{Affine roots and root subgroups}
The Lie algebra $\Lie G$ is the $\mu_{e}$-invariants on $\frg\ot K_{e}$ (where $\mu_{e}$ acts both on $K_{e}$ by Galois action and on $\frg$ via $\th_{\Out}$). The torus $A=\TT^{\mu_{e},\circ}$  (see \S\ref{no:group}) acts on  $\Lie G$ by the adjoint action and additionally this algebra carries a $\Gm$-action given by scaling the uniformizer $t^{-1/e}$ of $K_e$. The set of affine roots $\Psi_{\aff}$ with respect to $A$ can also be identified with the weights of $A\times \Gm$ for this action on $\Lie G$. The set of real roots $\Psi_\textup{re}\subset \Psi_{\aff}$ consists of those affine roots which are non-trivial on the torus $A$. \par 

Denote by $LG$ the loop group of $G$. For any real root $\alpha \in \Psi_{\re}$ there is a subgroup $U_\alpha \subset LG$  which is isomorphic to $\Ga$ over $k$ and whose Lie algebra is the $\alpha$-eigenspace under the action of $A\times \Gm$ on $\Lie G$, cf. \cite[Sections 2.1 \& 2.2]{Yun16}.

The Borel subgroup $\BB$ fixed in \S\ref{no:group} gives a set of simple affine roots $\D_{\aff}\subset \Psi_{\aff}$ and positive affine roots $\Psi^{+}_{\aff}\subset \Psi_{\aff}$.

\sss{Parahoric subgroups}
The maximal split torus $A\ot K$ of $G$ fixed in \S\ref{no:group} defines an apartment $\frA$ of the building of $G(K)$. Then $\Psi_{\aff}$ can be identified with a set of affine functions on $\frA$, whose vanishing affine hyperplanes give a stratification of $\frA$ into facets. There is a unique set of positive integers $\{n_{\a}\}_{\a\in \D_{\aff}}$ such that $\sum_{\a\in\D_{\aff}}n_{\a}\a=1$ as functions on $\frA$. 

The fundamental alcove $C\subset \frA$ is cut out by the inequalities $\a>0$ for all $\a\in \D_{\aff}$. Let $\frF\subset \ov C$ be a facet in the closure of $C$. Let $J=\{\a\in \D_{\aff}|\a|_{\frF}=0\}$. Let 
\begin{equation*}
m=m_{\frF}=\sum_{\a\in \D_{\aff}\bs J} n_{\a}\in \NN.
\end{equation*}
In the rest of the paper we also assume
\begin{equation*}
\mbox{The characteristic $p=\ch(k)$ does not divide $m$.}
\end{equation*}

Let $\bP\subset G(K)$ be the (standard) parahoric subgroup corresponding to $\frF$. Let $\bP\supset \bP^{+}\supset \bP^{++}$ be the first three steps in the Moy-Prasad filtration of $\bP$ with respect to $x_{\bP}$, the barycenter of $\frF$. In other words
\begin{equation*}
\bP^{+}=G(K)_{x_{\bP}, 1/m}, \quad \bP^{++}=G(K)_{x_{\bP}, 2/m}.
\end{equation*}
Let $L_{\bP}=\bP/\bP^{+}$ be the Levi factor of $\bP$ (a connected reductive group over $k$). There is a canonical section $L_{\bP}\incl \bP$ whose image contains $A$; we identify  $L_{\bP}$ with the image of this section. The quotient 
$V_{\bP}=\bP^{+}/\bP^{++}$ is a representation of $L_{\bP}$ over $k$. 

Let $\Psi(\bP)\subset \Psi_{\aff}$ be the affine roots that appear as the weights of $A\times \Gm$ on $\Lie\bP$; similarly define $\Psi(\bP^{+}), \Psi(\bP^{++}), \Psi(L_{\bP})$ and $\Psi(V_{\bP})$. Then  
\begin{eqnarray*}
\Psi(L_{\bP})=\{\a\in\Psi_{\aff}|\a(x_{\bP})=0\}; \quad \Psi(V_{\bP})=\{\a\in\Psi_{\aff}|\a(x_{\bP})=\frac{1}{m}\}.
\end{eqnarray*}

\sss{Cyclic grading on the Lie algebra} Let $\frg=\Lie\GG$. Then the barycenter $x_{\bP}$ gives a $\Zm$-grading on $\frg$ 
\begin{equation*}
\frg=\bigoplus_{i\in\Zm}\frg(i)
\end{equation*}
compatible with the $\ZZ/e\ZZ$-grading on $\frg$ obtained from $\th_{\Out}$ (under the reduction map $\Zm\to \ZZ/e\ZZ$) and such that $\frg(0)$ can be canonically identified with $\Lie(L_{\bP})$, and $\frg(i)$ can be canonically identified with $G(K)_{x_{\bP}, \frac{i}{m}}/G(K)_{x_{\bP}, \frac{i+1}{m}}$ for $i\ne0$. In particular we have $V_{\bP}\cong \frg(1)$ as $L_{\bP}$-modules. For more details, see \cite[Theorem 4.1]{RY14}.

\subsection{Euphotic representations}
Let $\bP$ be a standard parahoric subgroup of $G(K)$. Let $\psi\in V^{*}_{\bP}$ be a vector whose $L_{\bP}$-orbit is closed.  Using an $\Ad(\GG)$-invariant form on $\frg$, which exists by our assumption in \S\ref{no:group split}, the $L_{\bP}$-module $V^{*}_{\bP}$ may be identified with $\frg(-1)$. Then the $L_{\bP}$-orbit of $\psi$  is closed if and only if $\psi\in\frg(-1)$ is semisimple as an element of $\frg$.

Fix an additive character $\psi_{k}:k\to\Qlbar^{\times}$. For an admissible representation $(\pi, \calV)$ of $G(K)$ over $\Qlbar$, let
\begin{equation*}
\cV^{(\bP^{+}, \psi)}=\{v\in \cV|\pi(g)v=\psi_{k}(\psi(g))v, \forall g\in \bP^{+}\}.
\end{equation*}

\begin{defn}\label{d:qep} Let $\psi\in V^{*}_{\bP}$ be a vector whose $L_{\bP}$-orbit is closed. Let $(\pi, \calV)$ be an irreducible admissible representation of $G(K)$. We say that it is \textit{euphotic} with respect to $(\bP,\psi)$ if $\calV^{(\bP^+,\psi)}\neq 0$.
\end{defn}

\sss{Action of $L_{\psi}$} Let $L_{\psi}$ be the stabilizer of $\psi$ under $L_{\bP}$. Then $L_{\psi}$ is a (not necessarily connected) reductive group over $k$. In fact, since $\psi$ can be identified with a semisimple element of $\frg$ lying in $\frg(-1)$, its centralizer $\GG_{\psi}$ in $\GG$ is a reductive group whose Lie algebra $\frg_{\psi}$ is stable under the $\Zm$-grading. Viewing the $\Zm$-grading on $\frg$ as an action of $\mu_{m}$ on $\GG$, $L_{\bP}$ is the neutral component of $\GG^{\mu_{m}}$. Then $\GG_{\psi}$ is stable under the $\mu_{m}$-action, and $L_{\psi}\subset (\GG_{\psi})^{\mu_{m}}$ is the union of components that lie in $L_{\bP}$. 

Let $(\pi, \cV)$ be a euphotic representation of $G(K)$. There is an action of $L_{\psi}(k)$ on $\cV^{(\bP^{+}, \psi)}$. We will be interested in those $(\pi,\cV)$ such that $\cV^{(\bP^{+}, \psi)}$ contains a principal series representation of $L_{\psi}(k)$. More precisely, let $B_{\psi}\subset L_{\psi}^{\circ}$ be a Borel subgroup of  the neutral component $L_{\psi}^{\circ}$ of $L_{\psi}$. Let $T_{\psi}$ be the quotient torus of $B_{\psi}$.

\begin{defn}\label{d:qep chi} Let $\psi\in\frg(-1)$ be semisimple and let $\chi: T_{\psi}(k)\to \Qlbar^{\times}$ be a character. Let $(\pi, \calV)$ be an irreducible admissible representation of $G(K)$. We say that it is \textit{euphotic} with respect to $(\bP,\psi,\chi)$ if the action of $B_{\psi}(k)$ on $\calV^{(\bP^+,\psi)}$ contains a nonzero eigenvector on which $B_{\psi}(k)$ acts via the character $B_{\psi}(k)\to T_{\psi}(k)\xr{\chi}\Qlbar^{\times}$.
\end{defn}

By Frobenius reciprocity, an irreducible admissible $(\pi,\cV)$ is euphotic with respect to $(\bP,\psi,\chi)$ if and only if it is a quotient of the compact induction
\begin{equation*}
c-\ind^{G(K)}_{\bP^{+}\cdot B_{\psi}(k)}((\psi_{k}\circ\psi)\boxtimes\chi).
\end{equation*}

\subsection{Relation with epipelagic representations}\label{ss:H} For simplicity in this subsection we assume $G$ is split over $F$ (i.e., $e=1$) and $L_{\psi}^{\circ}$ is split over $k$. We lift $T_{\psi}$ to a maximal split torus $T_{\psi}\subset L_{\psi}$ over $k$. Up to changing $\psi$ by an element in the same $L_{\bP}$-orbit, we may assume $T_{\psi}\subset A$ ($A$ is a maximal split torus of $L_{\bP}$). Then $\HH=C_{\GG}(T_{\psi})$ is a Levi subgroup of $\GG$ containing $\TT$. From the construction, $\HH$ is stable under the $\mu_{m}$-action on $\GG$ which gives the $\Zm$-grading.

We claim that the induced $\Zm$-grading on $\frh=\Lie \HH$ is {\em stable} in the sense of \cite[\S5.3]{RLYG}. Indeed, it suffices to show that the stabilizer $(\HH_{\psi})^{\mu_{m}}$ is finite modulo $T_{\psi}$, but this is true because $(\HH_{\psi})^{\mu_{m}}$ is a reductive group containing $T_{\psi}$ as a maximal torus which at the same time is central, hence $(\HH_{\psi})^{\mu_{m}}/T_{\psi}$ is finite. The theory developed in \cite[Corollary 15]{RLYG} then attaches a {\em regular elliptic conjugacy class} $[w]$ in the extended Weyl group $W_{\mathrm{ext}}(\HH,\TT)$ ( component group of the normalizer of $\TT$ in $\Aut(\HH)$). One checks that $[w]$ indeed is a well-defined conjugacy class in $W(\GG,\HH,\TT)=(N_{\GG}(\HH)\cap N_{\GG}(\TT))/\TT$.

Let $H=\HH\ot_{k}K$, a Levi subgroup of $G$. The $\Zm$-grading on $\frh$ gives a standard parahoric subgroup $\bP_{H}\subset H$ and $\psi\in\frh(-1)$ can be viewed as a linear function on $\bP^{+}_{H}/\bP_{H}^{++}$. We expect that a euphotic representation of $G(K)$ with respect to $(\bP, \psi,\chi)$ should be a composition factor of a parabolic induction from an epipelagic representation of $H(K)$ with respect to $(\bP_{H}, \psi)$ in the sense of \cite{RY14}, whose central character restricts to $\chi$ on $T_{\psi}\subset ZH$.


\subsection{Predictions on the Langlands parameter}\label{ss:local L param} 
Again for simplicity we assume $e=1$, i.e., $G$ is split, and that $L_{\psi}$ is split over $k$.

\sss{} Let $\dG$ be the Langlands dual of $\GG$ (over $\Qlbar$) equipped with a maximal torus $\dT$ and an isomorphism $\xch(\dT)\cong \xcoch(\TT)$. Then the roots $\Phi(\dG,\dT)$ are identified with the coroots $\Phi^{\vee}(\GG,\TT)$. Recall the Levi subgroup $\HH$ introduced in \S\ref{ss:H}. Let $\dH\subset \dG$ be the Levi subgroup containing $\dT$ such that $\Phi(\dH,\dT)\subset \Phi(\dG,\dT)$ is identified with the set $\Phi^{\vee}(\HH,\TT)\subset \Phi^{\vee}(\GG,\TT)$.




\sss{}\label{sss:local L param} Let $W_K\supset I_K\supset I^+_K$ be the Weil group, inertia group and wild inertia of the local field $K$. Let $\pi$ be a euphotic representation of $G(K)$ with respect to $(\bP, \psi,\chi)$. Let $\r_{\pi}: W_{K}\to \dG$ be the Langlands parameter of $\pi$. 

For $p=\ch(k)$ large, we make the following predictions on $\r_{\pi}$:

\begin{enumerate}
\item\label{image NS} Consider the torus $\dS = [\dH, \dH] \cap \dT$. Then, up to $\dG$-conjugacy, one should be able to arrange that $\r_{\pi}(I^{+}_{K})\subset \dS$. We also expect that $\r_{\pi}(I^{+}_{K})$ is regular in $\dS$, i.e., its centralizer in $\dG$ is $C_{\dG}(\dS)$. This implies that $\r_{\pi}(I_{K})$ lies in the normalizer $N_{\dG}(\dS)$ of $\dS$ in $\dG$.  

\item\label{tame w} By \eqref{image NS}, $\r_{\pi}$ induces a homomorphism
$$
\r_{\pi}^{\tame}: I^{\tame}_{K}=I_{K}/I^{+}_{K} \to N_{\dG}(\dS)/\dS.
$$ 
The group $N_{\dG}(\dS)/\dS$ is a possibly disconnected reductive group with $\dT/\dS=\dH^{ab}$ as a maximal torus.   Let $\r_{\pi}^{\tame,ss}:I^{\tame}_{K}\to N_{\dS}(\dS)/\dS$ be the semisimplification of $\r_{\pi}^{\tame}$. 

There is an inclusion $N_{\dG}(\dH,\dT):=N_{\dG}(\dH)\cap N_{\dG}(\dT)\subset N_{\dG}(\dS)$. Then up to conjugacy $\r_{\pi}^{\tame,ss}$ should have image in $N_{\dG}(\dH,\dT)/\dS$. The composition
\begin{equation*}
I^{\tame}_{K}\xr{\r_{\pi}^{\tame,ss}} N_{\dG}(\dH,\dT)/\dS\to N_{\dG}(\dH,\dT)/\dT=:W(\dG,\dH,\dT)
\end{equation*}
should map a topological generator of $I^{\tame}_{K}$ to a regular elliptic element  $w\in W(\dG,\dH,\dT)=W(\GG,\HH,\TT)$ which is in the regular elliptic conjugacy class attached to $\bP$ in  \S\ref{ss:H}.

\item\label{rho ss} Let $N_{\dG}(\dH,\dT)'\subset N_{\dG}(\dH,\dT)$ be the preimage of the cyclic group $\j{w}\subset W(\dG,\dH,\dT)$.  Then by \eqref{tame w}, up to conjugacy $\r_{\pi}^{\tame,ss}$ should have image in $N_{\dG}(\dH,\dT)'/\dS$, which is an extension of $\j{w}$ by $\dT/\dS$. Since $w$ acts trivially on $T_{\psi}$, the projection $\dT\to\dT/\dS\to \dT_{\psi}$ dual to the inclusion $T_{\psi}\subset \TT$ extends to a homomorphism $N_{\dG}(\dH,\dT)'/\dS\to \dT_{\psi}$.  Then the composition
\begin{equation*}
I^{\tame}_{K}\xr{\r_{\pi}^{\tame,ss}} N_{\dG}(\dH,\dT)'/\dS\to  \dT_{\psi}
\end{equation*}
should correspond to the character $\chi$ of $T_\psi(k)$ under local class field theory.

\item The slopes of the adjoint representation $\Ad(\r_{\pi}): W_{K}\to \Aut(\dg)$ are either $0$ or $1/m$, and 
\begin{equation}\label{Sw}
\Swan(\Ad(\r_{\pi}))=\dim L-\dim L_{\psi}. 
\end{equation}
\end{enumerate}
The prediction on Swan conductors \eqref{Sw} is based on the following heuristics. The Swan conductor should be $1/m$ for each root $\a^{\vee}\in\Phi(\dG,\dT)=\Phi^{\vee}(\GG,\TT)$ which is nontrivial on the image $\r(I^{+}_{K})$, and should be zero on other root spaces and on $\dt$. Those $\a^{\vee}$ such that $\a^{\vee}|_{\r(I^{+}_{K})}=1$ correspond exactly to the coroots of the centralizer $\GG_{\psi}$ with respect to $\TT$. Let $R'=\{\a\in\Phi(\GG,\TT)|\a(\psi)\ne0\}$, then  $\Swan(\Ad(\r_{\pi}))=\#R'/m$. On the other hand, the $\mu_{m}$-action on $\frg$ preserves $R'$ and freely permutes $R'$. We get that $\frg(0)=\frg^{\mu_{m}}$ is the direct sum of $\frg_{\psi}\cap \frg(0)=\Lie L_{\psi}$ and $1$-dimension from the sum of root spaces for each $\mu_{m}$-orbit of $R'$. Therefore the number of $\mu_{m}$-orbits on $R'$, which is $\#R'/m$, is the same as $\dim L-\dim L_{\psi}$, hence the prediction \eqref{Sw}.

Let $I^{\tame}_{K}(m)\subset I^{\tame}_{K}$ be the unique subgroup of index $m$. Assume the character $\chi$ is sufficiently generic, then $\r^{\tame}_{\pi}$ will be semisimple and by \eqref{rho ss} above, $\r^{\tame}_{\pi}(I^{\tame}_{K}(m))$ should be conjugated into $\dT/\dS$. The genericity of $\chi$ should imply that the centralizer of $\r^{\tame}_{\pi}(I^{\tame}_{K}(m))$ in $C_{\dG}(\dS)/\dS$ is $\dT/\dS$. Therefore, for $\chi$ sufficiently generic,  we should have $\dG^{\r_{\pi}(I_{K})}=C_{\dG}(\dS)^{\r^{\tame}_{\pi}(I^{\tame}_{K})}=\dT^{w}$. Therefore we expect to have 
\begin{equation}\label{ginv}
\dim \dg^{\r_{\pi}(I_{K})}=\dim \dt^{w}= \dim T_{\psi} =\rk L_{\psi}. 
\end{equation}

\section{Euphotic automorphic data}
In this section we define euphotic automorphic data, and give a criterion for them to be rigid.

\subsection{Pre-euphotic automorphic data} 
\sss{Opposite parahorics} Let $\bP_{\infty}$ be a standard parahoric subgroup of $G(F_{\infty})$ corresponding to a facet $\frF\subset \ov C$ in the apartment $\frA_{\infty}$ of the building of $G(F_{\infty})$ corresponding to the split torus $A_{F_{\infty}}$. Let $\frA_{0}$ be the apartment  in the building of $G(F_{0})$ corresponding to the split torus $A_{F_{0}}$. Then there is a unique isomorphism $\frA_{\infty}\cong\frA_{0}$ characterized by the following two conditions
\begin{enumerate}
\item The natural action of $\xcoch(A)_{\RR}$ on $\frA_{\infty}$ corresponds to the {\em opposite} of the natural action of $\xcoch(A)_{\RR}$ on $\frA_{0}$.
\item The special vertex in $\frA_{\infty}$ corresponding to the parahoric $G(\cO_{\infty})$ maps to the special vertex in $\frA_{0}$ corresponding to the parahoric $G(\cO_{0})$.
\end{enumerate}
Let us denote both $\frA_{\infty}$ and $\frA_{0}$ by $\frA$ under this identification. Let $\bP_{0}$ be the parahoric subgroup of $G(F_{0})$ corresponding to the same facet $\frF$ that we used to define $\bP_{\infty}$.  We say $\bP_{0}$ thus constructed is {\em opposite} to $\bP_{\infty}$. Then the Levi factors $L_{\bP_{\infty}}$ and $L_{\bP_{0}}$ can be canonically identified, which we denote by $L_{\bP}$, or simply $L$. 



\begin{defn}\label{d:pre qep data} A {\em pre-euphotic automorphic datum} is a quadruple $(\bP_{\infty}, \psi,\chi, \bQ_{0})$ where
\begin{itemize}
\item $\bP_{\infty}$ is a standard parahoric subgroup of $G(F_{\infty})$;
\item $\psi\in V_{\bP_{\infty}}^{*}$ with closed orbit under $L$;
\item $\chi: T_{\psi}(k)\to \Qlbar^{\times}$ is a character.
\item Let $\bP_{0}\subset G(F_{0})$ be the parahoric subgroup opposite to $\bP_{\infty}$. Then $\bQ_{0}$ is a parahoric subgroup of $G(F_{0})$ which is contained in $\bP_{0}$ and contains the torus $A$. 
\end{itemize}
\end{defn}

The parahoric  $\bQ_{0}\subset \bP_{0}$ corresponds to a facet $\frF'$ in $\frA_{0}$ whose closure contains $\frF$. Let $x_{\bQ}$ be the barycenter of $\frF'$. The parahoric $\bQ_{0}$ also determines a parabolic subgroup $Q$ of $L=L_{\bP_{0}}$ (containing $A$) such that $\bQ_{0}$ is the preimage of $Q$ under the projection $\bP_{0}\to L$. 

\sss{Weyl groups}\label{sss:W} Let $\WW$ be the Weyl group of $\GG$ with respect to $\TT$. The Weyl group $W$ of $G$  with respect to $A$ can be identified with the fixed points $\WW^{\mu_{e}}$. The Iwahori-Weyl groups of $G(F_{\infty})$ and $G(F_{0})$ with respect to $A$ can be identified under the identification $\frA_{\infty}=\frA_{0}$; we denote it by $\tilW=\xcoch(A)\rtimes W$. For $w\in W$, choose a lifting of it in $N_{\GG}(\TT)(k)^{\mu_{e}}$; for an arbitrary element $w=(\l,w_{1})\in \xcoch(A)\rtimes W=\tilW$, we have its lifting $\dot w=t^{\l}\dot{w}_{1}$.

Let $W_{\aff}\subset \tilW$ be the affine Weyl group generated by affine simple reflections. Let $\Om=\Stab_{\tilW}(C)$, the stabilizer of the fundamental alcove under $\tilW$. Then the projection induces an isomorphism $\Om\isom \tilW/W_{\aff}$, and $\Om$ is a finite abelian group.  

Let $W_{\bP}$ (resp. $W_{\bQ}$) denote the Weyl group of $L$ (resp. $L_{\bQ}$, the Levi of $\bQ_{0}$ or $Q$), both as subgroups of $W_{\aff}$. We have $W_{\bQ}\subset W_{\bP}$.


%

\subsection{The space of automorphic forms}

\sss{A space of automorphic forms}\label{sss:auto conditions} We will impose more conditions on the data $(\bP_{\infty}, \psi,\chi, \bQ_{0})$. To motivate these conditions, we consider automorphic representations $\pi=\ot'_{x\in |X|}\pi_{x}$ of $G(\AA_{F})$ for which
\begin{enumerate}
\item $\pi_{x}$ is unramified for $x\ne 0,\infty$;
\item $\pi_{0}^{\bQ_{0}}\ne0$;
\item $\pi_{\infty}$ is euphotic with respect to $(\bP_{\infty}, \psi,\chi)$ in the sense of Definition \ref{d:qep chi}.
\end{enumerate}

Any such automorphic representation $\pi$ contains a nonzero vector in the following space
\[\calF=\textup{Fun}\left(G(F)\backslash G(\AA_F)/( \bQ_0 \times \prod_{x\neq 0,\infty} G(\calO_x))\right)^{(B_\psi \bP_\infty^+,\mu)} \]
of $\Qlbar$-valued functions on which $B_\psi\bP_\infty^+\subset \bP_{\infty}$ acts via the character $\mu$ defined by $\mu|_{B_\psi}=\chi$ and $\mu|_{\bP_\infty^+}=\psi_k\circ\psi$. 
Let $\Gamma_0=G(k[t,t^{-1}]) \cap \bQ_{0}$. Through the equality 
\begin{equation}\label{dcos}
G(F)\backslash G(\AA_F)/ (\bQ_0 \times \prod_{x\neq 0,\infty} G(\calO_x)\times \bP_\infty^{++})=  \bP_\infty^{++}\backslash G(F_\infty)/\Gamma_0 
\end{equation}
of double cosets, cf. \cite[2.12]{Yun16} we identify
\[\calF=\textup{Fun}(G(F_\infty)/\Gamma_0)^{(B_\psi\bP_\infty^+,\mu)} .\]
We have the Birkhoff decomposition \cite[3.2]{Yun16}
\begin{equation}\label{Birk}
G(F_\infty)=\coprod_{[w]\in W_{\bP} \backslash \tilW/W_{\bQ} }\bP_\infty \dot{w}\Gamma_0.
\end{equation}
From this we get a decomposition
\begin{equation*}
\cF=\bigoplus_{[w]\in W_{\bP}\backslash \tilW / W_{\bQ}}\cF_{w}
\end{equation*}
where $\cF_{[w]}$ is the subspace of functions supported on $\bP_{\infty}\dot{w}\Gamma_{0}$.

\sss{The space $\cF_{[w]}$}\label{sss:Yw}
Let $w\in \tilW$. We have
\begin{equation}\label{QwL}
\bP_{\infty}^{+}\bs \bP_{\infty}\dot{w}\Gamma_{0}/\Gamma_{0}= L(k)/ Q_{w}(k)=( L/Q_{w})(k)
\end{equation}
where $Q_{w}=\Ad(w)\Gamma_{0}\cap L=\Ad(w)\bQ_{0}\cap L$ is the parabolic subgroup of $L$ containing $A$ with roots
\begin{equation*}
\Psi(Q_{w})=\{\a\in \Psi_{\re}|\a(x_{\bP})=0, \a(w x_{\bQ})\le0\}.
\end{equation*}

Let $f\in \calF_{w}$. Assume $f$ is not identically zero on the double coset $\bP_\infty^+\ell\dot{w}\Gamma_0$ for some $\ell\in L$. For any $\alpha\in \Psi(V_{\bP})$ (i.e., $\a(x_{\bP})=1/m$)  and $u\in U_\alpha$ we have
\[f(\ell u\dot w)={}^{\ell^{-1}}\psi(u)f(\ell \dot w).\] 
If we furthermore assume that $\alpha(w x_{\bQ})\le 0$, then $U_{w^{-1}\alpha} \subset \Gamma_0$ and  we find that
\[f(\ell\dot{w})=f(\ell\dot{w}\Ad(\dot{w}^{-1})(u))=f(\ell u \dot{w})={}^{\ell^{-1}}\psi(u)f(\ell\dot{w}).\]
This implies that ${}^{\ell^{-1}}\psi$ vanishes on the space
\[V_w:={\bigoplus_{\substack{ \alpha(x_{\bP})=1/m\\ \alpha(wx_{\bQ} ) \le 0 }}} V_{\bP}(\a).\]
In other words
\begin{equation*}
{}^{\ell^{-1}}\psi\in V_{w}^{\bot}=\bigoplus_{\substack{\alpha(x_{\bP})=-1/m\\ \alpha(w x_{\bQ} )<0} } V^{*}_{\bP}(\a)\subset V_{\bP}^{*}
\end{equation*}
We may alternatively write $V_{w}^{\bot}$ into a sum of $A$-weight spaces
\begin{equation*}
V_w^\bot= \bigoplus_{\langle \abar,w x_{\bQ}-x_{\bP}\rangle < \frac{1}{m}}V_{\bP}^*(\abar). 
\end{equation*}
Note that $V_{w}^{\bot}$ is stable under $Q_{w}$. \par

\begin{defn} Let $Y_{w}$ be the closed subscheme of the partial flag variety $L/Q_{w}$ defined by
\begin{equation*}
Y_{w}=\{ \ell\cdot Q_{w} \in L/Q_{w} \mid \,^{\ell^{-1}}\psi\in V_{w}^\bot\}.
\end{equation*}
\end{defn}

\begin{remark}
\begin{enumerate}
\item The variety $Y_{w}$ is a Hessenberg variety in the sense of \cite{GKM}, attached to the $L$-module $V_{\bP}^{*}$ and the $Q_{w}$-stable subspace $V_{w}^{\bot}$. It carries an action of the stabilizer $L_{\psi}$ by  right translation.

\item If $w=1$, or more generally if $\frF_{\bP}$ is in the closure of $w\frF_{\bQ}$, then $Y_{w}=L/Q_{w}$.

\item If we change $w$ to $ww_{1}$ for some $w_{1}\in W_{\bQ}$, then both $Q_{w}$ and $V^{\bot}_{w}$ are unchanged, hence $Y_{w}=Y_{ww_{1}}$. 
\item If we change $w$ to $w_{2}w$ for some $w_{2}\in W_{\bP}$, then we may lift $w_{2}$ to $\ddot w_{2}\in L$, and right multiplication by $\ddot w_{2}$ induces an isomorphism $ L/Q_{w}\isom L/Q_{w_{2}w}$ which restricts to an isomorphism $Y_{w}\isom Y_{w_{2}w}$.
\end{enumerate}
\end{remark}

The above discussion implies that any function $f\in \cF_{[w]}$ must be supported on $Y_{w}(k)$ (as a function on $(L/Q_{w})(k)$ under \eqref{QwL}). In view of the eigen property under $B_{\psi}$, we get the following description of $\cF_{w}$. 

\begin{lemma}\label{l:Fw} For any $w\in \tilW$, let $[w]$ be its $(W_{\bP},W_{\bQ})$ coset in $\tilW$. Then there is a canonical isomorphism
\begin{equation*}
\cF_{[w]}\cong \textup{Fun}(Y_{w}(k))^{(B_{\psi}(k), \chi)}
\end{equation*}
where the right side is the space of eigenfunctions on $Y_{w}(k)$ under the left translation of $B_{\psi}(k)$ with eigen-character $\chi$. The isomorphism is given by $\cF_{w}\ni f\mapsto f_{w}$ where $f_{w}(\ell Q_{w})=f(\ell \dot w\Gamma_{0})$ for all $\ell Q_{w}\in Y_{w}(k)$.
\end{lemma}

\subsection{Spectrally meager varieties} In examining when the space $\cF_{w}$ is zero we arrive at the following notion.

\begin{defn}\label{d:sp meager}
\begin{enumerate}
\item Let $H$ be a connected reductive group over $k$ with Borel subgroup $B_{H}$. Let $Y$ be a scheme of finite type over $k$ with an $H$-action. We say that $Y$ is {\it spectrally meager} if for any geometric point $y\in Y(\ov k)$ the stabilizer $\Stab_{B_{H}}(y)$ contains a nontrivial torus. 
\item If $Y$ is a spectrally meager $H$-scheme, let $\cS(Y)$ be the collection of (nontrivial) subtori of $T_{H, \ov k}$ (the universal Cartan of $H$) given by the images of $\Stab_{B_{H}}(y)^{\circ}\to T_{H}$ for all $y\in Y(\ov k)$. 
\end{enumerate}
\end{defn}

\begin{remark}
\begin{enumerate}
\item The definition of spectrally meager $H$-scheme does not depend on the choice of the Borel subgroup $B_{H}$. It is therefore intrinsic to the $H$-scheme $Y$. 
\item If $Y$ is a spectrally meager $H$-scheme, the collection $\cS(Y)$ of subtori  is finite. Indeed, we may partition $Y$ into finitely many locally closed connected $B_{H}$-stable subschemes $\{Y_{\a}\}$ such that the torus part of the stabilizer of $B_{H}$ on each point of $Y_{\a}$ has the same dimension, then each $Y_{\a}$ contributes a single torus in $\cS(Y)$.
\item The terminology ``spectrally meager'' may be justified as follows. 

Let $\chi: T_{H}(k)\to \Qlbar^{\times}$ be a character with the property that $\chi|_{S(k)}\ne 1$ for any $S\in\cS(Y)$ defined over $k$, then the $H(k)$-module $\Fun(Y(k))$ does not contain any simple constituent of the principal series representation $\Ind_{B_{H}(k)}^{H(k)}(\chi)$. The same property holds after any finite base change $k'/k$.

On the other hand, making the obvious definition over $\CC$, if $Y$ is an affine $H$-variety  over $\CC$ which is spectrally meager, then $\cO(Y)$ as an algebraic $H$-module contains the  irreducible $H$-module $V_{\l}$ with highest weight $\l$ only if $\l\in \xcoch(S)^{\bot}$ for some $S\in \cS(Y)$, i.e., $\l$ lies in the union of finitely many proper sublattices in $\xch(T)$.
\end{enumerate}
\end{remark}

\begin{cor}[of Lemma \ref{l:Fw}]\label{c:Fw zero} Let $w\in \tilW$. If $Y_{w}$ is spectrally meager as an $L_{\psi}$-scheme,  and $\chi$ is nontrivial on $S(k)$ for any torus $S\in \cS(Y_{w})$ that is defined over $k$, then $\cF_{[w]}=0$.
\end{cor}
\begin{proof}
For any $y\in Y_{w}(k)$, the stabilizer $\Stab_{B_{\psi}}(y)$ maps to a nontrivial torus $S\in \cS(Y_{w})$. Since $\chi|_{S(k)}\ne 1$, all $(B_{\psi}(k), \chi)$-eigenfunctions on $Y_{w}(k)$ must vanish at $y$. We conclude that $\cF_{w}=0$ by Lemma \ref{l:Fw}. 
\end{proof}


\subsection{Euphotic automorphic data}

\begin{defn}\label{d:qep data} A pre-euphotic automorphic datum $(\bP_{\infty}, \psi,\chi,\bQ_{0})$ is called a {\em euphotic automorphic datum} if it satisfies the following conditions:
\begin{enumerate}
\item\label{open B orbit}  For $w\in \Om$, any Borel subgroup $B_{\psi}$ of $L_{\psi}$ acts on $Y_{w}$ with an open orbit with finite stabilizers.

\item For any $w\in \tilW-W_{\bP}\Om W_{\bQ}$,  the $L_{\psi}$-scheme $Y_{w}$ is spectrally meager. 

\item Let $\cK_{\chi}$ be the Kummer local system on $T_{\psi}$ attached to $\chi$ (see \cite[Appendix A.3.5]{Yun14}). Then for any $S\in \cup_{w\notin W_{\bP}\Om W_{\bQ}}\cS(Y_{w})$ (this is a subtorus of $T_{\psi,\kbar}$ over $\ov k$), the restriction $\cK_{\chi}|_{S}$ is a nontrivial local system.
\end{enumerate}
We call the euphotic automorphic datum $(\bP_{\infty},\psi,\chi,\bQ_{0})$ {\em strict} if moreover the following holds:
\begin{enumerate}
\item For any $w\in \Omega$ and any $x\in (L/Q_w)(\ov k)$ outside the open $B_{\psi}$-orbit let $S_{x}$ be the image of $\Stab_{B_{\psi}}(x)^{\circ}\to T_{\psi, \ov k}$ and assume that $\cK_{\chi}|_{S_{x}}$ is a nontrivial local system (in particular, $S_{x}$ is a nontrivial torus).
\item The stabilizer on the open $B_{\psi}$-orbit is $ZG$.
\end{enumerate}
\end{defn}

\begin{remark}\label{r:sph}
\begin{enumerate}
\item The conditions for a pre-euphotic automorphic datum to be a euphotic automorphic datum can be checked after base changing the situation to $\ov k$.
\item The open $B_{\psi}$-orbit condition in Definition \ref{d:qep data} is saying that $Y_{w}$ is a spherical $L_{\psi}$-variety. This implies that $B_{\psi}$ has finitely many orbits on $Y_{w}$ by \cite{Br86}.
\item As $w$ varies in $\tilW$, there are only finitely many different $L_{\psi}$-equivariant isomorphism types of the schemes $Y_{w}$. Indeed, there are only finitely many possibilities for $Q_{w}$ and $V_{w}$. Therefore the union $\cup_{w\notin W_{\bP}\Om W_{\bQ}}\cS(Y_{w})$ is a finite set.
\end{enumerate}
\end{remark}

\begin{prop}\label{c:fdim} Let $(\bP_{\infty},\psi,\chi, \bQ_{0})$ be a euphotic automorphic datum. Then 
\begin{enumerate}
\item The space $\calF$ is finite-dimensional and consists of cusp forms. In particular, any automorphic representation $\pi$ satisfying the conditions in \S\ref{sss:auto conditions} is cuspidal.
\item For any finite field extension $k'/k$, consider the similarly defined space $\cF^{k'}$ using the base change automorphic data $(\bP_{\infty},\psi,\chi, \bQ_{0})$. Then $\dim_{\Qlbar}\cF^{k'}$ is bounded independent of $k'$.
\end{enumerate}
\end{prop}
\begin{proof}
(1) For any $w\in  \tilW-W_{\bP}\Om W_{\bQ}$,  the assumptions in Corollary \ref{c:Fw zero} are satisfied (for any $S\in \cS(Y_{w})$ defined over $k$, $\chi|_{S(k)}$ is nontrivial if and only if $\cK_{\chi}|_{S_{\ov k}}$ is nontrivial). Therefore $\cF_{[w]}=0$.

If $w\in \Om$,  the space $\cF_{[w]}=\textup{Fun}(Y_{w}(k))^{(B_{\psi}(k),\chi)}$ has finite dimension because there are finitely many $B_{\psi}$-orbits on $Y_{w}$ over $\ov k$ (see Remark \ref{r:sph}(2)), hence finitely many rational orbits as well.  we conclude that $\cF$ is finite-dimensional and stable under the spherical Hecke operators at all places $x\notin\{0,\infty\}$. By \cite[Lemme 8.24]{La18}, $\cF$ consists of cusp forms.


(2) The universal bound for $\dim\cF^{k'}$ comes from bounding the number of $B_{\psi}(k')$-orbits on $Y_{w}(k')$ for $w\in \Om$. The number of such orbits are bounded by $\sum_{x}\#\pi_{0}(\Stab_{B_{\psi}}(x))$ where $x$ runs over a set of representatives of the finite set $Y_{w}(\ov k)/B_{\psi}(\ov k)$.
\end{proof}

\begin{remark}
One may generalize the notion of a (pre-)euphotic automorphic datum by adding a character $\y$ of $L_{\bQ}(k)$ (or a rank one character local system $\cK_{\y}$ on $L_{\bQ}$). We leave it to the reader to modify the third condition in Definition \ref{d:qep data} in this situation (which should involve $\cK_{\chi}$ and $\cK_{\y}$).
\end{remark}

\section{Hecke eigencategory and local systems}\label{s:sh}

Unless otherwise stated, in this subsection all $\ell$-adic sheaves are over the relevant spaces base-changed to $\ov k$.

\subsection{Automorphic sheaves}\label{ss: auto sheaves}
Let $(\bP_{\infty},\psi,\chi,\bQ_{0})$ be a euphotic automorphic datum. 
\sss{A category of automorphic sheaves} 
We denote by
\begin{equation*}
\Bun:=\Bun_G(\bQ_0, \bP_\infty^{++})
\end{equation*}
the moduli stack of $G$-bundles on $\PP^{1}$ with level structure $\bQ_0$ at $0$ and $\bP_\infty^{++}$ at $\infty$. It carries an action of $B_{\psi}\ltimes V_{\bP}$ by changing the level structure at $\infty$. The character $\psi_{k}\circ\psi:V_\bP=\bP_\infty^{+}/\bP_\infty^{++}\rightarrow \Qlbar^*$ determines an Artin-Schreier sheaf $\AS_\psi$ on $V_{\bP}$.  Similarly we get a Kummer sheaf on $T_\psi$ whose pullback along $B_\psi \rightarrow T_\psi$ we denote by $\calK_\chi$.

Let $\calD(\psi,\chi)$ be the derived category of $\Qlbar$-complexes on $\Bun_{\ov k}$ with $(B_{\psi,\ov k} \ltimes V_{\bP,\ov k},\cK_{\chi}\boxtimes \AS_{\psi})$-equivariant structures. In \cite{Yun14}, a more elaborate notion of geometric automorphic data is defined, including the data of a character sheaf on the center $ZG$. In our case we simply take the trivial local system on $ZG$. For the details we refer to \cite[Section 2.6]{Yun14}. \par

More generally, for any scheme $S$ over $k$, we define $\cD(S, \psi,\chi)$ to be the  derived category of $\Qlbar$-complexes on $S_{\ov k}\times_{\ov k} \Bun_{\ov k}$ with $(B_{\psi,\ov k} \ltimes V_{\bP,\ov k},\cK_{\chi}\boxtimes \AS_{\psi})$-equivariant structures.

Let $\cP(\psi,\chi)\subset \cD(\psi,\chi)$ and $\cP(S, \psi,\chi)\subset \cD(S, \psi,\chi)$ be the full abelian subcategory of perverse sheaves.

\begin{lemma}\label{l:simple obj}
\begin{enumerate}
\item The category  $\cP(\psi,\chi)$ has finitely many simple objects up to isomorphism.
\item Let $S$ be a scheme of finite type over $k$. Any simple perverse sheaf in $\cP(S,\psi,\chi)$ is of the form $\cF_{S}\boxtimes \cA$, where $\cF_{S}$ is a simple perverse sheaf on $S_{\ov k}$, and  $\cA\in \cP(\psi,\chi)$ is a simple object.
\end{enumerate}
\end{lemma}
\begin{proof}
All stacks in the proof are understood to be over $\ov k$;  we omit the base change  $(-)_{\ov k}$ from the notations.

(1) Stratify $\Bun$ into locally closed substacks $\Bun_{[w]}$ indexed by $[w]\in W_{\bP}\bs \tilW/W_{\bQ}$ using the Birkhoff decomposition \eqref{dcos} and \eqref{Birk}. By the discussion in \S\ref{sss:Yw}, the restriction of any $\cA\in \cD(\psi,\chi)$ to $\Bun_{[w]}$ can be identified with an object $\cA_{[w]}\in D_{(B_{\psi}, \cK_{\chi})}(Y_{w})$. By the genericity condition on the Kummer sheaf $\cK_{\chi}$ (see Definition \ref{d:qep data}), for any $w\in \tilW-W_{\bP}\Om W_{\bQ}$, any geometric point $y\in Y_{w}$, the restriction of $\cK_{\chi}$ to $\Stab_{B_{\psi}}(y)$ is nontrivial. Therefore $\cA_{[w]}=0$. This implies that any object $\cA\in \cD(\psi,\chi)$ has vanishing stalks and costalks outside the open strata $\sqcup_{w\in \Om}\Bun_{[w]}$ (one for each connected component of $\Bun$). For $w\in \Om$ let $j_{w} : Y_{w} \times V_{\bP} \hookrightarrow  L/Q_w \times V_{\bP} \cong \Bun_{[w]} \hookrightarrow \Bun$ be the embedding. Here the first arrow is a closed embedding and the second is the open inclusion of the open stratum. Then the sum of the functors $j^{*}_{w}$ gives a t-exact equivalence
\begin{equation*}
j^{*}: \cD(\psi,\chi)\isom \bigoplus_{w\in \Om}D_{(B_{\psi} \ltimes V_{\bP}, \cK_{\chi}\boxtimes\AS_{\psi})}(Y_{w} \times V_{\bP})\cong \bigoplus_{w\in \Om}D_{(B_{\psi},\cK_{\chi})}(Y_{w}).
\end{equation*}
By Remark \ref{r:sph}(2), $Y_{w}$ has finitely many $B_{\psi}$-orbits. This implies that $\Perv_{(B_{\psi},\cK_{\chi})}(Y_{w})$ has finitely many simple objects. Therefore the same is true for $\cP(\psi,\chi)$.

(2) The above argument shows that the restriction map along the embedding $\id_{S}\times j: S\times \	\sqcup_{w\in \Omega} Y_w\times V_{\bP}\incl S\times \Bun$ induces an equivalence
\begin{equation*}
\cD(S,\psi,\chi)\isom \bigoplus_{w\in \Om}D_{(B_{\psi}, \cK_{\chi})}(S\times Y_{w}).
\end{equation*}
Let $\cB\in \Perv_{(B_{\psi}, \cK_{\chi})}(S\times Y_{w})$ be a simple object, for some $w\in \Om$. Let $Y_{w}=\sqcup_{\a\in \Sigma}Z_{\a}$ be the stratification into $B_{\psi}$-orbits. Then $\cB$ is the middle extension of a local system $\cB_{0}$ on a locally closed  $B_{\psi}$-stable substack of $S\times Y_{w}$, which is necessarily of the form $S'\times Z_{\a}$ for some $\a\in \Sigma$ and $S'\subset S$ locally closed irreducible. Choose a point $z\in Z_{\a}$ and let $\G_{\a}$ be the stabilizer of $B_{\psi}$ at $z$. Then via restriction to $S'\times\{z\}$,  $(B_{\psi},\cK_{\chi})$-equivariant locally systems on $S'\times Z_{\a}$ are the same as $(\G_{\a}, \cK_{\chi}|\G_{\a})$-equivariant local systems on $S'$, with the trivial action of $\G_{\a}$ on $S'$. Let $\G_{\a}^{\circ}$ be the neutrual component of $\G_{\a}$. If $\cK_{\chi}|\G^{\circ}_{\a}$  is nontrivial, there are no such local systems. If $\cK_{\chi}|\G_{\a}^{\circ}$ is trivial, then $\cK_{\chi}|\G_{\a}$ descends to $\pi_{0}(\G_{\a})$ and gives a class $\xi\in \upH^{2}(\pi_{0}(\G_{\a}), \Qlbar^{\times})$ by \cite[A.4.1]{Yun14}. The class $\xi$ gives a central extension $1\to \Qlbar^{\times}\to \Xi\to \pi_{0}(\G_{\a})\to 1$. By \cite[A.4.4]{Yun14}, $(\G_{\a}, \cK_{\chi}|\G_{\a})$-equivariant local systems on $S'$ (with the trivial $\G_{\a}$-action) are the same as local systems on $S'$ with an action of the group $\Xi$ such that $\Qlbar^{\times}\subset \Xi$ acts by scaling on the local system. Thus any irreducible $(B_{\psi},\cK_{\chi})$-equivariant local system $\cB_{0}$ on $S'\times Z_{\a}$ must be of the form $\cF_{0}\boxtimes \rho$   for an irreducible local system $\cF_{0}$ on $S'$ and an irreducible representation $\rho$ of $\Xi$ on which $\Qlbar^{\times}$ acts by scaling. View $\rho$ as a $(B_{\psi},\cK_{\chi})$-equivariant local system $\cA_{0}$ on $Z_{\a}$. We have $\cB_{0}\cong \cF_{0}\boxtimes\cA_{0}$. Let $\cF_{S}$ and $\cA$ be the middle extensions of $\cF_{0}$ and $\cA_{0}$ respectively to $\ov S'$ and $\ov Z_{\a}$, then $\cB\cong\cF_{S}\boxtimes\cA$.
\end{proof}


\sss{Geometric Hecke operators} We briefly review the construction of geometric Hecke operators. For details we refer to \cite[\S4.2]{Yun14}. First consider the case $G$ is split. Let $\Hk$ be the Hecke correspondence
\begin{equation*}
\xymatrix{\Bun & \Hk \ar[r]^{\orr{h}}\ar[l]_{\oll{h}} & 
\Gm \times \Bun}
\end{equation*}
which classifies $(x,\cE,\cE',\a)$ where $x\in \Gm$, $\cE,\cE'\in \Bun$ and $\a$ is an isomorphism of $G$-torsors $\cE|_{X\bs \{x\}}\cong \cE'|_{X\bs \{x\}}$ respecting the level structures. The maps $\oll{h}$ and $\orr{h}$ send $(x,\cE,\cE',\a)$ to $\cE$ and $(x,\cE')$ respectively.

By the geometric Satake equivalence, for each $V\in \Rep(\dG)$ there is a $G\tl{t}\rtimes\Aut(k\tl{t})$-equivariant perverse sheaf on the affine Grassmannian $\Gr_{G}=G\lr{t}/G\tl{t}$ denoted $\IC_{V}$. This perverse sheaf can be ``spreadout'' over $\Hk$ (still denoted $\IC_{V}$) such that its $*$-restriction to every fiber of $\orr{h}$ is isomorphic to $\IC_{V}$. The geometric Hecke operator is the functor
\begin{equation*}
\bT_{V}=\orr{h}_{!}(\oll{h}^{*}(-)\ot \IC_{V}): \cD(\psi,\chi)\to \cD(\Gm, \psi,\chi).
\end{equation*}
The formation of $\bT_{V}$ is additive in $V$.

More generally, for any finite set $I$, any scheme $S$, $V\in \Rep(\dG^{I})$, there is a functor
\begin{equation*}
\bT^{I}_{S,V}: \cD(S, \psi,\chi)\to \cD(\Gm^{I}\times S, \psi,\chi)
\end{equation*}
defined using the version of the Hecke stack that modifies the bundle simultaneously at a collection of points indexed by $I$.
These functors have factorization structures: for $I=I_{1}\sqcup I_{2}$, $V_{i}\in \Rep(\dG^{I_{i}})$ and $V=V_{1}\boxtimes V_{2}$, there is a canonical isomorphism
\begin{equation}\label{factorize}
\bT^{I}_{S,V}\cong \bT^{I_{1}}_{\Gm^{I_{2}}\times S,V_{1}}\circ\bT^{I_{2}}_{S, V_{2}}.
\end{equation}
Moreover, for any surjection $I\surj J$ with the diagonal map $\dG^{J}\incl \dG^{I}$ and $\Gm^{J}\incl \Gm^{I}$, there is a canonical isomorphism
\begin{equation*}
\bT^{I}_{S,V}|_{\Gm^{J}\times S\times \Bun}\cong \bT^{J}_{S,V|_{\dG^{J}}}.
\end{equation*}

When $G$ is quasi-split, the only modification to the above discussions is that $\Gm$ should be replaced with the $\mu_{e}$-covering $\wt\Gm$ over which $G$ is split.

\begin{prop}\label{p:eigencat} Let $(\bP_{\infty}, \psi,\chi, \bQ_{0})$ be a euphotic automorphic datum. \begin{enumerate}
\item The functor $\bT_{V}[1]$ is exact for the perverse t-structures.
\item For any simple perverse sheaf $\cA\in \cP(\psi,\chi)$ and $V\in \Rep(\dG)$, $\bT_{V}(\cA)$ is isomorphic to a finite direct sum
\begin{equation*}
\bT_{V}(\cA)\cong \bigoplus_{\cA'}E(V)_{\cA,\cA'}\boxtimes \cA'.
\end{equation*}
Here $\cA'$ runs over simple objects in $\cP(\psi,\chi)$ and $E(V)_{\cA,\cA'}$ is a semisimple local system on $\GG_{m,\ov k}$.

\item For simple perverse sheaves $\cA,\cA'\in \cP(\psi,\chi)$ and $V\in \Rep(\dG)$ there is a canonical isomorphism
\begin{equation}\label{EV dual}
E(V)_{\cA,\cA'}\cong E(V^{\vee})^{\vee}_{\cA',\cA}.
\end{equation}

\item More generally, for any finite set $I$, $V\in \Rep(\dG^{I})$ and any simple perverse sheaf $\cA\in \cP(\psi,\chi)$, $\bT^{I}_{V}(\cA)$ takes the form
\begin{equation}\label{TIV}
\bT^{I}_{V}(\cA)\cong \bigoplus_{\cA'}E^{I}(V)_{\cA,\cA'}\boxtimes \cA'.
\end{equation}
Here $E^{I}(V)$ is a finite direct sum of local systems on $\Gm^{I}$ of the form $\boxtimes_{i\in I}E_{i}$, where $E_{i}$ are semisimple local systems on $\GG_{m,\ov k}$.

\item Note that $\mu_{e}$ acts on $\dG$ by pinned automorphisms and on $\wt\Gmg$ by deck transformations. For $\z\in \mu_{e}$,  $V\in \Rep(\dG^{I})$, let $V^{\z}$ be the representation given by the composition $\dG^{I}\xr{\z^{I}}\dG^{I}\to \GL(V)$; let $\z^{I}$ also denote the diagonal action of $\z$ on $\wt\Gmg^{I}$. Then there is an isomorphism functorial in $\cA,\cA'$ and $V$, and compatible with the group structure of $\mu_{e}$:
\begin{equation*}
(\z^{I})^{*}E^{I}(V)_{\cA,\cA'}\cong E^{I}(V^{\z})_{\cA,\cA'}, \quad \forall \z\in \mu_{e}.
\end{equation*}
\end{enumerate}
\end{prop}
\begin{proof} (1)(2) We first show a weak version of (1): suppose $\cA$ is supported on the neutral component, i.e., it is a clean extension from $\Bun_{[1]}$, then $\bT_{V}(\cA)[1]$ is perverse for any $V\in \Rep(\dG)$. The proof is similar to the argument in \cite[third paragraph in the proof of Theorem 3.8]{Yun16}, and we only give a sketch. The key point being that the map $\orr{h}:{}_{[1]}\Hk\to \Gm\times\Bun$ is ind-affine, where ${}_{[1]}\Hk\subset\Hk$ is the preimage of $\Bun_{[1]}$ under $\oll{h}$. This boils down to the fact that $\Bun_{[1]}\subset \Bun^{0}$ (the neutral component of $\Bun$) is the non-vanishing locus of a section of a certain determinant line bundle, which pulls back to an ample line bundle on the affine Grassmannian. Therefore ${}_{[1]}\Hk\subset\Hk$ is the nonvanishing locus of a section of a line bundle relative ample with respect to $\orr{h}$, hence the ind-affineness of $\orr{h}|{}_{[1]}\Hk$.

Then we prove a weak version of (2): $\bT_{V}(\cA)\cong \op_{\cA'}E(V)_{\cA,\cA'}\boxtimes \cA'$ for some $E(V)_{\cA,\cA'}$ that is a direct sum of {\em shifted} semisimple local systems on $\GG_{m,\ov k}$. By the decomposition theorem ($\oll{h}$ is locally a fibration and $\orr{h}$ is ind-proper), $\bT_{V}(\cA)$ is a semisimple complex on $\Gm\times\Bun$.  By Lemma \ref{l:simple obj}, we conclude that $\bT_{V}(\cA)$ has the required form for semisimple complexes $E(V)_{\cA,\cA'}\in D(\GG_{m,\ov k})$. The fact that $\bT_{V}(\cA)$ is locally constant along $\Gm$ follows from the ULA property of both $\oll{h}^{*}\cA$ and $\IC_{V}$ with respect to the projection to $\Gm$. For details we refer to the argument in\cite[Lemma 4.4.6]{Yun14}.

Now we prove (1) in general. Note that the same argument for the weaker version may fail for $\cA$ supported on other components of $\Bun$: it may happen for some $\om\in \Om$ that the boundary of $\Bun_{[\om]}$ has codimension $>1$, so it cannot be the non-vanishing locus of a section of a line bundle.

For $\om\in \Om$, let $\Bun^{\om}$ be the corresponding component of $\Bun$ (so that $\Bun_{[\om]}\subset \Bun^{\om}$ is open), and $\cD_{\om}(\psi,\chi)$ be the subcategory of $\cD(\psi,\chi)$ supported on $\Bun^{\om}$. Similarly define $\cP_{\om}(\psi,\chi)$. Let $\cB\in \cP_{\om}(\psi,\chi)$. Since $\bT_{V}(\cB)$ is locally constant along $\Gm$, it suffices to fix any geometric point $x\in \Gm$, and check that $\bT_{V}(\cB)|_{\{x\}\times \Bun}$ is perverse. Let us denote the geometric Hecke operator at $x$ by $V\star_{x}(-)=\bT_{V}(-)|_{\{x\}\times \Bun}: \cD(\psi,\chi)\to \cD(\psi,\chi)$.  The weak version of (1) that is already proven says that:
\begin{equation}\label{w1}
\mbox{For any $V\in \Rep(\dG)$, $V\star_{x}(-)$ sends $\cP_{0}(\psi,\chi)$ to $\cP(\psi,\chi)$.}
\end{equation}
 Choose $V\in \Rep(\dG)$ whose central character corresponds to $-\om$ under the canonical isomorphism $\xch(Z\dG)\cong \Om$. Then $\cK=V\star_{x}\cB\in \cD_{0}(\psi,\chi)$. Now $\cB$ is a direct summand of $(V^{\vee}\ot V)\star_{x}\cB=V^{\vee}\star_{x}(\cK)$.  By \eqref{w1}, $\cB$ is a direct summand of $V^{\vee}\star_{x}\pH^{0}\cK$, therefore a direct summand of $V^{\vee}\star_{x}\cA$ for some simple object $\cA\in \cP_{0}(\psi,\chi)$. Now for any $V'\in\Rep(\dG)$, $V'\star_{x}\cB$ is a direct summand of $V'\star_{x}(V^{\vee}\star_{x}\cA)\cong (V'\ot V^{\vee})\star_{x}\cA$, which is perverse by \eqref{w1}.

Finally (1) together with the weak version of (2) implies the full version of (2).

(3) Let $\s$ be the involution on $\Hk$ that switches the two $G$-bundles. Then there is a well-known isomorphism between $\IC_{V}$ and $\s^{*}\DD_{\orr{h}}(\IC_{V^{\vee}})$ (both on $\Hk$, where $\DD_{\orr{h}}$ denotes the relative Verdier duality with respect to the map $\orr{h}$). This follows from the similar statement for the Satake category. From this and standard sheaf-theoretic functor manipulations we obtain an adjunction
\begin{equation*}
\uHom(\bT_{V^{\vee}}(\cA'), \Qlbar\bt \cA)\cong\uHom(\Qlbar\bt\cA', \bT_{V}(\cA))
\end{equation*}
as complexes on $\Gm\times\Bun$, functorial in $\cA,\cA'\in \cD(\psi,\chi)$. In view of the decomposition in (2), we get \eqref{EV dual}.

(4) Same argument as above shows that $\bT^{I}_{V}(\cA)$ is a semisimple perverse sheaf on $\Gm^{I}\times \Bun$ locally constant along $\Gm^{I}$. By Lemma \ref{l:simple obj},  we can write $\bT^{I}_{V}(\cA)$ in the form \eqref{TIV}, with $E^{I}(V)_{\cA,\cA'}$ a local system on $\Gm^{I}$ upon shifting by $|I|$. Now the factorization structure of $\bT^{I}_{V}$ allows us to conclude that each $E^{I}(V)_{\cA,\cA'}$ is indeed an external tensor product. For example, take $I=\{1,2\}$, and $V=V_{1}\boxtimes V_{2}\in \Rep(\dG^{I})$, then $\bT_{V_{2}}(\cA)\cong \op E(V_{2})_{\cA,\cA'}\boxtimes\cA'$. Acting on both sides by $\bT_{\Gm,V_{1}}$ again, the left side becomes $\bT^{I}_{V}(\cA)$ by the factorization isomorphism \eqref{factorize}, and the right side becomes
\begin{equation*}
\bigoplus_{\cA',\cA''} E(V_{1})_{\cA',\cA''}\boxtimes E(V_{2})_{\cA,\cA'}\boxtimes \cA''
\end{equation*}
with both $\cA'$ and $\cA''$ run through simple objects in $\cP(\psi,\chi)$.  We conclude
\begin{equation}\label{EV12}
E^{\{1,2\}}(V_{1}\boxtimes V_{2})_{\cA,\cA''}\cong \bigoplus_{\cA'}E(V_{1})_{\cA',\cA''}\boxtimes E(V_{2})_{\cA,\cA'}.
\end{equation}
The case of general $I$ follows by an iteration of the same argument.

(5) follows from the $\Out(\GG)=\Out(\dG)$-equivariance of the geometric Satake equivalence. See \cite[\S4.2.3]{Yun14}, and \cite[Appendix B]{HNY}.
\end{proof}

\subsection{Eigen local systems}
Next we will extract $\LG$-local systems from the category $\cP(\psi,\chi)$.

\sss{} Let $\cP^{ss}(\psi,\chi)\subset\cP(\psi,\chi)$ be the subcategory of semisimple objects.   Let $\Loc(\wt\Gmg)$ be the tensor category of $\Qlbar$-local systems (of finite rank) over $\wt\Gmg$. Let $\cC$ be the full subcategory of $\Loc(\wt\Gmg)$ consisting of finite direct sums of simple factors of $E(V)_{\cA,\cA'}$ when $V$ runs over $\Rep(\dG)$ and $\cA,\cA'$ simple objects in $\cP(\psi,\chi)$.  Since $E(V_{1}\ot V_{2})_{\cA,\cA''}\cong \bigoplus_{\cA'}E(V_{1})_{\cA',\cA''}\ot E(V_{2})_{\cA,\cA'}$ (restricting \eqref{EV12} to the diagonal), we see that $\cC$ is stable under tensor product. Proposition \ref{p:eigencat}(3) shows that $\cC$ is closed under duality. Therefore $\cC$ is a semisimple rigid tensor subcategory of $\Loc(\wt\Gmg)$, hence neutral Tannakian. 

By Proposition \ref{p:eigencat}, $\cP=\cP^{ss}(\psi,\chi)$ is a factorizable $\cR=\Rep(\dG)$-module category with coefficients in $\cC$, in the sense of \S\ref{as:define fact mod}. Applying Corollary \ref{c:fact mod} (and Remark \ref{r:fact mod equiv} for nonsplit $G$), we get the following result.

\begin{theorem}\label{th:eigen loc} Let $(\bP_{\infty}, \psi,\chi, \bQ_{0})$ be a euphotic automorphic datum. Then there are finitely many semisimple $\mu_{e}$-equivariant $\dG$-local systems $\{E_{\s}\}_{\s\in \Sigma}$ over $\wt\Gmg$ (for some index set $\Sigma$), and a decomposition
\begin{equation*}
\cP^{ss}(\psi,\chi)=\bigoplus_{\s\in \Sigma} \cP_{\s}
\end{equation*}
such that
\begin{enumerate}
\item Let $\dG_{\s}=\Aut_{\dG}(E_{\s})$  (a reductive group over $\Qlbar$ which can be identified with subgroup of $\dG$ up to conjugacy; it is equipped with a $\mu_{e}$-action). Each $\cP_{\s}$ is an $E_{2}$-module category under $\Rep(\dG_{\s})$ (the action of $W\in \Rep(\dG_{\s})$ on $\cA\in \cP_{\s}$ by $W\bu\cA$) with functorial isomorphisms $W\bu\cA\cong W^{\z}\bu\cA$ for $\z\in \mu_{e}$.

\item For any $\cA\in \cP_{\s}$ and $V\in \Rep(\dG)$, there is an isomorphism functorial in $V$ and $\cA$
\begin{equation*}
\bT_{V}(\cA)\cong \bigoplus_{E}E\boxtimes ([E_{\s}(V):E]\bu \cA)\in \cD(\Gm, \psi,\chi)
\end{equation*}
Here the direct sum is over all irreducible local systems $E$ over $\wt\Gmg$, $E_{\s}(V)\in\Loc(\wt\Gmg)$ is the (semisimple) local system on $\wt\Gmg$ associated to $E_{\s}$ and $V$, and $[E_{\s}(V):E]$ is the multiplicity space of $E$ in $E_{\s}(V)$, viewed as an object in $\Rep(\dG_{\s})$. 

Moreover, there is a version of the above isomorphism for any finite set $I$ and $V\in \Rep(\dG^{I})$, and these isomorphisms are compatible with the factorization structures.

\end{enumerate}
\end{theorem}

We recall that a $\dG$-local system on $\wt\Gmg$ is called semisimple if the Zariski closure of the image of $\pi_{1}(\wt\Gmg)$ in $\dG$ is reductive. 

\sss{Descent to $k$}\label{sss:des}
After a finite extension of $k$, we may assume that each simple object $\cA$ in $\cP(\psi,\chi)$ has the property $\Fr^{*}\cA\cong \cA$, where $\Fr:\Bun\to \Bun$ is the Frobenius morphism with respect to $k$. In particular, each summand $\cP_{\s}$ in Theorem \ref{th:eigen loc} is stable under $\Fr^{*}$. Fix a Weil structure $\Fr^{*}\cA\cong \cA$ for each simple $\cA\in \cP(\psi,\chi)$. Applying Remark \ref{r:fact mod equiv} to the $\G=\Fr^{\ZZ}$-equivariant structure on the factorizable $\Rep(\dG)$-module structure on $\cP^{ss}(\psi,\chi)$ with coefficients in $\cC$ (where $\Fr$ acts trivially on $\Rep(\dG)$ and on $\cP^{ss}(\psi,\chi)$ by the fixed Weil structures, and it acts by $\Fr^{*}$ on $\cC$), we conclude that each $E_{\s}$ viewed as a tensor functor $\Rep(\dG)\to \cC$ carries a $\Fr$-equivariant structure. In other words,  $E_{\s}$ carries a Weil structure (depending on the choice of Weil structures of simple objects in $\cP(\psi,\chi)$).

\begin{remark}
\begin{enumerate}
\item A $\mu_{e}$-equivariant $\dG$-local system on $\wt\Gmg$ is the same thing as a $\LG=\dG\rtimes\mu_{e}$-local system on $\Gmg$, such that the induced $\mu_{e}$-cover of $\Gmg$ is $\wt\Gmg$. Therefore, after choosing a base point $x\in \Gmg$, the $E_{\s}$ in the above theorem is the same data as continuous homomorphisms $\r_{\s}: \pi_{1}(\Gmg,x)\to \LG(\Qlbar)$. The discussion in  \S\ref{sss:des} gives an extension of $\r_{\s}$ to the Weil group of $\Gm$ so it is a Langlands parameter in the usual sense for the quasi-split group $G$ over $F=k(t)$.

\item The upshot of the above theorem is that, we not only can extract a $\LG$-local system $E_{\s}$ from each indecomposable summand $\cP_{\s}$ of the category $\cP^{ss}(\psi,\chi)$, but there is a {\em residual action} of $\Rep(\Aut(E_{\s}))$ on $\cP_{\s}$. The $\Rep(\Aut(E_{\s}))$-module category $\cP_{\s}$ may be viewed as a secondary invariant attached to the automorphic datum in question that is not covered by the usual Langlands parameter $E_{\s}$. The relationship between this secondary invariant and global $L$-packets deserves further study. A closely related phenomenon is discussed in \cite{FW} under the name {\em fractional Hecke eigensheaves}.

\item The number of simple objects in $\cP(\psi,\chi)$ can be large in some examples. We will see in an example for $G=\Sp_{2n}$ (see \S\ref{orbitC}, case  (1)) that $\cP(\psi,\chi)$ has $2^{n}$ simple objects all supported on the open $B_{\psi}$-orbit of $L/Q$. It remains unclear how $\cP^{ss}(\psi,\chi)$ decomposes into indecomposables in this case.  

\item The above theorem (or rather Corollary \ref{c:fact mod}) improves \cite[Theorem 4.4.2]{Yun14}. In {\em loc.cit}, we consider the situation where the relevant category of perverse sheaves $\cP$ has a unique simple object on each connected component of $\Bun$ (indexed by $\Om$, and $\Om\cong \xch(Z\dG)$ when $G$ is split).  In this case, $\cP$ is necessarily indecomposable because the Hecke operators will mix up the components transitively. Corollary \ref{c:fact mod} then gives a $\dG$-local system $E$ (rather than a weak $\dG$-local system in the sense of \cite[Def. 4.3.2]{Yun14}), together with an $E_{2}$-action of $\Rep(\Aut(E))$ on $\cP$. Since $Z\dG\subset \Aut(E)$, $\Rep(\Aut(E))$ is graded by $\xch(Z\dG)$, and this grading is compatible with the decomposition of $\cP$ according the connected components of $\Bun$.

\item In practice, to calculate these local systems, we need to calculate the local systems $E(V)_{\cA,\cA'}$ that appear in Proposition  \ref{p:eigencat}. These can be calculated in the same way as described in \cite{Yun14}, as part of the derived direct image of a family of varieties over $\Gm$ closely related to the Beilinson-Drinfeld affine Grassmannian.
\end{enumerate}
\end{remark}

\begin{cor} \label{cor:strict eigensheaf} Assume that $G$ is split and simply-connected. Let $(\bP_\infty, \psi, \chi, \bQ_0)$ be a strict euphotic automorphic datum. Denote by $ZG^*$ the set of characters $\sigma:ZG(k)\rightarrow \Qlbar^\times$. Then there is a decomposition 
\[\calP(\psi,\chi)=\bigoplus_{\sigma\in ZG^*} \cP_\sigma\]
such that each $\cP_\sigma$ contains a unique simple perverse sheaf $\cA_\sigma$ which is a Hecke eigensheaf with semisimple eigen $\dG$-local system $E_\sigma$. 
\end{cor}
\begin{proof} 
Denote by $O$ the open $B_\psi$-orbit of $L/Q$. Using strictness the proof of Lemma \ref{l:simple obj} implies an equivalence of categories
\[\cD(\psi,\chi) \xrightarrow{\cong} D_{(B_\psi,\cK_\chi)}(L/Q) \xrightarrow{\cong} D_{(B_\psi,\cK_\chi)}(O). \]
Fixing a point $y\in O(\ov{k})$ identifies $\cP(\psi,\chi)$ with $\Rep(ZG,\Qlbar)$ and hence we obtain the desired decomposition and uniqueness of the simple perverse sheaf $\calA_\sigma$. More explicitly $\calA_\sigma=j_!\cF_\sigma$ is the perverse sheaf whose restriction to $O$ corresponds to the character $\sigma$.  For any $\cH\in \cD(S,\psi,\chi)$ we may decompose its restriction to $S\times y$ as follows 
\[(\id_S\times y)^*\cH \cong \bigoplus_{\sigma \in ZG^*} (\id_S\times y)^*\cH_{\sigma}, \]
according to the action of $ZG$. For $\cH\in \cP(S,\psi,\chi)$ (or any shifted perverse sheaf) we obtain a decomposition 
\[\cH\cong \bigoplus_{\sigma\in ZG^*} (\sigma^{-1}\otimes(\id_S\times y)^*\cH_\sigma)^{ZG}\boxtimes  j_!\cF_\sigma,\]
cf. \cite{Yun14}. Note that the automorphism group of any point of $\Bun_G(\bQ_0,B_\psi\bP_\infty^+)$ contains $ZG$. Therefore we may speak of the subcategory $\calD(\psi,\chi)_\sigma$ on which $ZG$ acts through $\sigma$. By \cite[\S 4.4.1]{Yun14} the geometric Hecke operator $\bT_V$ sends $\calD(\psi,\chi)_\sigma$ to $\calD(\Gm,\psi,\chi)_\sigma$ and by Proposition \ref{p:eigencat} $\bT_V(\cA_\sigma)[1]$ is perverse. Therefore we have
\[\bT_V(\cA_\sigma)=\bT_V(\cA_\sigma)_\sigma \cong ({\sigma}^{-1}\otimes(\id_S\times y)^*\bT_V(\cA_\sigma))^{ZG} \boxtimes  j_!\cF_\sigma \] 
using the decomposition above. Again by Proposition \ref{p:eigencat} $\bT_V(\cA_\sigma)$ is locally constant along $\Gm$ and the claim follows.

\end{proof}

%
%
%

\subsection{Local monodromy and rigidity}\label{ss:rigidity} In this subsection we assume that $G$ is split, so that $\LG=\dG$. Let $(\bP_{\infty}, \psi,\chi, \bQ_{0})$ be a euphotic automorphic datum. By Theorem \ref{th:eigen loc}, we have a $\dG$-local system $E_{\s}$ over $\GG_{m,\ov k}$ for each indecomposable summand $\cP_{\s}$ of $\cP^{ss}(\psi,\chi)$.  Let $\r_{\s}: \pi_{1}(\GG_{m,\ov k},\ov \y)\to \dG(\Qlbar)$ be the geometric monodromy representation attached to $E_{\s}$.

\sss{Lusztig's bijection} To describe the local monodromy of $E_{\s}$ at $0$, recall Lusztig's bijection \cite[Theorem 4.8]{Lus89}:
\begin{equation}\label{L bij}
\xymatrix{\{\textup{two-sided cells in }W_{\aff}\}\ar@{<->}[r] & \{ \textup{unipotent classes in } \dG\}}
\end{equation}

\begin{prop}\label{p:mono 0} 
For any local system $E_{\s}$ attached to the euphotic automorphic datum 
\[(\bP_{\infty}, \psi,\chi, \bQ_{0}),\]

the local monodromy $\r_{\s}|_{I_{F_{0}}}$ ($I_{F_{0}}$ is the inertia group at $0$) is tame, and maps a topological generator of $I^{\tame}_{F_{0}}$ into the unipotent class of $\dG$ which corresponds to the two-sided cell $c_{\bQ}$ of $W_{\aff}$ containing the longest element of $W_{\bQ}$ under Lusztig's bijection \eqref{L bij}. 
\end{prop}
\begin{proof} It suffices to prove the case where $G$ is simply-connected. The proof is almost the same as in the epipelagic case, cf. \cite[\S4.11-4.18]{Yun16}, replacing $\bP_0$ in {\em loc.cit.} by $\bQ_0$.  The only thing that needs to be adapted in our situation is \cite[Lemma 4.12]{Yun16}. Here the analogous statement should be: consider the action of $D_{\bQ_{0}}(G\lr{t}/\bQ_{0})$ on $\cD(\psi,\chi)$, then any perverse sheaf $\cK\in \Perv_{\bQ_{0}}(G\lr{t}/\bQ_{0})$ acts as a t-exact endo-functor of $\cD(\psi,\chi)$. Consider the Hecke correspondence $\Hk_{0}$ that classifies modifications of $\Bun$ at $0$, with two maps $\oll{h},\orr{h}: \Hk_{0}\to\Bun$. Let $\Hk'_{0}=\oll{h}^{-1}(\Bun_{[1]})\cap \orr{h}^{-1}(\Bun_{[1]})$. Since $\Bun_{[1]}$ is the preimage of the  open stratum $[\pt/L]\subset \Bun_{G}(\bP_{0}, \bP_{\infty})$, we may identify fibers of $\orr{h}'=\orr{h}|\Hk'_{0}: \Hk'_{0}\to \Bun_{[1]}$ with $\Aut_{X-\{0\}}(\cE_{\bP_{0},\bP_{\infty}})/L$ where $\cE_{\bP_{0},\bP_{\infty}}$ is the open point in  $\Bun_{G}(\bP_{0}, \bP_{\infty})$ with automorphism $L$. Equivalently, we may identify $\Aut_{X-\{0\}}(\cE_{\bP_{0},\bP_{\infty}})/L$  with $\Aut_{X-\{0\}}(\cE_{\bP_{0},\bP^{+}_{\infty}})$ by choosing a $\bP^{+}_{\infty}$-reduction of $\cE_{\bP_{0},\bP^{+}_{\infty}}$. From this we see that the fibers of $\orr{h}'$ are ind-affine. Since $\orr{h}'$ is a Zariski locally trivial fibration, it is ind-affine. 
\end{proof}

\begin{prop}\label{p:rig} Under the assumptions in the beginning of \S\ref{ss:rigidity}, assume further
\begin{itemize}
\item The restriction $\r_{\s}|_{I_{F_{\infty}}}$ is as predicted in \S\ref{sss:local L param}.
\item The image of $\r_{\s}$ does not lie in any proper Levi subgroup of $\dG$ (equivalently, $\Aut(E_{\s})$ is finite).
\end{itemize}
Then $E_{\s}$ is cohomologically rigid in the sense that
\begin{equation*}
\cohog{*}{X_{\kbar}, j_{!*}\Ad(E_{\s})}=0.
\end{equation*} 
Here $j:\Gm\incl X$ is the open embedding,  and $\Ad(E_{\s})$ is the adjoint local system attached to $E_{\s}$.
\end{prop}
\begin{proof}
We have the exact sequence
\begin{equation*}
0\to \dg^{\r_{\s}(\pi_{1}(\GG_{m,\kbar}))}\to \dg^{\r_{\s}(I_{F_{0}})}\oplus \dg^{\r_{\s}(I_{F_{\infty}})} \to \cohoc{1}{\GG_{m,\kbar}, \Ad(E_{\s})} \to \cohog{1}{X_{\kbar}, j_{!*}\Ad(E_{\s})}\to 0.
\end{equation*}
So the cohomological rigidity condition boils down to
\begin{equation}\label{sum0}
\dim\cohoc{1}{\GG_{m,\kbar}, \Ad(E_{\s})}-\dim\dg^{\r_{\s}(I_{F_{0}})}-\dim\dg^{\r_{\s}(I_{F_{\infty}})}+\dg^{\r_{\s}(\pi_{1}(\GG_{m,\kbar}))}=0.
\end{equation}
By the second assumption, $\dim\dg^{\r_{\s}(\pi_{1}(\GG_{m,\kbar}))}=0$. Let $u\in \dG$ be a unipotent element in the class corresponding to the longest element $w_{\bQ,0}$ of $W_{\bQ}$. By Proposition \ref{p:mono 0}, $\dim\dg^{\r_{\s}(I_{F_{0}})}=\dim\dG_{u}=\rk \dG+2\dim \cB_{u}$ ($\cB_{u}$ is the Springer fiber of $u$).  By \cite[Lemma 4.6]{Yun16}, we have $\dim \cB_{u}=\ell(w_{\bQ,0})$. Hence
\begin{equation*}
\dim \dg^{\r_{\s}(I_{F_{0}})}=\rk \dG+2\ell(w_{\bQ,0})=\dim L_{\bQ}. 
\end{equation*}
By \eqref{ginv}, $\dim\dg^{\r_{\s}(I_{F_{\infty}})}=\dim T_{\psi}$. 
By the Grothendieck-Ogg-Shafarevich formula, $\dim\cohoc{1}{\GG_{m,\kbar}, \Ad(E_{\s})}=\Swan(\Ad(E_{\s}))=\dim L-\dim L_{\psi}$ as predicted in \S\ref{sss:local L param}. Using these calculations, \eqref{sum0} is equivalent to
\begin{equation*}
\dim L-\dim L_{\bQ}-\dim T_{\psi}=\dim L_{\psi}.
\end{equation*}
Since $\dim B_{\psi}=(\dim T_{\psi}+\dim L_{\psi})/2$, $\dim L-\dim L_{\bQ}=2\dim L/Q$, the above identity is equivalent to
\begin{equation*}
\dim L/Q=\dim B_{\psi},
\end{equation*}
which is guaranteed by the condition \eqref{open B orbit} in Definition \ref{d:qep data}.
\end{proof}


\section{An example in type $G_2$} \label{s:G2}
From this section on, we will give several families of examples of (strict) euphotic automorphic data.

\subsection{The rigid connection from \cite{Ja20}} The motivating example for our construction of rigid automorphic data is a certain rigid irregular $G_2$-connection discovered by the first-named author. By \cite[Theorem 1.1.]{Ja20} there is a rigid irregular connection $\calE$ on $\GG_{m,\CC}$ with differential Galois group $G_2$ and with the following local data. At $z=0$ the connection is regular singular and has subregular unipotent monodromy. On the punctured formal disc $D^{\circ}$ at $z=\infty$ the connection $\calE$ is isomorphic to 
\[\textup{El}(z^2, \alpha, (\lambda, \lambda^{-1}))\oplus \textup{El}(z^2, 2\alpha, 1)\oplus (-1) \]
where by $(\lambda, \lambda^{-1})$ we denote a regular singular formal connection of rank two with monodromy $\lambda$ and $\lambda^{-1}$ and similarly for $(-1)$. The formal connection $\textup{El}(z^2, \alpha, (\lambda, \lambda^{-1}))$ is an elementary connection in the sense of \cite[\S 2]{Sa08}. It is the direct image of a formal exponential connection twisted by a regular singular connection along a twofold covering of the formal disc. \par
The local data of this connection dictates our guess for the local representations in \S \ref{sss:auto conditions}. The parahoric subgroup $\bQ_0$ should correspond to the unipotent conjugacy class of the connection at $z=0$ as in \S\ref{ss:rigidity}, (1). \par
The choice of the character $\psi$ corresponds to the occurence of the formal exponential connection, an additive parameter, and the character $\chi$ reflects the multiplicative parameter at $z=\infty$ of the formal connection, given by the regular singular connection $(\lambda, \lambda^{-1})$. Note that in addition the formal connection at $z=\infty$ becomes diagonalisable after pullback to a two-fold cover. 

\subsection{Constructing the automorphic form}
Assume $G$ is split of type $G_2$ and denote by $\Delta=\{\alpha_1, \alpha_2\}$ the simple roots of $G_2$ where $\alpha_1$ is the long root.
Consider the parahoric $\bP_{\infty}$ with $L\cong\SO_{4}$ with roots $\alpha_2$ and the highest root $\eta$. We have $V=\Sym^{3}(\St)\ot\St'$, where $\St$ is the standard representation of the short root $\SL_{2}\incl \SO_{4}$, and $\St'$ is the standard representation of the long root $\SL_{2}$. In this case $m=2$ and $V\cong V^{*}$. We may identify $V^{*}$ with the space of bihomogeneous polynomials in two sets of variables $(x,y)$ and $(u,v)$ that are cubic in $(x,y)$ and linear in $(u,v)$. Then take $\psi=x^{3}u+y^{3}v$. We have $L_{\psi}\cong \Gm\rtimes \mu_{2}$: the projection $L_{\psi}\to \PGL_{2}$ to the short root factor is an isomorphism onto the normalizer of a maximal torus $A$ in $\PGL_{2}$; the other projection $L_{\psi}\to \PGL_{2}$ onto the long root factor has image $N_{\PGL_{2}}(A)$ with kernel $\mu_{3}$. We then have $T_\psi=B_{\psi}=\Gm\subset L_{\psi}$ (with index $2$) acting as $t\cdot(x,y,u,v)=(tx,t^{-1}y, t^{-3}u, t^{3}v)$. Take $Q\subset L$ such that $L/Q\cong \PP^{1}$ is the flag variety of the short root factor. The choice of $Q$ determines the parahoric $\bQ_0$. Then $T_{\psi}$ acts on $L/Q$ with an open free orbit.
\begin{prop} Let $\chi:T_\psi(k) \rightarrow \Qlbar^\times$ be a non-trivial character. The automorphic datum $(\bP_\infty,\psi,\chi,\bQ_0)$ is euphotic and strict. 
\end{prop}
\begin{proof} Note that for dimension reasons any point in $L/Q$ outside the open $B_\psi$-orbit will have a positive dimensional stabilizer. Therefore since $L_\psi$ is a torus this immediately implies that $(\bP_\infty, \psi, \chi, \bQ_0)$ is strict if it is euphotic. \par For $w\neq \id$ we will prove that if $Y_w\neq \emptyset$ then $Y_w$ is finite. The first step is to single out the cases in which $Y_w$ is empty. \par
Suppose $w$ is given such that all weights of $V_w^\perp$ lie in a half-space in $\xch(A)_\RR$ not containing $0$. For ${}^{\ell^{-1}} \psi \in V_w^\perp$ we can then find a torus $T'$ such that $0\in \overline{T'.{}^{\ell^{-1}}\psi}$ and since the orbit of $\psi$ is closed this implies $\psi=0$, a contradiction. In this case we therefore get $Y_w=\emptyset$. \par
Let $\beta=\alpha_1+3\alpha_2$ and suppose $w$ is given such that $-\alpha_1$ and $\beta$ (resp. $\alpha_1$ and $-\beta$) are not weights of $V_w^\perp$. In this case every $v\in V_w^{\perp}$ is a reducible polynomial contradicting the irreducibility of $\psi$. Again this implies $Y_w=\emptyset$. \par 
These observations determine a region $\cU \subset \xcoch(A)_\RR$ such that $Y_w=\emptyset$ if $wx_{\bQ}\notin \cU$ in the following way. Recall that for example $-\beta$ is a weight of $V_w^{\perp}$ if and only if 
\[\langle -\beta,wx_{\bQ}-x_{\bP}\rangle < \frac{1}{2}. \]
This is equivalent to 
\[\langle -\beta,wx_{\bQ}-x_{\bP}\rangle \le \frac{1}{2}-\frac{1}{5} \]
and noting that $\beta(x_{\bP})=1/2$ this is furthermore equivalent to 
\[\langle \beta, wx_{\bQ} \rangle \ge \frac{1}{5} .\]
Combining all cases in which $Y_w=\emptyset$ we find that if $Y_w\neq \emptyset$ then $wx_{\bQ}\in \cU$ with 
\[\cU = \{x\in \xcoch(A)_\QQ \mid 0 < \langle \a_1, x\rangle < 1  \textup{\, or \, } 0 < \langle \b, x\rangle < 1 \},\]
a union of two strips in the plane. 
\begin{figure} 

\resizebox{8cm}{!}{

\begin{tikzpicture}

\clip (-5,-5) rectangle (5,5);

\foreach\i in {-2,-1,0,2}{
 \draw[-, thick](\i*1.73, -5)--(\i*1.73,5);
}
 \draw[-, thick, name path=v](1.73, -5)--(1.73,5);
\foreach\i in {-1,1}{
 \draw[-, thick](-6,\i*3)--(6,\i*3);
}
 \draw[-, thick, name path=h](-6,0)--(6,0);

\foreach\i in {-3,...,3}{
 \draw[-, thick](-1.73*3,-3+\i*2)--(1.73*3,3+\i*2);
}
\draw[-,very thick,red, name path=a](-1.73*3,-3)--(1.73*3,3);
\draw[-,very thick,red, name path=b](-1.73*3,-3-2)--(1.73*3,3-2);

\foreach\i in {-3,...,3}{
 \draw[-, thick](-1.73*3,3+\i*2)--(1.73*3,-3+\i*2);
}
\draw[-,very thick,red, name path=c](-1.73*3,3)--(1.73*3,-3);
\draw[-,very thick,red, name path=d](-1.73*3,3+2)--(1.73*3,-3+2);

\foreach\i in {-2,...,2}{
 \draw[-, thick](\i*3.46-3.46,-6)--(\i*3.46+3.46,6);
}

 \draw[-, very thick, blue](-3.46,-6)--(3.46,6) ;

\foreach\i in {-2,...,2}{
 \draw[-, thick](\i*3.46+3.46,-6)--(\i*3.46-3.46,6);
}
\path [name intersections={of=a and c,by=x1}];
\path [name intersections={of=b and c,by=x2}];
\path [name intersections={of=v and h,by=x3}];

\fill [opacity=0.4,gray]
            (x1) \foreach \i in {2,3}{ -- (x\i) } -- cycle; 

\end{tikzpicture}
}

\caption{Affine root system of type $G_2$ with the region $\cU$, the line ${\alpha_1+2\alpha_2=0}$ and the fundamental alcove. }

\end{figure}
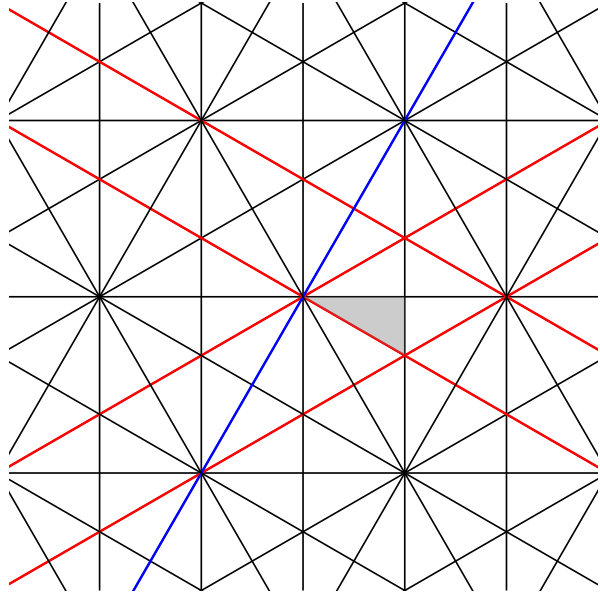

It remains to prove that $Y_w$ is finite whenever $wx_{\bQ}\in \cU$. By $W_L$-symmetry it suffices to consider just one quadrant, e.g. the one defined by $\langle \eta,x\rangle >0$ and $\langle \alpha_2 , x \rangle >0$. This leaves us with the case where $wx_{\bQ}$ lies in the above quadrant and additionally satisifes $\langle \alpha_1+2\alpha_2, wx_{\bQ} \rangle <0$ and the case where $w=s_{\alpha_1+\alpha_2}$ is a simple reflection across the hyperplane perpendicular to $\alpha_1+\alpha_2$. Note that $Y_{w'} \subset Y_w$ whenever $w'$ is in the first case and $w=s_{\alpha_1+\alpha_2}$, so actually it suffices to prove that $Y_w$ is finite in the second case. \par
In this case we have that $L/Q_w=\PP^1\times \PP^1$ and $V_w^*$ is the direct sum of weight spaces corresponding to 
\[\{\alpha_1, \a_1+\a_2, \a_1+2\a_2, \beta, -(\a_1+2\a_2), -\beta\}.\]
Let $C\subset \PP^1_{(x:y)}\times \PP^1_{(u:v)}=L/Q_w$ be the curve defined by $\psi=x^3u+y^3v=0$. If ${}^{\ell^{-1}} \psi\in V_w^*$ then $\psi(\ell \mod Q_w)=0$. Moreover in this case the projection $\pi:C\rightarrow \PP^1$ onto the second factor (which is a finite map of degree $3$) is ramified in $\ell \mod Q_w$. Thus $Y_w$ is contained in the ramification locus of $\pi$ which is finite. 
\end{proof}

\begin{cor}  There is a unique cuspidal automorphic representation satisfying 
\begin{enumerate}
\item $\pi_{x}$ is unramified for $x\ne 0,\infty$;
\item $\pi_{0}^{\bQ_{0}}\ne0$;
\item $\pi_{\infty}$ is euphotic with respect to $(\bP_{\infty}, \psi,\chi)$.
\end{enumerate} 
In addition $\dim \pi_{0}^{\bQ_{0}}=\dim \pi_{\infty}^{(\bP_{\infty}^{+},\psi)}=1$ and $\pi$ appears with multiplicity one in the automorphic spectrum of $G$.
\end{cor}
\begin{proof} This is immediate from the proof of Proposition \ref{c:fdim}. In this case the space of functions $\calF$ is one-dimensional and the statement follows. 
\end{proof}
Corollary \ref{cor:strict eigensheaf} implies the following geometric version of the above statement.
\begin{cor}
There is a Hecke eigensheaf $A_\pi$ on $\Bun_{G_2}(\bQ_0, \bP_\infty^{++})$ with semisimple eigen $G_2$-local system $E_\pi$. Under the assumptions in \S \ref{ss:rigidity} the local system $E_\pi$ is cohomologically rigid. 
\end{cor}

\section{The hyperspecial cases}\label{s: hyperspecial}

\subsection{The setup}

\sss{} In this section and the next, we assume that $G$ over $F$ is split and simply-connected. We consider the special case where $\bP_{\infty}=G(\cO_{\infty})$. The reductive quotient of $\bP_{\infty}$ over $k$ is $\GG$; by abuse of notation we will also denote $\GG$ by $G$, $\TT$ by $T$, etc.

In this case, the grading on $\frg$ is trivial, and $\psi\in\frg=\frg(-1)$. Extending $k$ if necessary, we may assume $\psi\in \Lie  T$. Then $L_{\psi}=G_{\psi}$ is a Levi subgroup of $G$.  We will use $P_{\psi}$ to denote a parabolic subgroup of $G$ containing $G_{\psi}$ as a Levi subgroup. Note that only the associate class of $P_{\psi}$ is well-defined.

Recall that $Q$ denotes another parabolic subgroup of $G$, the level at $0$, chosen in such a way that any Borel subgroup $B_\psi\subset G_\psi$ acts on the partial flag variety $G/Q$ with an open almost free orbit. This is equivalent to requiring that $G/P_\psi\times G/Q$ is a spherical $G$-variety and 
\begin{equation}\label{dim eq}
\dim G_{\psi}+\dim L_{Q}=\#\Phi_{G}
\end{equation}
where $\Phi_G$ is the set of roots of $G$ and $L_Q$ is the Levi quotient of $Q$. 

\sss{} Stembridge \cite{St03} has classified pairs of parabolic subgroups $(P_{\psi},Q)$ such that  $G/P_{\psi}\times G/Q$ is $G$-spherical. In type $A$ and $C$ this was preceded by work of Magyar-Weymann-Zelevinsky, see \cite{MWZ99} and \cite{MWZ00}. In this classification, the following are the ones that satisfy the dimension equality \eqref{dim eq}. There are no examples of exceptional types. 

\sss{Notation}

In the sequel we will concentrate on the case where $G$ is one of the groups $\SL(V), \Spin(V)$ or $\Sp(V)$, for some finite-dimensional vector space $V$ over $k$ ($\ch(k)\ne2$) equipped with a quadratic form in the case $G=\Spin(V)$ or a symplectic form in the case $G=\Sp(V)$. 

For $d\ge1$,  write $P_{d}\subset G$ for the stabilizer of a $d$-dimensional subspace, isotropic in the case outside type $A$. Similarly, for $1\le d<d'$, let $P_{d,d'}$ denote the stabilizer of a $d$-dimensional subspace inside a $d'$-dimensional subspace, both being isotropic outside type $A$.  

For parabolic subgroups $P'$ and $P''$ of $G$,  we write $(P_\psi,Q)\sim (P',P'') $ to denote that $P_\psi$ is conjugate to  $P'$ and $Q$ is conjugate to  $P''$. \par

Below we often base change to $\overline{k}$ without changing the notation. \par

\subsection{Type $A_{n-1}$, $n\ge 2$, \cite[Theorem 2.4]{MWZ99} }\label{orbitA} Let $\GG=\SL_{n}$ and let $G$ be the split form of $\GG$ over $F$.  Let $\l_{Q}$ be the partition of $n$ determined by the blocks of a Levi factor of $Q$. Let $\l_{\psi}$ be the  partition of $n$ determined by the multiplicities of the eigenvalues of $\psi$ (for example, $\l_{\psi}=(223)$ means that $\psi$ has three distinct eigenvalues, with multiplicities $2$, $2$ and $3$). 
The following are the only cases where $B_{\psi}$ has an open orbit on $L/Q$ with finite stabilizers.
\begin{enumerate}
\item $n\ge2$, $\l_{Q}=(1,n-1)$, $\l_{\psi}=(1^{n})$; 
\item $n\ge2$, $\l_{Q}=(1^{n})$, $\l_{\psi}=(1, n-1)$;
\item $n=2m$, $m\ge2$, $\l_{Q}=(m,m)$, $\l_{\psi}=(m,m-1,1)$;
\item $n=2m$, $m\ge2$, $\l_{Q}=(m,m-1,1)$, $\l_{\psi}=(m,m)$;
\item $n=2m+1$, $m\ge2$,$\l_{Q}=(m+1,m)$, $\l_{\psi}=(m,m,1)$;
\item $n=2m+1$, $m\ge2$, $\l_{Q}=(m,m,1)$, $\l_{\psi}=(m+1,m)$;
\item $n=6$, $\l_{Q}=(2,2,2)$, $\l_{\psi}=(4,2)$;
\item $n=6$, $\l_{Q}=(4,2)$, $\l_{\psi}=(2,2,2)$.
\end{enumerate}

\begin{remark} Any tame rigid local system on $\PP^{1}-\{0,1,\infty\}$ is determined by the collection of conjugacy classes of its local monodromies around the punctures. For tame rigid local systems on $\PP^{1}-\{0,1,\infty\}$ of rank $n$ with generic semisimple regular monodromy at one puncture Simpson classifies the possible Jordan types of local monodromies in \cite[Theorem 4]{Si91}. They are in canonical bijection with the above list (up to interchanging $\l_{Q}$ and $\l_{\psi}$) in the sense that the collections $(\l_{Q}, \l_\psi, (1^n))$ exhaust Simpson's list.
\end{remark}

\subsection{Type $B_{n}$, $n\ge2$, \cite[Corollary 1.3.B.]{St03}} \label{orbitB} 
Let $\GG=\Spin(2n+1)$ and let $G$ be the split form of $\GG$ over $F$. The action of $B_{\psi}$ on $G/Q$ has an open orbit with finite stabilizers if we have one of the following.
\begin{enumerate}
\item Any $n$, $(P_{\psi}, Q)\sim (P_{n}, P_{n})$  (Siegel parabolic);
\item $n=2$,  $(P_{\psi}, Q)\sim (P_{1},P_{2})$;
\item $n=2$, $(P_{\psi},Q) \sim (P_{2}, P_{1}) $;
\item $n=3$, $(P_{\psi}, Q) \sim (P_{1},P_{2})$;
\item $n=3$, $(P_{\psi},Q) \sim (P_{2}, P_{1}) $.
\end{enumerate}

\subsection{Type $C_{n}, n\ge3$, \cite[Corollary 1.3.C.]{St03}}  \label{orbitC}Let $\GG=\Sp(2n)$ and let $G$ be the split form of $\GG$ over $F$. The action of $B_{\psi}$ on $G/Q$ has an open orbit with finite stabilizers if we have one of the following.
\begin{enumerate}
\item Any $n$, $(P_{\psi}, Q)\sim (P_{n}, P_{n})$  (Siegel parabolic);
\item $n=3$, $(P_{\psi},Q)\sim (P_{1}, P_{2})$;
\item $n=3$, $(P_{\psi},Q) \sim (P_{2}, P_{1}) $.
\end{enumerate}

\subsection{Type $D_{n}$, \cite[Corollary 1.3.D.]{St03}}  \label{orbitD} Let $\GG=\Spin(2n)$ and let $G$ be the split form of $\GG$ over $F$. Note that there are two conjugacy classes of $n$-dimensional isotropic subspaces (permuted by $O(2n)$) whose stabilizers we simply denote by $P_{n}$ (two conjugacy classes of maximal parabolics of $G$). The action of $B_{\psi}$ on $G/Q$ has an open orbit with finite stabilizers if we have one of the following.
\begin{enumerate}
\item $n=4$, $(P_{\psi}, Q)\sim (P_{4}, P_{1,2})$, or anything in the same orbit under the outer automorphisms of $G$;
\item $n=4$, $(P_{\psi}, Q)\sim (P_{1,2}, P_{4})$, or anything in the same orbit under the outer automorphisms of $G$;
\item $n=5$,  $(P_{\psi}, Q)\sim (P_{5}, P_{3})$;
\item $n=5$,  $(P_{\psi}, Q)\sim (P_{3}, P_{5})$;
\item $n=6$, $(P_{\psi}, Q)\sim (P_{6}, P_{3})$;
\item $n=6$, $(P_{\psi}, Q)\sim (P_{3}, P_{6})$.
\end{enumerate}

\begin{theorem} \label{thm: hypersp quasi-epi} Assume $\psi$ and $Q$ are in any of the above cases. Recall that $\psi\in\Lie T$ so that $T_{\psi}=T$. 
\begin{enumerate}
\item Assume that $\chi:T(k) \rightarrow \Qlbar^\times$ is non-trivial on the connected center of any maximal Levi subgroup of $G$ containing $T$. Then the automorphic datum $(G(\calO_\infty),\psi, \chi, \bQ_0)$ is euphotic. 
\item All cases in type $A$ and $D$ with $\chi$ as in (1) are strict euphotic automorphic data. In type $B$ and $C$ no case is strict. 
\end{enumerate}
\end{theorem}

The proof of Theorem \ref{thm: hypersp quasi-epi} is carried out in the following section.

\begin{remark}\label{rem: expl char} The condition on $\chi$ may be described more explicitly in each type. \\

{\bf Type $A_{n}$.} The character $\chi$ on the diagonal torus $T$ of $\SL(V)$ is given by a collection of characters $\chi_1,\dots, \chi_{n}$ of $k^{\times}$ modulo simultaneous multiplication by the same character of $k^{\times}$. Any maximal Levi subgroup is isomorphic to $S(\GL_a\times \GL_b)$ with $a+b=n$ and $a,b>0$. Its center is the subtorus given by the image of the embedding $\Gm\incl T$, $z\mapsto (z^{b'},\cdots, z^{b'}, z^{-a'},\cdots, z^{-a'})$, where $a'=a/\gcd(a,b), b'=b/\gcd(a,b)$, and $z^{b'}$ appears $a$ times, $z^{-a'}$ appears $b$ times. We therefore require for any non-empty subset $I\subset \{1,\dots, n\}$ of cardinality $a$ and with non-empty complement $J$ of cardinality $b$ that
\[(\prod_{i\in I} \chi_i)^{b/\gcd(a,b)} \neq (\prod_{j\in J} \chi_j)^{a/\gcd(a,b)}.\]

{\bf Types $B_{n},C_{n},D_{n}$.} Identify $T\cong \Gm^{n}$ in the usual way, and write $\chi=(\chi_{1},\cdots, \chi_{n})$. The maximal Levi subgroups are of the form 
\[\GL_a \times G'\]
where $G'$ is a classical group of rank $n-a$ of the same type as $G$ (in the case of type $D_{n}$, $a\ne n-1$). The connected centers of maximal Levi subgroups are the images of maps $\Gm\to \Gm^{n}$,  $z\mapsto (\ph_{1}(z),\cdots, \ph_{n}(z))$ where $\ph_{i}(z)$ is either $1$ or $z$ or $z^{-1}$. Therefore the condition on $\chi$ is that, for  any  disjoint subsets $I\coprod J\subset \{1,2,\cdots, n\}$ such that $I\cup J\ne\vn$, we have
\begin{equation*}
\prod_{i\in I}\chi_{i}\ne \prod_{j\in J}\chi_{j}.
\end{equation*}
Here, when $I$ or $J$ is empty, the corresponding product is $1$. 


\end{remark}

\subsection{Stabilizers on Hessenberg varieties}
Recall that $V_{w}^{\perp}=\bigoplus_{\langle \alpha,wx_{\bQ} \rangle < 1} \frg_\alpha$ which we will denote by $\frg_w$. Then 
\[Y_w=\{gQ_w\in G/Q_w \mid \psi \in \,^g \frg_w \}.\]
By definition these spaces are Hessenberg varieties as defined for example in \cite{DMPS92}. The subvector space $\frg_w$ of $\frg$ is automatically a Hessenberg space, i.e. it is stable under the adjoint action of the parabolic subgroup $Q_w$ and it contains its Lie algebra $\frq_w$. For classical groups, Hessenberg varieties may be described concretely in terms of (isotropic) flags. \par 
Let $V$ be a finite-dimensional vector space over $k$ and let $G$ be $\SL(V), \Spin(V)$ or $\Sp(V)$ where we endow $V$ with a symmetric bilinear (resp. symplectic) form $\langle -,-\rangle$. Then $\psi\in \frg$ is an anti-self adjoint endomorphism of the vector space $V$ and the condition $\psi \in \,^g \frg_w$ may be translated into the condition that if 
\[0\subset F_1 \subset \dots \subset F_r \subset V\]
 is the flag corresponding to $g$ we have $\psi(F_i)\subset F_{h(i)}$ for a non-decreasing function $h$ satisfying $h(i)\ge i$ associated to the space $\frg_w$, cf. \cite[\S 2]{Ty06}. \par
In this section we study the stabilizers of a Borel subgroup $B_\psi$ of $G_\psi$ acting on the Hessenberg variety $Y_w$. The goal is prove that for $w\neq \id$ the Hessenberg variety $Y_w$ is a spectrally meager $G_\psi$-variety.
\begin{lemma} \label{lem: large stabs}
Let $Q \subset G$ be a parabolic subgroup with corresponding Lie algebra $\frq$ and $U\subset \frg$ a subspace containing $\frq$ and which is stable under the adjoint action of $Q$. Assume that there is a parabolic subgroup $P$ of $G$ with Lie algebra $\frp$  such that $Q\subset P$ and $U\subset \frp$. Let $\psi \in U$ be semisimple and let $G_\psi$ be the centralizer of $\psi$ in $G$.  Then the $G_\psi$-variety $Y_\psi(Q,U)=\{g\in G/Q \mid \psi \in \,^g U\}$ is spectrally meager. 
\end{lemma}
\begin{proof}Let $B_\psi\subset G_\psi$ be a Borel subgroup. Denote by $e\in G$ the identity element. It suffices to prove that the stabilizer of $eQ$ contains a non-trivial torus. Indeed, if $gQ\in Y_{\psi}(Q,U)$ we have $\psi \in {}^gU\subset {}^g\frp$ and the assumptions of the Lemma are satisfied for ${}^gQ, {}^gU$ and ${}^gP$. Since $\Stab_{B_{\psi}}(gQ)=B_{\psi}\cap {}^gQ= \Stab_{B_{\psi}}(e\cdot {}^gQ)$ we may conclude.  

We will argue in two steps. Let $M$ be the Levi quotient of $P$ and $\frakm$ its Lie algebra. Consider the map 
\[\pi:Y_\psi(Q,U)\rightarrow Y_\psi(P,\frp) \]
induced by the projection $G/Q\rightarrow G/P$ whose fiber above $eP$ is identified with $M/Q_M$ where $Q_M$ is the image of $Q$ in $M$. Since $\psi$ is semisimple and $\psi\in \frp$, the stabilizer $P_\psi$ of $\psi$ in $P$ is a parabolic subgroup of $G_\psi$ and its Levi quotient $P_\psi/U_\psi$ coincides with the stabilizer $M_{\bar{\psi}}$ of the image of $\psi$ in $\frakm$. Write $Y'=\pi^{-1}(eP)\subset  Y_\psi(Q,U)$. The parabolic $P_\psi$ acts on $Y'$ and the action factors through $M_{\bar{\psi}}$. Since $Y'\subset M/Q_M$ we know that the center $Z(M)$ stabilizes every point of $Y'$. Thus $H:=G_\psi\cap Q$ contains a group of the form $Z\cdot U_\psi$ where $Z$ is a torus surjecting onto $Z(M)$. \par
Let $B_\psi\subset G_\psi$ be a Borel subgroup. The second step is to analyze the action of $H$ on $G_\psi/B_\psi$. Using the Bruhat decomposition we write 
\[G_\psi/B_\psi=\coprod_{w\in W_\psi\backslash W} P_\psi wB_\psi/B_\psi \]
where $W_\psi$ is the Weyl group of $M_{\bar{\psi}}$. Since $H$ acts on each cell we only need to consider one such cell $P_\psi wB_\psi /B_\psi$. Note that $P_\psi wB_\psi /B_\psi\cong P_\psi/ (P_\psi \cap \,^w B_\psi)$ and $\Stab_H(pwB_\psi)=\Stab_H(pP_\psi\cap\,^wB_\psi)$. Let $\bar{H}=H/U_\psi$. This acts on $M_{\bar{\psi}}$ and we have that
\[\Stab_H(p(P_\psi\cap\,^wB_\psi))\rightarrow\Stab_{\bar{H}}(\bar{p}) \]
is surjective (for the action of $\bar{H}$ on $M_{\bar{\psi}} / \bar{B}$). Here $\bar{B}$ denotes the image of $P_\psi\cap\,^wB$ in $M_{\bar{\psi}}$. Finally, since $M_{\bar{\psi}} / \bar{B}$ is the full flag variety of $M_{\bar{\psi}}$ every point is stabilized by the center $Z(M)\subset \bar{H}$ of $M$. 
\end{proof}
\begin{cor} Let $Q, U$ and $\psi$ be as in the previous Lemma. Assume that there is a maximal parabolic $P$ with Lie algebra $\frp$ such that $Q'=Q\cap P$ is a parabolic subgroup of $G$ and $\psi\in \frp$. Then $Y_\psi(Q,U)$ is spectrally meager for the $G_\psi$-action. 
\end{cor}
\begin{proof}
Define $U'=U\cap \frp$. Note that since $\psi\in \frp$, actually $eQ\in Y_\psi(Q,U')$. We have a surjective morphism $Y_\psi(Q',U')\rightarrow Y_\psi(Q,U')$. For any point $y\in Y_\psi(Q',U')$ denote by $\bar{y}$ the image under this morphism. We have that $\Stab_{B_\psi}(y)\subset \Stab_{B_\psi}(\bar{y})$, so it suffices to prove the claim for $eQ'\in Y_\psi(Q',U')$. This immediately follows from the previous Lemma \ref{lem: large stabs}. 
\end{proof}
\begin{remark} In particular for flag varieties of classical groups to prove that $Y_w$ is spectrally meager it suffices to prove that for any (isotropic) flag $F_\bullet\in Y_w$ we can refine it to a flag such that one of the (isotropic) subspaces in the resulting flag is $\psi$-stable. From the proof of Lemma \ref{lem: large stabs} it follows that in this case the stabilizer of any point in $Y_w$ contains the center of the Levi subgroup of the maximal parabolic subgroup stabilizing the $\psi$-stable space. \end{remark}
Denote by $\Phi$ the set of roots of $G$ and let $\Delta=\{\alpha_1,\dots,\alpha_r\}$ be the positive simple roots. By finite Weyl group symmetry we may and will assume that $wx_{\bQ}$ is dominant. 
\begin{cor} Assume that there is a simple root $\alpha_i$ such that $\langle \alpha,wx_{\bQ} \rangle \ge 1$. Then $Y_w$ is spectrally meager. \end{cor}
\begin{proof} Let $P$ be the maximal proper parabolic subgroup of $G$ containing $Q_{w}$ but not the root subgroup corresponding to $\alpha_i$. If $\frg_w \not\subset \frp=\Lie P$, then $\frg_{\alpha_{i}} \subset \frg_{w}$ and $\langle \alpha_{i},wx_{\bQ} \rangle < 1$. Vice versa if $\langle \alpha_{i},wx_{\bQ}\rangle \ge 1$, then $\frg_w \subset \frp$. Therefore Lemma \ref{lem: large stabs} proves the claim.
\end{proof}
This shows that we can restrict ourselves to the study of $Y_w$ for $w$ such that $0\le \langle \alpha_i,wx_{\bQ}\rangle < 1$ for all simple roots $\alpha_i \in \Delta$.

\section{Detailed analysis of stabilizers}\label{s: analyze stabs}

In the following we will carry out a case-by-case analysis of the cases with an almost free open orbit listed in \S \ref{s: hyperspecial}. We will show that any flag in any Hessenberg variety $Y_w$ for $w\neq \id$ may be refined to a flag s.th. one of the spaces appearing in it is $\psi$-stable. This will prove the first part of Theorem \ref{thm: hypersp quasi-epi}. Strictness in types $A$ and $D$ is proved in Section \ref{s: unit coset}. \par

\begin{remark} Let $V$ be a vector space over $k$. In types $B$, $C$ and $D$ we will denote the bilinear form by $\langle-,-\rangle$. We will often make use of the fact that $\frg$ is the Lie algebra of endomorphisms of $V$ which are anti-selfadjoint with respect to the given bilinear form. In particular the non-zero eigenvalues of $\psi$ occur in pairs $x,-x$ and the eigenspaces of $x$ and $-x$ are dual with respect to the given bilinear form. For any non-zero $x$ the eigenspace $V(x)$ is isotropic and orthogonal to all other eigenspaces apart from $V(-x)$. In types $B$,$C$ and $D$ we will always denote the eigenvalues by $\pm x$ (and $\pm y$ in a single case in type $D$). Since $\psi$ is semisimple we can decompose any vector $v$ according to the eigenvalues of $\psi$ and if $x$ is an eigenvalue, $v_x$ is the summand of $v$ in the $x$-eigenspace.
When working with flags $F_\bullet$ of a vector space $V$ we will denote flags and the condition on them imposed by $\psi$ in the following way
\[\begin{tikzcd}[column sep=.7em]
0 \arrow[r,symbol=\subset]  &\dots \arrow[r,symbol=\subset]  &F_k \ar[bend left=50]{r} \arrow[r,symbol=\subset]& F_l \ar[bend left=50]{r} \arrow[r,symbol=\subset]& F_m \arrow[r,symbol=\subset] & \dots  \arrow[r,symbol=\subset] & V
\end{tikzcd}\]
By this we mean that $\psi(F_k)\subset F_l$ and the indices denote the dimensions of the spaces.

\end{remark}

\subsection{Type $A_{n-1}$} Let $V$ be a finite-dimensional $k$-vector space and $\GG=\SL(V)$. We identify $\xcoch(S)\cong \{k=(k_i)\in \ZZ^n \mid \sum_{i=1}^n k_i=0\}$ 
such that for the standard basis $e_1^*, \dots, e_n^*$ of $\ZZ^n$ we have $\alpha^\vee_i=e_{i}^*-e_{i+1}^*$ for the simple roots $\alpha_i$. 
We write $wx_{\bQ}=(x_1,\dots , x_n)$ in coordinates for that basis. Recall that
\[Y_w=\{ gQ_w\in G/Q_w \mid \psi\in \,^g \frg_w \}. \]
The parabolic $Q_w$ is determined by the roots that evaluate non-positively on $wx_{\bQ}$. The space $\frg_{w}$ contains the root space $\frg_{\alpha}$ for any root $\alpha$ if and only if $\langle \alpha, wx_{\bQ} \rangle < 1$. Let $Q’$ be the opposite parabolic to $Q$, so that $-x_{\bQ} = x_{\bQ’}$. To describe the Hessenberg varieties explicitly in terms of flags we will instead work with the variety $Y_{-w}$ defined by the point $-wx_{\bQ}$. This will only switch around the type of $Q$. \par
Similarly we define $\frg_{-w}$ so that $\frg_{\alpha} \subset \frg_{-w}$ if and only if and only if $\langle \alpha, wx_{\bQ} \rangle > -1$. Write $\alpha_{ij}=\sum_{k=i}^{j-1} \alpha_k$ for and $\alpha_{ji}=-\alpha_{ij}$ for $i<j$. We have $\langle \alpha_{ji}, wx_{\bQ} \rangle = x_j-x_i$ and $\frg_{\alpha_{ji}}\subset \frg_{-w}$ if and only if 
\[\langle \alpha_{ji}, wx_{\bQ} \rangle > -1 \]
or in other words if and only if $x_i-x_j < 1$. In terms of $\psi$ and a flag 
\[0\subset F_k \subset \dots \subset F_r \subset V \]
the condition  $\frg_{\alpha_{ji}}\subset \frg_{-w}$ is satisfied if and only if $\psi(F_i)\subset F_j$. We can therefore read off the condition imposed on any flag $F_\bullet \in Y_w$ from the coordinates of $wx_{\bQ}$. The number of different entries in $wx_{\bQ}$ determines the number of different spaces in the flag. The numbers of entries that coincide determine the dimensions of the associated graded spaces. We label the associated graded by the entries of $wx_{\bQ}$. For example if $n=5$ and $wx_{\bQ}=(x_1,x_2,x_2,x_3,x_4)$ we have
\[0\subset_{x_1} F_1 \subset_{x_2} F_3 \subset_{x_3} F_4 \subset_{x_4} V  \]
and $\psi$ is allowed to map $F_1$ to any space whose associated graded is labelled by a number $y$ such that $x_1-y<1$.  \par
The Weyl group acts on $\xcoch(\TT)\otimes \QQ$  by permutation of the coordinates and the coroot lattice by integer translations with vectors  $(k_1,\dots,k_n)$ such that $\sum_{i=1}^n k_i=0$. The arguments in each case work for any parabolic subgroup $Q$ corresponding to the partition $\lambda_Q$. Therefore we only argue for one of these.

\subsubsection{Case (1), \ref{orbitA}}

The barycenter of $\bQ_0$ is $x_{\bQ}=(\frac{n-1}{2n}, -\frac{1}{2n},\dots,-\frac{1}{2n})$. If $wx_{\bQ}\neq x_{\bQ}$ we find that $wx_{\bQ}$ has at least four different entries. Mod $\ZZ$ only one entry may be congruent to $\frac{n-1}{2n}$, all others are congruent to $-\frac{1}{2n}$. Four of them being pairwise different forces two consecutive entries to have a difference of $1$. Therefore the corresponding space in the flag has to be $\psi$-stable.

\subsubsection{Case (2), \ref{orbitA}}

The barycenter is $x_{\bQ}=(\frac{n-1}{2n}, \frac{n-3}{2n}, \dots, \frac{3-n}{2n}, \frac{1-n}{2n})$. Let $wx_{\bQ}=(x_1,\dots, x_n)$ be dominant and not equal to $x_{\bQ}$. Assume that there is an index $i_0$ such that $x_{i_0}-x_{i_0+1}\ge 2/n$. Then we have
\[x_1-x_n = \sum_{i=1}^{n-1}x_i-x_{i+1}\ge (n-2)\frac{1}{n}+\frac{2}{n}=1.\]
Since any flag $F_\bullet\in Y_w$ is a full flag this implies that $\psi(F_1)\subset F_j$ for some $j\le n-1$. Let $v\in F_1\setminus\{0\}$ and write $v=v_x+v_y$ according to the eigenvalues of $\psi$. Assume the $x$-eigenspace $V(x)$ is the one-dimensional eigenspace. If $v_x=0$, $F_1$ is $\psi$-stable and we are done. If $v_x\neq 0$, $V(x)\subset F_1+\psi(F_1)\subset F_j$ and $F_j$ is $\psi$-stable. \par
We may therefore assume that $x_i-x_{i+1}\le 1/n$. Write $x_i=q_i+k_i$ with $q_i\in \{\frac{n-1}{2n}, \frac{n-3}{2n}, \dots, \frac{3-n}{2n}, \frac{1-n}{2n}\}$ and $k_i\in \ZZ$. We have $\sum_{i=1}^n k_i=0$. Since
\[|q_i-q_{i+1}|\le \frac{n-1}{n} \]
we have the following three cases
\begin{align*}
q_i-q_{i+1}=-\frac{n-1}{n}, k_i-k_{i+1}=1, \\
q_i-q_{i+1}=\frac{n-1}{n}, k_i-k_{i+1}=-1, \\
q_i-q_{i+1}=\frac{1}{n}, k_i-k_{i+1}=0. \\
\end{align*}
Now $q_i-q_{i+1}=-\frac{n-1}{n}$ implies that $q_i=\frac{1-n}{2n}$ and $q_{i+1}=\frac{n-1}{2n}$ and vice versa in the second case. In particular, since the $q_i$ are pairwise different, only one of the first two cases may appear and it can only happen for a single index. This implies that the integer vector $(k_1, \dots, k_n)$ has only two distinct entries, but this contradicts the assumption $\sum_{i=1}^n k_i=0$. Therefore the only case in which  $x_i-x_{i+1}\le 1/n$ is the case in which $wx_{\bQ}=x_{\bQ}$.

\subsubsection{Case (3), \ref{orbitA}}
We have $x_{\bQ}=(\frac{m}{2n}, \dots, \frac{m}{2n}, -\frac{m}{2n}, \dots, -\frac{m}{2n})$ with both entries occuring $m$-times. If $wx_{\bQ}\neq x_{\bQ}$ and $wx_{\bQ}=(x_1,\dots, x_n)$ is dominant, then it has at least four different entries. The condition  $0\le \langle \alpha_i, wx_{\bQ}\rangle \le 1/2$ implies that any flag in $Y_w$ contains a part of the form
\[\begin{tikzcd}[column sep=.7em]
\dots \arrow[r,symbol=\subset]  &F_1 \ar[bend left=50]{r} \arrow[r,symbol=\subset]& F_2 \ar[bend left=50]{r} \arrow[r,symbol=\subset]& F_3 \arrow[r,symbol=\subset] & \dots.
\end{tikzcd}\]
Then $W=F_2+\psi(F_2)+\psi^2(F_1)$ is $\psi$-stable. Indeed, let $x,y,z$ be the eigenvalues of $\psi$ with eigenspaces $V(x),V(y),V(z)$ of dimensions $m, m-1,1$ respectively. If every vector $V\in F_1$ is of the form $v=v_x+v_y$, then $F_1+\psi(F_1)$ is $\psi$-stable. Otherwise $W\supset V(z)$, so for any $w=w_x+w_y+w_z\in F_2$ we get $w_x+w_y\in W$ and $xw_x+yw_y\in W$, i.e. $w_x, w_y\in W$, proving the claim. 
\subsubsection{Case (4), \ref{orbitA}}
Here we have 
\[x_{\bQ}=\left(\frac{m+1}{3n}, \dots, \frac{m+1}{3n}, -\frac{m-1}{3n}, \dots, -\frac{m-1}{3n}, - \frac{3m-1}{3n} \right)\]
 with entries occurring $m$-times, $m-1$ times and one time in this order. We may assume that for $x_{\bQ}\neq wx_{\bQ}=(x_1,\dots, x_n)$ we have $0\le x_i-x_{i+1} \le 1/3$, because otherwise we can insert a space of the form $F+\psi(F)$. \par
In addition we know that $wx_{\bQ}$ must have at least five different entries, say $y_1,\dots y_5$ and their classes mod $\ZZ$ need to be ordered as follows
\[\left(\frac{m+1}{3n}, -\frac{m-1}{3n}, -\frac{3m-1}{3n}, \frac{m+1}{3n}, -\frac{m-1}{3n}\right),\]
since $-\frac{3m-1}{3n}$ may only appear once. Write $y_1=\frac{m+1}{3n}+k_1$ and so on. The condition $0\le y_i-y_{i+1} \le 1/3$ implies that $k_1=k_2=k_3$ and $k_4=k_5=k_3-1$. Since all other entries are congruent to $\frac{m+1}{3n}$ or $-\frac{m-1}{3n}$ mod $\ZZ$ we conclude that the integer vector by which we translate has the shape 
\[(k,\dots, k, k-1,\dots, k-1), \]
a contradiction to the sum of these entries vanishing. 
\subsubsection{Case (5) \& (6), \ref{orbitA}}
The arguments in these cases are the same as in the previous two. 

\subsubsection{Case (7), \ref{orbitA}}
The barycenter is $x_{\bQ}=(1/3,1/3,0,0,-1/3,-1/3)$. We may assume that for $wx_{\bQ}=(x_1,\dots,x_6)$ we have $0\le x_i-x_{i+1} \le 1/3$. Otherwise there is an index $i$ such that $x_i-x_{i+2}=1$ and $\psi(F_i)\subset F_{i+1}$. In that case since $\psi$ has only two eigenvalues we may insert the stable space $F_i+\psi(F_i)$. \par
Now assume that all entries of $wx_{\bQ}=(x_1,\dots,x_6)$ are pairwise different. In that case the classes of the entries must be ordered as follows 
\[\left(\frac{2}{3},\frac{1}{3},0,\frac{2}{3},\frac{1}{3},0\right),\]
because otherwise we will find two successive entries whose difference is at least $2/3$. If $x_1=2/3+k_1$ and so on where $k_i\in \ZZ$ then we find that $k_1=k_2=\dots = k_6$. This contradicts the condition $\sum_{i=1}^6 k_i=0$. \par
We are therefore reduced to the following situation. 
 For $wx_{\bQ}\neq x_{\bQ}$ we always have exactly four different proper subspaces and $\psi$ maps as follows 
\[\begin{tikzcd}[column sep=.7em]
\dots \arrow[r,symbol=\subset]  &F_1 \ar[bend left=50]{rr} \arrow[r,symbol=\subset]& F_2 \ar[bend left=50]{rr} \arrow[r,symbol=\subset]& F_3 \arrow[r,symbol=\subset]& F_4 \arrow[r,symbol=\subset] & \dots.
\end{tikzcd}\]
We are left to consider two cases, either $\dim(F_1)=1$ or $\dim(F_1)=2$. \par
In the first case we assume that $F_1$ is not already stable. Choose a non-zero vector $v\in F_1$ and extend it to a basis $v,w$ of $F_2$. Write $v=v_x+v_y$ according to the eigenvalues of $\psi$ (where the $y$-eigenspace $V(y)$ has dimension $2$) and similarly for $w$. The space $F_2+\psi (F_2)$ is spanned by $v_x, w_x, v_y$ and $w_y$. If $v_y$ and $w_y$ are linearly independent, $V(y)\subset F_2+\psi (F_2) \subset F_4$ and hence $F_4$ is stable. If they are linearly dependent, $F_2+\psi(F_1)$ is $\psi$-stable and we may insert that. \par
In the second case choose a basis $v,w$ of $F_2$. Then with the same notation as before $F_2+\psi F_2$ contains $v_x, w_x, v_y$ and $w_y$. If $v_y$ and $w_y$ are linearly dependent $F_2$ contains a non-zero eigenvector and we may insert the line spanned by it. If not, $F_2+\psi(F_2)$ contains the whole $y$-eigenspace $V(y)$ and we are done. 

\subsubsection{Case (8), \ref{orbitA}}
The barycenter is $x_{\bQ}=(1/6, 1/6, 1/6, 1/6, -1/3, -1/3)$. For $wx_{\bQ}\neq x_{\bQ}$ we will find at least four different entries $x_1,\dots, x_4$. We may assume as before that $0 \le x_i-x_{i+1} \le 1/2$, because otherwise a space in the flag will already be $\psi$-stable. Therefore $wx_{\bQ}$ will have exactly four different entries, say $y_1, y_2, y_3, y_4$ and the classes mod $\ZZ$ will be ordered either as $(-1/3, 1/6, -1/3, 1/6)$ or $(1/6, -1/3, 1/6, -1/3)$. The associated graded spaces that are labelled by $-1/3$ have to be one-dimensional and cannot occur consecutively. Therefore we get the following list of possible flags that we need to consider:
\begin{align*}
0\subset F_1\subset F_2\subset F_3 \subset V, \\ 
0 \subset F_2\subset F_3 \subset F_5 \subset V, \\
0 \subset F_3\subset F_4\subset F_5 \subset V, \\
0 \subset F_1\subset F_4 \subset F_5 \subset V, \\
0 \subset F_1\subset F_3 \subset F_4 \subset V. \\
\end{align*} 
The indices indicate the dimensions of the spaces and the endomorphism $\psi$ always maps a space to the next one. In the first three cases we may insert the spaces $F_1+\psi(F_1)+\psi^2(F_1)$, $F_2+\psi(F_2)+\psi^2(F_2)$ and $F_3+\psi(F_3)+\psi^2(F_3)$ respectively. \par
The last two cases are similar to each other. We will present the argument only for the fourth case. First note that we may assume that $F_4+\psi(F_4)=F_5$ since otherwise $F_4$ is already $\psi$-stable. Thus $F_1+\psi(F_1)+\psi^2(F_1)\subset F_4+\psi(F_4)$ and for $v=v_x+v_y+v_z\in F_1\setminus\{0\}$ we get $v_x, v_y, v_z\in F_4+\psi(F_4)$. We may assume that they are all non-zero because otherwise $F_1+\psi(F_1)$ is $\psi$-stable. Since $F_4\cap \psi(F_4)\neq 0$ there is a vector $w\in F_4$ such that $\psi(w)\in F_4$. Thus $F_4+\psi(F_4)$ contains $w_x,w_y,w_z$. If any pair $(w_x,v_x)$, $(w_y, v_y)$ or $(w_z, v_z)$ is linearly independent, the corresponding eigenspace lies in $F_4+\psi(F_4)$ and this would imply that $F_4+\psi(F_4)$ is $\psi$-stable. \par
We therefore assume that the above pairs are all linearly dependent. Since $w\neq 0$ we may additionally assume that $w_x\neq 0$ and hence we can write $v_x=\lambda w_x$ for some non-zero $\lambda$. Since $F_4\supset F_1+\psi(F_1)$ it contains $(x-y)v_x+(z-y)v_z$ and also $(x-y)w_x+(z-y)w_z$. In particular $\lambda w_z-v_z\in F_4$ and therefore if $ \lambda w_z-v_z\neq 0$ we get $v_z\in F_4$ and $F_1+\psi(F_1)+\psi^2(F_1)\in F_4$. If $ \lambda w_z-v_z=0$ we can do the same for $v_y$ and find that actually $v=\lambda w$. Since $w\in F_4\cap \psi(F_4)$ is arbitrary this implies $\dim(F_4\cap \psi(F_4))=1$ and hence $\ker(\psi)\cap F_4\neq 0$, i.e. say $z=0$ and $F_4$ contains a $0$-eigenvector $u_0$. Again if $u_0$ and $v_0$ are linearly independent, then $V(0)\subset F_4+\psi(F_4)$ and if they are dependent, then $F_1+\psi(F_1)+\psi^2(F_1)\subset F_4$, so we are done. 

\subsection{Type $B_n$} 
 Let $V$ be a finite dimensional $k$-vector space of odd dimension equipped with a non-degenerate symmetric bilinear form $\langle -,-\rangle$ and $\GG=\Spin(V)$. As in type $A$ we use the evident identification $\xch(\TT)\cong \ZZ^n$ to write $wx_{\bQ}\in \xcoch(\TT)$ in coordinates. A (partial) flag in type $B$ stabilized by a parabolic $Q_w$ can be thought of as a flag of isotropic spaces together with their complements
\[0\subset F_1 \subset \dots \subset F_{n-1} \subset F_n \subset F_{n-1}^\perp \subset F_1^\perp\subset V. \]
The condition that $\psi \in \frg_w$ imposes the same conditions as in type $A$ for the flag
\[0\subset F_1 \subset \dots \subset F_{n-1} \subset F_n.\]
In addition, mapping from a space $F_i$ to $F_j^\perp$ is determined by the value of the root $e_i+e_j$ (or $e_i$ for $j=i$) and hence for $wx_{\bQ}=(x_1,\dots,x_n)$ we may label the flag above as follows
\[0\subset_{x_1} F_1 \subset \dots \subset_{x_{n}} F_{n} \subset_{0} F_n^\perp \subset_{-x_n} F_{n-1}^\perp \subset F_1^\perp\subset_{-x_1} V. \]
Again $\psi$ is allowed to map $F_i$ to any space whose associated graded is labelled by $y$ such that $x_i-y<1$. \par
The finite Weyl group acts on $\xcoch(\TT)\otimes \QQ$ by arbitrary permutations and sign changes of the coordinates. We may translate using the coroot lattice, i.e. by integer vectors whose coordinates sum to an even number. We can therefore produce a finite list of possible cases for $wx_{\bQ}$ (using the condition that $0\le \langle \alpha_i,wx_{\bQ}\rangle < 1$ for all simple roots).
\subsubsection{Case (1), \ref{orbitB}, $(P_\psi,Q)\sim (P_n,P_n)$} In this case the barycenter is $x_{\bQ}=(1/2,\dots , 1/2)$. Therefore $wx_{\bQ}$ will have coordinates in $1/2+\ZZ$. The condition $0\le \langle \alpha_i,wx_{\bQ}\rangle < 1$ immediately implies that all coordinates of $wx_{\bQ}$ have to be equal to $1/2$, i.e. $wx_{\bQ}=x_{\bQ}$ is the only possibility and there is nothing to prove. 

\subsubsection{Case (2), \ref{orbitB}}
This case is similar to Case (1). Since $x_{\bQ}=(1/2,1/2)$, there is no other possibility than $wx_{\bQ}=x_{\bQ}$. 

\subsubsection{Case (3), \ref{orbitB}}
We have $x_{\bQ}=(1/2,0)$ and the only non-trivial possibility for $wx_{\bQ}$ is $wx_{\bQ}=(1,1/2)$. In this case the flag is
\[\begin{tikzcd}[column sep=.7em]
0 \arrow[r,symbol=\subset]  &F_1 \ar[bend left=70]{r} \arrow[r,symbol=\subset]& F_2 \ar[bend left=70]{r} \arrow[r,symbol=\subset]& F_2^\perp \ar[bend left=70]{r} \arrow[r,symbol=\subset]& F_1^\perp \arrow[r,symbol=\subset] & V.
\end{tikzcd}\]
and we may insert $F_1+\psi(F_1)$.

\subsubsection{Case (4), \ref{orbitB}} We have $x_{\bQ}=(1/2,1/2,0)$. The possible non-trivial cases for $wx_{\bQ}$ are
\vspace{1mm}
\begin{enumerate}[nosep, label=(\roman*)] 
\item $wx_{\bQ}=(1,1/2,1/2)$,
\item $wx_{\bQ}=(2,3/2,1/2)$.
\end{enumerate}
In case $(i)$ the corresponding flag is
\[\begin{tikzcd}[column sep=.7em]
0 \arrow[r,symbol=\subset]  &F_1 \ar[bend left=70]{r} \arrow[r,symbol=\subset]& F_3 \ar[bend left=70]{r} \arrow[r,symbol=\subset]& F_3^\perp \ar[bend left=70]{r} \arrow[r,symbol=\subset]& F_1^\perp \arrow[r,symbol=\subset] & V.
\end{tikzcd}\]
Choose a non-zero vector $v\in F_1$. Denote the non-zero eigenvalues of $\psi$ by $x$ and $-x$ and write $v=v_x+v_0+v_{-x}$ according to the eigenspace decomposition of $V$. We get that 
\[\langle v, \psi^2 v\rangle =0 \]
and from that it follows that $\langle v_x,v_{-x}\rangle =0$. Since $P_\psi\sim P_1$, $\dim V(x)=\dim V(-x)=1$ and either $v_x=0$ or $v_{-x}=0$. Assume that $v_{-x}=0$. In that case $W=F_1+\psi F_1$ contains $v=v_0+v_x$ and $v_x$. This space therefore has a basis of eigenvectors and is $\psi$-stable. Since $W\subset F_3$ it is automatically isotropic and we may insert this space into the given flag.
 \\
In case $(ii)$ the flag is 
\[\begin{tikzcd}[column sep=.7em]
0 \arrow[r,symbol=\subset]  &F_1 \ar[bend left=70]{r} \arrow[r,symbol=\subset]& F_2 \ar[bend left=70]{r} \arrow[r,symbol=\subset]& F_3 \ar[bend left=70]{r} \arrow[r,symbol=\subset]&F_3^\perp \ar[bend left=70]{r} \arrow[r,symbol=\subset]&  F_2^\perp \ar[bend left=70]{r} \arrow[r,symbol=\subset]& F_1^\perp \arrow[r,symbol=\subset] & V.
\end{tikzcd}\]
Because $P_\psi\sim P_1$, $\psi$ has three eigenvalues and if neither $F_1$ nor $F_2$ are $\psi$-stable, then $F_3=F_1+\psi F_1 +\psi^2F_1$ which is stable. 
\subsubsection{Case (5), \ref{orbitB}}
We have $x_{\bQ}=(1/2,0,0)$. The possible non-trivial cases for $wx_{\bQ}$ are
\vspace{1mm}
\begin{enumerate}[nosep, label=(\roman*)] 
\item $wx_{\bQ}=(1,1/2,0)$,
\item $wx_{\bQ}=(1,1,1/2)$.
\end{enumerate}
In case $(i)$ the flag is
\[\begin{tikzcd}[column sep=.7em]
0 \arrow[r,symbol=\subset]  &F_1 \ar[bend left=70]{r} \arrow[r,symbol=\subset]& F_2 \ar[bend left=70]{r} \arrow[r,symbol=\subset]& F_2^\perp \ar[bend left=70]{r} \arrow[r,symbol=\subset]& F_1^\perp \arrow[r,symbol=\subset] & V.
\end{tikzcd}\]
Let $v\in F_1$ be non-zero and write $v=v_x+v_0+v_{-x}$ according to the eigenspace decomposition of $V$. Since $\psi(v)\in F_2$, it's isotropic. We have $\psi(v)=xv_x-xv_{-x}$ and 
\[0=\langle \psi(v),\psi(v) \rangle=-2x\langle v_x,v_{-x} \rangle.\]
Therefore $\langle v_x,v_{-x} \rangle=0$. The space $W=F_1+\psi F_1+\psi^2 F_1$ is $\psi$-stable and we claim it is isotropic. Since $\psi(v)\in F_2$ and $\psi^2(v)\in F_2^\perp$ it is enough to prove that $\langle \psi^2(w), \psi^2(w')\rangle=0$ for any $w, w'\in F_1$. But $F_1$ is a line, so $w=\lambda v$ and $w'=\mu v$ for some $\lambda, \mu\in k$. Therefore 
\[\langle \psi^2(w), \psi^2(w')\rangle=\lambda \mu \langle x^2v_x+x^2v_{-x}, x^2v_x+x^2v_{-x} \rangle =2x^4\langle v_x, v_{-x} \rangle =0 \]
and we can insert $W$ into the above flag.  \\
In case $(ii)$ the flag is 
\[\begin{tikzcd}[column sep=.7em]
0 \arrow[r,symbol=\subset]  &F_2 \ar[bend left=70]{r} \arrow[r,symbol=\subset]& F_3 \ar[bend left=70]{r} \arrow[r,symbol=\subset]& F_3^\perp \ar[bend left=70]{r} \arrow[r,symbol=\subset]& F_2^\perp \arrow[r,symbol=\subset] & V.
\end{tikzcd}\]
For $v,w\in F_2$, since $\psi^2F_2\subset F_3^\perp\subset F_2^\perp$ we have 
\[0=\langle v, \psi^2(w)\rangle=x^2(\langle v_x,w_{-x} \rangle +\langle v_{-x},w_x\rangle). \]
The space $W=F_2+\psi F_2 + \psi^2 F_2$ is $\psi$-stable. We claim it is isotropic. For this it is enough to show that $\langle \psi^2(v), \psi^2(w) \rangle =0$ for any $v,w\in F_2$. But 
\[\langle \psi^2(v), \psi^2(w) \rangle = \langle x^2(v_x+v_{-x}), x^2(w_x+w_{-x}) \rangle = x^4(\langle v_x,w_{-x} \rangle +\langle v_{-x},w_x\rangle)\]
which vanishes by the previous calculation. Therefore $W$ is isotropic and since $F_3\subset W$ they have to be equal.

\subsection{Type $C_n$}
Let $V$ be a finite dimensional $k$-vector space of even dimension equipped with a symplectic form $\langle -,-\rangle$ and $\GG=\Sp(V)$. As before we may write $wx_{\bQ}$ in coordinates using the identification $\xch(\TT)\cong \ZZ^n$. Similar to type $B$ we consider flags
\[0\subset F_1 \subset \dots \subset F_{n-1} \subset F_n \subset F_{n-1}^\perp \subset F_1^\perp\subset V \]
and label them the same way, but without the middle step. The endomorphism $\psi$ is allowed to map $F_i$ to any space whose associated graded is labelled by $y$ such that $x_i-y<1$. \par
The finite Weyl group acts by arbitrary permutations and sign changes on the coordinates of $wx_{\bQ}$ and the coroot lattice by arbitrary integer vector translations.

\subsubsection{Case (1), \ref{orbitC}} The barycenter is $x_{\bQ}=(1/4,\dots, 1/4)$. Therefore any coordinate of $wx_{\bQ}=(x_1,\dots,x_n)$ lies in $1/4+\ZZ$ or $3/4+\ZZ$. Therefore $x_i-x_j=0$ or $x_i-x_j=1/2$ for any $i \le j$. By the condition $0\le \langle \alpha_n,wx_{\bQ} \rangle < 1$ we conclude that $x_n=1/4$. Therefore if all coordinates agree, $wx_{\bQ}=x_{\bQ}$. For flags of the form
\[0\subset F \subset F' \subset (F')^\perp \subset F^\perp \subset V\]
we find that $\psi(F)\subset F'$ since the difference of any two distinct coordinates is $1/2$. Since $P_\psi\sim P_n$, the space $W=F+\psi(F)$ is stable and can be inserted into the flag. We can use this argument whenever we have at least three distinct coordinates. In the case that we only have two distinct coordinates the flag is of the form
\[0\subset F\subset F^\perp \subset V\]
with $\psi(F)\subset F^\perp$ and we can insert the space $W=F+\psi(F)$. Let $v+\psi(v'), w+\psi(w')\in W$. Then
\[\langle v+\psi(v'), w+\psi(w') \rangle=\langle v,w \rangle + \langle v,\psi(w') \rangle + \langle \psi(v'),w\rangle +\langle \psi(v'),\psi(w') \rangle= \langle \psi(v'),\psi(w') \rangle\]
since $F$ is isotropic and $\psi(F)\subset F^\perp$. Since $\psi$ is anti-self adjoint, $\langle \psi(v'),\psi(w') \rangle=-\langle v',\psi^2(w') \rangle$. Now since $P_\psi\sim P_n$ is the Siegel parabolic, $\psi$ has eigenvalues say $x$ and $-x$ and it follows that $\psi^2$ is scalar multiplication by $x^2$. Therefore 
\[-\langle v',\psi^2(w') \rangle=-x^2\langle v',w' \rangle=0 \]
and we find that $W$ is isotropic.

\subsubsection{Case (2), \ref{orbitC}}
We have $x_{\bQ}=(1/3,1/3,0)$. The possible non-trivial cases for $wx_{\bQ}$ are
\vspace{1mm}
\begin{enumerate}[nosep, label=(\roman*)] 
\item $wx_{\bQ}=(2,1/3,0)$,
\item $wx_{\bQ}=(2/3,2/3,0)$,
\item $wx_{\bQ}=(1,2/3,1/3)$,
\item $wx_{\bQ}=(4/3,2/3,0)$,
\item $wx_{\bQ}=(1,1/3,1/3)$,
\item $wx_{\bQ}=(4/3,1,1/3)$.
\end{enumerate}
In case $(i)$ we have flags
\[\begin{tikzcd}[column sep=.7em]
0 \arrow[r,symbol=\subset]  &F_1 \ar[bend left=50]{rr}  \arrow[r,symbol=\subset]& F_2 \ar[bend left=50]{rr} \arrow[r,symbol=\subset]&  F_2^\perp  \arrow[r,symbol=\subset] &F_1^\perp  \arrow[r,symbol=\subset]& V.
\end{tikzcd}\]
We claim that $W=F_2+\psi(F_1)$ is isotropic  and that if neither $F_1$ nor $F_2$ are $\psi$-stable then $F_2+\psi(F_1)$ is $\psi$-stable. Because $\psi(F_1)\subset F_2$ it suffices to prove that $\langle \psi v, \psi v' \rangle =0 $ for any $v,v'\in F_1$. This is clear since $F_1$ is one-dimensional. \par
To prove $W$ is stable, let $v\in F_1$ be non-zero. We have
\[0=\langle v, \psi v\rangle = -2x\langle v_x, v_{-x} \rangle,\]
i.e. $v_x=0$ or $v_{-x}=0$. Assume we have $v_{-x}=0$ and let $w=w_x+w_0+w_{-x}=F_2$ be any vector. Then 
\[0=\langle v, \psi w \rangle =-x \langle v_x, w_{-x} \rangle. \]
If $v_x=0$, then $F_1$ is $\psi$-stable, so we get that $w_{-x}=0$ and $F_2+\psi(F_1)$ contains $v_x,v_0$ and $w_0$. If $v_0$ and $w_0$ are linearly dependent, then $V(x)\subset F_2$ and by what we said above no vector in $F_2$ has a component in the $(-x)$-eigenspace. This implies that $F_2$ is $\psi$-stable. We may therefore assume that $v_0$ and $w_0$ are linearly independent and hence $F_2+\psi(F_1)$ is spanned by eigenvectors. \par
In case $(ii)$ the flags are of the form 
\[\begin{tikzcd}[column sep=.7em]
0 \arrow[r,symbol=\subset]  &F_2 \ar[bend left=70]{r} \arrow[r,symbol=\subset]& F_2^\perp \arrow[r,symbol=\subset] & V.
\end{tikzcd}\]
Choose a basis $v,w$ for $F_2$.  Since $\dim V(x)=1$ there are scalars $\lambda, \mu$ such that $\lambda v_x+\mu w_x=0$ and not both of them are zero. Let $u=\lambda v+\mu w$. Since $v,w$ are a basis, $u\neq 0$. If $u_{-x}=0$ then $u_0\neq 0$ and hence $F_2$ contains the eigenline spanned by $u_0$. If $u_{-x}\neq 0$ then for any $u'\in F_2$ we have
\[0=\langle u', \psi u \rangle = \langle u_x',u_{-x} \rangle, \]
i.e. $u_x'=0$. In this case $F_2+\psi(F_2)$ is $\psi$-stable and isotropic. \par
In case $(iii)$ we have flags
\[\begin{tikzcd}[column sep=.7em]
0 \arrow[r,symbol=\subset]  &F_1 \arrow[r,symbol=\subset]&F_2 \ar[bend left=70]{r} \arrow[r,symbol=\subset]&F_3 \ar[bend left=70]{r} \arrow[r,symbol=\subset]&F_2^\perp  \arrow[r,symbol=\subset]& F_1^\perp \arrow[r,symbol=\subset] & V.
\end{tikzcd}\]
We claim that $W=F_2+\psi(F_2)+\psi^2(F_2)$ is isotropic. To prove that it suffices to show that $\langle \psi v, \psi^2w \rangle=0$ and that $\langle \psi^2 v, \psi^2 w\rangle =0$ for any $v,w \in F_2$. We have
\[\langle \psi v, \psi^2 w\rangle = \langle \psi v, \psi^2 w\rangle +\langle \psi v, x^2w_0\rangle =x^2 \langle \psi v, w\rangle =0. \]
The same argument works for $\langle \psi^2 v, \psi^2 w\rangle =0$ and thus $W$ is isotropic. If $F_2$ is not $\psi$-stable, then $F_3=F_2+\psi(F_2)$ and $F_3\subset W$. Hence $W=F_3$. \par
In case $(iv)$ the flags are
\[\begin{tikzcd}[column sep=.7em]
0 \arrow[r,symbol=\subset]  &F_1 \ar[bend left=70]{r}  \arrow[r,symbol=\subset]& F_2 \ar[bend left=70]{r} \arrow[r,symbol=\subset]&  F_2^\perp  \arrow[r,symbol=\subset] \ar[bend left=70]{r} &F_1^\perp  \arrow[r,symbol=\subset]& V
\end{tikzcd}\]
and we may take $F_1+\psi(F_1)+\psi^2(F_1)$ which is isotropic with the same argument as in case (iii). \par
In case $(v)$ consider
\[\begin{tikzcd}[column sep=.7em]
0 \arrow[r,symbol=\subset]  &F_1 \ar[bend left=70]{r}  \arrow[r,symbol=\subset]& F_3 \ar[bend left=70]{r} \arrow[r,symbol=\subset] &F_1^\perp  \arrow[r,symbol=\subset]& V.
\end{tikzcd}\]
For $v\in F_1$ non-zero we have $v_x=0$ or $v_{-x}=0$ and hence $F_1+\psi(F_1)$ is $\psi$-stable. \par
For case $(vi)$ we get full flags
\[\begin{tikzcd}[column sep=.7em]
0 \arrow[r,symbol=\subset]    &F_1 \ar[bend left=70]{r}\arrow[r,symbol=\subset]&F_2 \ar[bend left=70]{r} \arrow[r,symbol=\subset]&F_3 \ar[bend left=70]{r} \arrow[r,symbol=\subset]&F_2^\perp \ar[bend left=70]{r} \arrow[r,symbol=\subset]& F_1^\perp \arrow[r,symbol=\subset] & V
\end{tikzcd}\]
and we may simply insert $F_1+\psi(F_1)+\psi^2(F_1)$.

\subsubsection{Case (3), \ref{orbitC}}
We have $x_{\bQ}=(1/3,0,0)$. The possible non-trivial cases for $wx_{\bQ}$ are
\vspace{1mm}
\begin{enumerate}[nosep, label=(\roman*)] 
\item $wx_{\bQ}=(2/3,0,0)$,
\item $wx_{\bQ}=(1,1/3,0)$,
\item $wx_{\bQ}=(1,2/3,0)$,
\item $wx_{\bQ}=(1,1,1/3)$.
\end{enumerate}
In case (i) the flags are of the form 
\[\begin{tikzcd}[column sep=.7em]
0 \arrow[r,symbol=\subset]  &F_1 \ar[bend left=70]{r} \arrow[r,symbol=\subset]& F_1^\perp \arrow[r,symbol=\subset] & V.
\end{tikzcd}\]
Let $v\in F_1$ be non-zero. Then $\langle v, \psi^2 v\rangle=-\langle v,v\rangle =0$, hence $W=F_1+\psi(F_1)+\psi^2(F_1)\subset F_1^\perp$. Since also $\langle v, \psi v\rangle = 0$ we find that $\langle v_x,v_{-x} \rangle = 0$. This implies that $W$ is isotropic.  \par
In case $(ii)$ we have flags
\[\begin{tikzcd}[column sep=.7em]
0 \arrow[r,symbol=\subset]  &F_1 \ar[bend left=70]{r} \arrow[r,symbol=\subset]& F_2 \ar[bend left=40]{rr} \arrow[r,symbol=\subset]& F_2^\perp \arrow[r,symbol=\subset]& F_1^\perp \arrow[r,symbol=\subset] & V.
\end{tikzcd}\]
If $F_1$ is not stable, then $F_2=F_1+\psi(F_1)$. Let $F_1$ be spanned by $v=v_x+v_0+v_{-x}$. As before we find that $\langle v_x, v_{-x}\rangle=0$ and hence $\langle \psi v, \psi^2 v\rangle = 0$. Thus $\psi(F_2)\subset F_2^\perp$ and we may insert $F_1+\psi(F_1)+\psi^2(F_1)$. \par
In case $(iii)$ we consider 
\[\begin{tikzcd}[column sep=.7em]
0 \arrow[r,symbol=\subset]  &F_1  \arrow[r,symbol=\subset]\ar[bend left=70]{r}& F_2 \ar[bend left=70]{r} \arrow[r,symbol=\subset]& F_2^\perp \ar[bend left=70]{r} \arrow[r,symbol=\subset]& F_1^\perp \arrow[r,symbol=\subset] & V.
\end{tikzcd}\]
This works the same as case (ii). \par
In case $(iv)$ we have
\[\begin{tikzcd}[column sep=.7em]
0 \arrow[r,symbol=\subset]  &F_2 \ar[bend left=70]{r}  \arrow[r,symbol=\subset]& F_3 \ar[bend left=70]{r} \arrow[r,symbol=\subset]& F_2^\perp  \arrow[r,symbol=\subset]& V.
\end{tikzcd}\]
We have $F_2\cap \psi(F_2)\neq 0$ and there is a $v\in F_2$ such that $\psi(v)\in F_2$. Therefore $F_2$ is spanned by two vectors of the form $w_x+w_0$ and $u_{-x}+u_0$ and $F_2+\psi(F_2)$ is $\psi$-stable.

\subsection{Type $D_n$}
Let $V$ be a finite dimensional $k$-vector space of even dimension equipped with a non-degenerate symmetric bilinear form $\langle -,-\rangle$ and $\GG=\Spin(V)$. Using the identification $\xch(\TT)\cong \ZZ^n$ we write $wx_{\bQ}$ in coordinates. A full flag in type $D_n$ is the data of a flag of isotropic spaces
\[\begin{tikzcd}[column sep=.7em, row sep=.7em]  & & & & & F_n \arrow[dr,symbol=\subset]  & & & & & \\
0 \arrow[r,symbol=\subset]  &F_1  \arrow[r,symbol=\subset]& F_2 \arrow[r,symbol=\subset]& \dots \arrow[r,symbol=\subset]& F_{n-1}\arrow[ur,symbol=\subset] \arrow[dr,symbol=\subset] & &  F_{n-1}^\perp \arrow[r,symbol=\subset]  & \dots  \arrow[r,symbol=\subset]  & F_2^\perp  \arrow[r,symbol=\subset]  & F_1^\perp  \arrow[r,symbol=\subset] & V. \\
& & & & & F_n' \arrow[ur,symbol=\subset] & & & & &
\end{tikzcd}\]
with two Lagrangian subspaces $F_n$ and $F_n'$ such that $F_n\cap F_n'=F_{n-1}$. The labelling is done as before. We denote the non-zero eigenvalues of $\psi$ by $\pm x$ (and $\pm y$ in the case $n=4$). The finite Weyl group acts by arbitrary permutations and even sign changes on the coordinates of $wx_{\bQ}$. The coroot lattice acts by translation with integer vectors whose coordinates sum to an even integer.  \par

\subsubsection{Case (1), \ref{orbitD}} \label{FlagsD12}
We have $x_{\bQ}=(1/2,1/4,0,0)$. Because $\psi$ has only two eigenvalues we may eliminate cases in the orbit of $x_{\bQ}$ for which $\psi$ maps a space to the next one. This leaves us with the following possibilities.
\begin{enumerate}[nosep, label=(\roman*)] 
\item $wx_{\bQ}=(3/4,1/2,0,0)$,
\item $wx_{\bQ}=(1,1/2,1/4,0)$,
\item $wx_{\bQ}=(1,3/4,1/2,0)$,
\item $wx_{\bQ}=(1,1,1/2,1/4)$. 
\end{enumerate}
In case $(i)$ the flags are of the form 
\[\begin{tikzcd}[column sep=.7em, row sep=.7em]  
0 \arrow[r,symbol=\subset]  &F_1   \arrow[r,symbol=\subset]& F_2 \ar[bend left=70]{r}\arrow[r,symbol=\subset]& F_2^\perp \arrow[r,symbol=\subset]& F_1^\perp  \arrow[r,symbol=\subset] & V. 
\end{tikzcd}\]
It's enough to show that $F_2+\psi(F_2)$ is isotropic. For that it suffices to prove that $\langle \psi v, \psi w\rangle=0$ for all $v,w\in F_2$. This follows from the fact that $\psi^2(w)=x^2 w$. Indeed we have
\[\langle \psi v, \psi w\rangle=-\langle  v, \psi^2 w\rangle = -x^2 \langle  v, w\rangle =0. \]
The same argument works in cases $(iii)$ and $(iv)$. \par
In case $(ii)$ we have flags
\[\begin{tikzcd}[column sep=.7em, row sep=.7em]  
& & & & F_4 \arrow[dr,symbol=\subset] &  & & & \\
0 \arrow[r,symbol=\subset]& F_1  \ar[bend left=50]{rr} \arrow[r,symbol=\subset] &F_2  \ar[bend left=90]{rrrr} \arrow[r,symbol=\subset]&F_3  \arrow[ur,symbol=\subset]\arrow[dr,symbol=\subset] & &F_3^\perp  \ar[bend left=50]{rr}  \arrow[r,symbol=\subset]& F_2^\perp  \arrow[r,symbol=\subset] &F_1^\perp   \arrow[r,symbol=\subset]&  V. \\\
& & & & F_4' \arrow[ur,symbol=\subset] &  & & & \\
\end{tikzcd}\]
Assume that $F_2+\psi(F_1)\neq F_3$. Then $\psi(F_1)\subset F_2$ and we may take $F_1+\psi(F_1)$. If $F_2+\psi(F_1)=F_3$, then $F_3\subset F_2+\psi(F_2)$ and this space is isotropic by the same argument as before. Therefore this space is a Lagrangian subspace containing $F_3$ and as such it is either $F_4$ or $F_4'$ and we are done.  \par
%

\subsubsection{Case (2), \ref{orbitD}}
We have $x_{\bQ}=(1/4,1/4,1/4,1/4)$ and the possible non-trivial cases are
\begin{enumerate}[nosep, label=(\roman*)] 
\item $wx_{\bQ}=(3/4,3/4,1/4,1/4)$,
\item $wx_{\bQ}=(5/4, 3/4, 1/4, -1/4)$.
\end{enumerate}
In case $(i)$ we consider flags 
\[\begin{tikzcd}[column sep=.7em, row sep=.7em]  
0 \arrow[r,symbol=\subset]  &F_2  \ar[bend left=70]{r} \arrow[r,symbol=\subset]& F_4 \ar[bend left=70]{r}\arrow[r,symbol=\subset]& F_2^\perp  \arrow[r,symbol=\subset] & V. 
\end{tikzcd}\]
Choose a basis $u,v\in F_2$. We first want to prove that $v_x'=0$ for all $v'\in F_2$. Since $\dim V(x)=1$ there are scalars $\lambda_x, \mu_x$ which are not both zero such that $\mu_x u_x+\lambda_x v_x=0$. Define $w=\lambda_x v+\mu_x u$. Then $w\neq 0$ and since $\langle w,\psi^2 w\rangle=0$ we find that $\langle w_y,w_{-y}\rangle =0$. This implies that $w_y=0$ or $w_{-y}=0$. Without loss of generality we may assume that $w_y=0$, i.e. $w=w_0+w_{-y}+w_{-x}$. Now $\psi w, \psi^2 w\in F_2^\perp$ implies that $w_{-x}\in F_2^\perp$ and $w_{-y}\in F_2^\perp$. If $w_{-x}=0$ then if also $w_{-y}=0$ the space $F_2$ contains $w_0\neq 0$ and in particular the line spanned by it and we may insert that. So either $w_{-x}\neq 0$ or $w_{-y}\neq 0$ and we may assume that $w_{-x}\neq 0$ (otherwise we end up eliminating the $y$-component in $F_2$). Now for all $v'\in F_2$ we have $\langle v',w_{-x}\rangle =0$ and therefore $v'=v_y'+v_0'+v_{-y}'+v_{-x}'$. \par
In the second step we will prove that also $v_y'=0$ for all $v'\in F_2$. As before there are scalars $\lambda_{x}, \mu_{-x}$ such that 
\[\lambda_{-x}v_{-x}+\mu_{-x}u_{-x}=0 .\]
We define $w'=\lambda_{-x}v+\mu_{-x}u\neq 0$. The same argument as before shows that $w_{-y}'\in F_2^\perp$ and that we may assume it's non-zero. This proves that $v_y'=0$ for all $v'\in F_2$. \par
A simple calculation now shows that $W=F_2+\psi(F_2)+\psi^2(F_2)$ is $\psi$-stable and isotropic. If neither $F_2$ nor $F_2+\psi(F_2)$ are already stable, then $W$ is either $F_4$ or some other Lagrangian and we may insert that space. \par
In case $(ii)$ the flags are
\[\begin{tikzcd}[column sep=.7em, row sep=.7em]  
0 \arrow[r,symbol=\subset]& F_1  \ar[bend left=70]{r} \arrow[r,symbol=\subset] &F_2  \ar[bend left=70]{r} \arrow[r,symbol=\subset]&F_3  \arrow[r,symbol=\subset] \ar[bend left=40]{rrr}& F_4 \arrow[r,symbol=\subset]&F_3^\perp   \arrow[r,symbol=\subset]& F_2^\perp  \arrow[r,symbol=\subset] &F_1^\perp   \arrow[r,symbol=\subset]&  V. 
\end{tikzcd}\]
Since $\psi^2(F_2)\subset F_2^\perp$ we can insert the space $F_2+\psi(F_2)+\psi^2(F_2)$ and the argument is the same as in case (i).

\subsubsection{Case (3), \ref{orbitD}}
We have $x_{\bQ}=(1/4,1/4,0,0,0)$ and may assume that $x_i-x_{i+1}\le 1/2$ for $i\le 3$ (otherwise we may insert a space of the form $F+\psi(F)$ immediately). This leaves us with the following non-trivial cases 
\begin{enumerate}[nosep, label=(\roman*)] 
\item $wx_{\bQ}=(3/4,3/4,1/4,0,0)$,
\item $wx_{\bQ}=(5/4,3/4,3/4,0,0)$,
\item $wx_{\bQ}=(5/4,1,3/4,3/4,0)$.
\end{enumerate}
In case $(i)$ the flags are of the form 
\[\begin{tikzcd}[column sep=.7em, row sep=.7em]  
0 \arrow[r,symbol=\subset]  &F_2  \ar[bend left=70]{rr} \arrow[r,symbol=\subset]& F_3 \ar[bend left=70]{rr}\arrow[r,symbol=\subset]& F_3^\perp \arrow[r,symbol=\subset]& F_2^\perp  \arrow[r,symbol=\subset] & V. 
\end{tikzcd}\]
It's enough to show that actually $\psi(F_3)\subset F_3^\perp$. This will imply that $F_3+\psi(F_3)$ is isotropic and stable. Choose a basis $v_1,v_2$ of $F_2$ and extend to the basis $v_1,v_2,v_3$ of $F_3$. We need to show that $\langle u,\psi(v)\rangle =0$ for all $u,v\in F_3$. Write $u=\sum \lambda_i v_i$ and $v=\sum \mu_iv_i$ in the chosen basis. We have
\[\langle u,\psi(v) \rangle = \mu_3 \langle u,\psi(v_3)\rangle \]
since $\psi(F_2)\subset F_3^\perp$. Furthermore since also $\psi(F_3)\subset F_2^\perp$ we have 
\[ \mu_3 \langle u,\psi(v_3)\rangle =\lambda_3\mu_3 \langle v_3, \psi(v_3) \rangle. \]
The last term vanishes because for any $w=w_x+w_{-x}$ we have
\[\langle w, \psi(w) \rangle = -x\langle w_x,w_{-x} \rangle + x \langle w_{-x},w_x \rangle =0  \]
since the pairing is symmetric. \par
Cases $(ii)$ and $(iii)$ work the same as case $(i)$ in §\ref{FlagsD12}.

\subsubsection{Case (4), \ref{orbitD}} 
We have $x_{\bQ}=(1/4,1/4,1/4,1/4,1/4)$ and the non-trivial possibilities for $wx_{\bQ}$ are 
\begin{enumerate}[nosep, label=(\roman*)] 
\item $wx_{\bQ}=(3/4,3/4,1/4,1/4,1/4)$,
\item $wx_{\bQ}=(5/4,3/4,1/4,1/4,-1/4)$,
\item $wx_{\bQ}=(7/4,5/4,3/4,1/4,1/4)$,
\end{enumerate}
In case $(i)$ we consider flags
\[\begin{tikzcd}[column sep=.7em, row sep=.7em]  
0 \arrow[r,symbol=\subset]  &F_2  \ar[bend left=70]{r} \arrow[r,symbol=\subset]& F_5 \ar[bend left=70]{r}\arrow[r,symbol=\subset]& F_2^\perp  \arrow[r,symbol=\subset] & V. 
\end{tikzcd}\]
We claim that $W=F_2+\psi(F_2)+\psi^2(F_2)$ is isotropic. Note that we may assume $\psi$ to be injective on $F_2$, because otherwise we can insert its kernel. It suffices to show that $\langle \psi v, \psi^2 w\rangle =0$ and that $\langle \psi^2 v, \psi^2 w\rangle =0$ for all $v,w\in F_2$. We have 
\[\langle \psi v, \psi^2 w\rangle=\langle \psi v, \psi^2 w\rangle+\langle \psi v ,x^2 w_0\rangle =x^2\langle \psi v, w\rangle =0 \]
and similarly we get $\langle \psi^2 v, \psi^2 w\rangle =0$. This implies that the dimension of $W$ is at most $5$ and hence we have $F_2\cap \psi(F_2)\neq 0 $ or $(F_2+\psi(F_2))\cap \psi^2(F_2)\neq 0$. In the first case $F_2$ contains a non-zero vector $w$ such that $\psi w\in F_2$ and hence it contains $2xw_x+xw_0$ and $-2xw_{-x}-xw_0$. We may assume that $w_x, w_0, w_{-x}\neq 0$ because otherwise $F_2$ contains an eigenvector. Now $F_2$ is spanned by $2xw_x+xw_0$ and $-2xw_{-x}-xw_0$ and hence $F_2+\psi(F_2)$ is $\psi$-stable. \par 
If $F_2\cap \psi(F_2)=0$ and $(F_2+\psi(F_2))\cap \psi^2(F_2)\neq 0$ we either have that $W$ is five-dimensional in which case it's Lagrangian and we may insert it into the flag or $\psi^2(F_2)\subset F_2+\psi(F_2)$ in which case $F_2+\psi(F_2)$ is $\psi$-stable. \par
In cases $(ii)$ and $(iii)$ we may simply insert the space $F_1+\psi(F_1)+\psi^2(F_1)$ which is automatically stable and isotropic.

\subsubsection{Case (5), \ref{orbitD}}
We have $x_{\bQ}=(1/4,1/4,1/4,0,0,0)$ and as for $n=5$ we may assume that $x_i-x_{i+1} \le 1/2$ for $i\le 4$. This leaves us with the following cases
\begin{enumerate}[nosep, label=(\roman*)] 
\item $wx_{\bQ}=(3/4,3/4,1/4,0,0,0)$,
\item $wx_{\bQ}=(1,1,3/4,3/4,1/4,0)$,
\item $wx_{\bQ}=(5/4,1,3/4,3/4,0,0)$,

\end{enumerate}

The arguments in these cases are the same as for $n=5$.

\subsubsection{Case (6), \ref{orbitD}} 

We have $x_{\bQ}=(1/4,1/4,1/4,1/4,1/4,1/4)$. The list of possible cases is

\begin{enumerate}[nosep, label=(\roman*)]
\item $wx_{\bQ}=(3/4,3/4,1/4,1/4,1/4,1/4)$,
\item $wx_{\bQ}=(3/4,3/4,3/4,3/4,1/4,1/4)$,
\item $wx_{\bQ}=(5/4,3/4,1/4,1/4,1/4,-1/4)$,
\item $wx_{\bQ}=(5/4,3/4,3/4,3/4,1/4,-1/4)$,
\item $wx_{\bQ}=(5/4,5/4,3/4,3/4,1/4,1/4)$,
\item $wx_{\bQ}=(5/4,5/4,5/4,3/4,1/4,-1/4))$,
\item $wx_{\bQ}=(7/4, 5/4, 3/4, 1/4, 1/4, 1/4)$,
\item $wx_{\bQ}=(7/4, 7/4, 5/4, 3/4, 1/4, -1/4)$,
\item $wx_{\bQ}=(9/4, 7/4, 5/4, 3/4, 1/4, 1/4)$.
\end{enumerate}

In case $(i)$ we look at flags
\[\begin{tikzcd}[column sep=.7em, row sep=.7em]  
0 \arrow[r,symbol=\subset]  &F_2  \ar[bend left=70]{r} \arrow[r,symbol=\subset]& F_6 \ar[bend left=70]{r}\arrow[r,symbol=\subset]& F_2^\perp  \arrow[r,symbol=\subset] & V. 
\end{tikzcd}\]
We may assume that $\psi |_{F_2}$ is injective. As before $W=F_2+\psi(F_2)+\psi^2(F_2)$ is isotropic and $\psi$-stable. If $\dim(W)=6$ it's Lagrangian and either equal to $F_6$ or we may insert it as a second Lagrangian subspace. If $\dim(W)<6$ we have $F_2\cap \psi(F_2)\neq 0 $ or $(F_2+\psi(F_2))\cap \psi^2(F_2)\neq 0$. \par
If $F_2\cap \psi(F_2)\neq 0 $ we conclude that $F_2+\psi(F_2)$ is $\psi$-stable as in the corresponding case for $n=5$. If $F_2\cap \psi(F_2)=0 $ and $(F_2+\psi(F_2))\cap \psi^2(F_2)\neq 0$ there is a vector $v+\psi(u)\in \psi^2(F_2)$. This implies $v_0=0$ and $\psi(u)=x^2(v+\psi(u))-\psi^2(v)\in \psi^2(F_2)$. Therefore $\psi(u)=\psi^2(u')$ for some $u'\in F_2$ and hence $u\in F_2\cap \psi(F_2)=0 $. In the end we find that $v\neq 0$, i.e. $F_2$ contains the non-zero vector $v=v_x+v_{-x}$. If either component vanishes, then $v$ is an eigenvector and we may insert the line spanned by it into the flag. If both are non-zero, then since $v_x,v_{-x}\in \psi^2(F_2)$ we get $v_x,v_{-x}\in (F_2+\psi(F_2))\cap \psi^2(F_2)$ and this intersection has to be two-dimensional. But that implies $\psi^2(F_2)\subset (F_2+\psi(F_2)$ and hence $F_2+\psi(F_2)$ is $\psi$-stable. \par
In case $(ii)$ we consider flags of the form 
\[\begin{tikzcd}[column sep=.7em, row sep=.7em]  
0 \arrow[r,symbol=\subset]  &F_4  \ar[bend left=70]{r} \arrow[r,symbol=\subset]& F_6 \ar[bend left=70]{r}\arrow[r,symbol=\subset]& F_4^\perp  \arrow[r,symbol=\subset] & V. 
\end{tikzcd}\]
The space $F_4+\psi(F_4)+\psi^2(F_4)$ is isotropic. If $F_4$ is not stable, then $F_4+\psi(F_4)$ is at least five dimensional. If this space itself is not already stable, then $F_4+\psi(F_4)+\psi^2(F_4)$ is Lagrangian and we may use it to refine the flag. 
In case $(iii)$ we may simply insert the space $F_1+\psi(F_1)+\psi^2(F_1)$. Case $(iv)$ is similar to case $(ii)$, we may insert the space $F_4+\psi(F_4)+\psi^2(F_4)$. \par
In case $(v)$ we have flags
\[\begin{tikzcd}[column sep=.7em, row sep=.7em]  
0 \arrow[r,symbol=\subset]  &F_2  \ar[bend left=70]{r} \arrow[r,symbol=\subset]&F_4  \ar[bend left=70]{r} \arrow[r,symbol=\subset]& F_6 \ar[bend left=70]{r}\arrow[r,symbol=\subset]& F_4^\perp  \arrow[r,symbol=\subset] \ar[bend left=70]{r}&F_2^\perp \arrow[r,symbol=\subset] & V. 
\end{tikzcd}\]

We may insert $F_2+\psi(F_2)$ if it is equal to $F_4$. Otherwise $F_2 \cap \psi(F_2)\neq 0$ and there is a vector $w\in F_2$ such that $\psi(w)\in F_2$. We then find that $F_2+\psi(F_2)$ contains $w_x,w_{-x}$ and $w_0$. If any of these vanishes, $F_2$ contains an eigenvector. Otherwise $F_2+\psi(F_2)$ is generated by these three vectors. \par
Cases $(vi)-(ix)$ are easy, simply use the spaces $F+\psi(F)+\psi^2(F)$ where $F$ always denotes the first-non zero step of the flag. 

\subsection{Stabilizers on the unit coset} \label{s: unit coset} In the previous section we proved the first part of Theorem \ref{thm: hypersp quasi-epi}. In the following we analyze the unit coset in more detail. We prove strictness in types $A$ and $D$. 

\begin{remark} The covering $\Spin(V) \rightarrow \SO(V)$ induces an isomorphism on flag varieties. The stabilizers only differ by the kernel of the covering. Therefore in types $B$ and $D$ we may (and will) work with $\SO(V)$ in the following. 
\end{remark}

\subsection{} Fix a maximal torus $T$ and a Borel $B$ containing it. We may assume $Q\supset B$ and $\psi\in\frt$.  Then $G_{\psi}\supset T$, and we may assume $T\subset B_{\psi}\subset B$. Let $P_{\psi}=BG_{\psi}$; this is a standard parabolic subgroup of $G$ containing $G_{\psi}$ as a Levi subgroup. Let $N, N^{\psi}, N_{\psi}$ be the unipotent radicals of $B$, $P_{\psi}$ and $B_{\psi}$ respectively.

Let $W_{Q}$ be the Weyl group of the Levi subgroup of $Q$ containing $T$. The double cosets $B\bs G/Q$ are parametrized by $W/W_{Q}$.  Let $w\in W$, and let $\dot w$ be any lift of $w$ to $N_{G}(T)$. Then we can identify the $B$-orbit of $\dot wQ/Q$ with $N/N\cap {}^{w}Q$ (write ${}^{w}Q=\Ad(\dot w)Q$). The left translation action of $B_{\psi}=N_{\psi}T$ on $B\dot wQ/Q$ becomes left translation of  $N_{\psi}$ on $N/N\cap {}^{w}Q$ and the action of $T$ by conjugation.  Using the left $N_{\psi}$-action, every $B_{\psi}$ orbit on $N/N\cap{}^{w}Q$ intersects $N^{\psi}/N^{\psi}\cap{}^{w}Q$. On $N^{\psi}/N^{\psi}\cap{}^{w}Q$, there is the residual action of $B_{\psi}\cap{}^{w}Q$ by conjugation. Therefore, it suffices to show that the stabilizers of the action of $B_{\psi}\cap{}^{w}Q$ on $N^{\psi}/N^{\psi}\cap{}^{w}Q$ by conjugation contain nontrivial tori, except in one case.

Let $W_{\psi}\subset W$ be the Weyl group of $G_{\psi}$ (with respect to $T$). 

\begin{lemma} Let $w_{0}$ be the longest element in $W$ with respect to the simple reflections defined by $B$. If $w$ is not in the double coset $W_{\psi}w_{0}W_{Q}$, then the action of $T$ on $N^{\psi}/N^{\psi}\cap{}^{w}Q$ has positive dimensional stabilizers.
\end{lemma}

\begin{proof}Let $N_{Q}^{-}$ be the unipotent radical of the parabolic of $G$ opposite to $Q$ and containing $T$. The inclusion $N^{\psi}\cap {}^{w}N_{Q}^{-}\subset N^{\psi}$ induces an isomorphism $N^{\psi}\cap {}^{w}N_{Q}^{-}\isom N^{\psi}/N^{\psi}\cap{}^{w}Q$. To verify the claim, it suffices to show that the set of roots $\Phi(N^{\psi}\cap {}^{w}N^{-}_{Q})$ does not span $\xch(\TT)_{\QQ}$ rationally, unless $w\in W_{\psi}w_{0}W_{Q}$. This can be checked case by case.

For example, consider the case where $G=\Sp(V)$ and $P_{\psi}$ and $Q$ are both Siegel parabolic subgroups. Let $\{e_{1},\cdots, e_{n}, e_{-n},\cdots, e_{-1}\}$ be a symplectic basis of $V$; let $T$ be the diagonal torus with respect to this basis.  The above claim is equivalent to the following statement: let $L,L'$ be two Lagrangian subspaces spanned by part of the basis. Let $N_{L}$ be the unipotent radical of the stabilizer of $L$, and similarly define $N_{L'}$. Then as long as $L\ne L'$, the action of $T$ on $L\cap L'$ has positive dimensional  stabilizers. 

This latter statement can be proved as follows. Identify $N_{L}$ with $\Sym^{2}(L)$ using the symplectic form. Then $N_{L}\cap N_{L'}\subset N_{L}$ is identified with the subspace $\Sym^{2}(L\cap L')$ of $\Sym^{2}(L)$. If $L\ne L'$, let $T_{1}$ be the subtorus of $T$ corresponding the basis elements $e_{\pm i}$ such that $e_{\pm i}\notin L\cap L'$. Then $T_{1}$ acts trivially on $\Sym^{2}(L\cap L')\cong N_{L}\cap N_{L'}$. 
\end{proof}

\begin{remark} More precisely let $U$ denote the rational span of $\Phi(N^{\psi}\cap {}^{w}N^{-}_{Q})$. If $U\neq \xch(T)_{\QQ}$ then the stabilizer contains $\bigcap_{\alpha \in \Phi(G)\cap U} \ker(\alpha)$. The set $\Phi(G)\cap U$ is the root system of a Levi subgroup of $G$ and $\bigcap_{\alpha \in \Phi(G)\cap U} \ker(\alpha)$ is its center. 
\end{remark}

\subsection{} It remains to treat the case $w\in W_{\psi}w_{0}W_{Q}$.  Write $w\in vw_{0}W_{Q}$ for some $v\in W_{\psi}$. We need to consider the stabilizers of the conjugation action of $B_{\psi}\cap {}^{vw_0}Q$ on $N^{\psi}/N^{\psi}\cap{}^{vw_0}Q$. Conjugating by a lifting of $v$ in $N_{G_{\psi}}(T)$, we may as well consider the action of ${}^{v^{-1}}B_{\psi}\cap {}^{w_0}Q$ on $N^{\psi}/N^{\psi}\cap {}^{w_0}Q$. Let $Q'=({}^{w_0}Q)^{-}$ be the opposite parabolic of ${}^{w_0}Q$ and let $M=Q'\cap {}^{w_0}Q$ (common Levi of $Q'$ and ${}^{w_0}Q$). Then ${}^{v^{-1}}B_{\psi}\cap {}^{w_0}Q$ always contains a Borel subgroup $B_{H}$ of the Levi subgroup $H=G_{\psi}\cap M$ of $G$. Identifying $N^{\psi}/N^{\psi}\cap {}^{w_0}Q$ with $N^{\psi}\cap N_{Q'}$, we reduce to considering  the action of $B_{H}$ on $N^{\psi}\cap N_{Q'}$ by conjugation (the whole $H$ acts on $N^{\psi}\cap N_{Q'}$ by conjugation). We remark that after this reduction, the roles played by $P_{\psi}$ and $Q'$ are symmetric. 

Below we describe case-by-case the action of $H$ on $N^{\psi}\cap N_{Q'}$ in linear algebra terms. 
In the following, $V_{n}, V'_{n}$ always denote an $n$-dimensional space; when $n=1$, we use $\Gm(V_{1}), \Gm(V'_{1})$ to denote the one-dimensional torus that acts on $V_{1}, V'_{1}$ by scaling. 

In all cases $N^{\psi}\cap N_{Q'}$ is a vector group, and we describe it as a representation of $H$. The Borel subgroup $B_{H}$ of $H$ acts on $N^{\psi}\cap N_{Q'}$ with an open orbit with finite stabilizer. In the following we analyze the orbits of $B_{H}$ acting on $N^{\psi}\cap N_{Q'}$ more precisely. Recall that by $P_{d}$ we denote the stabilizer of a $d$-dimensional subspace $V_d$ and by $P_{d,d'}$ the stabilizer of a flag $0\subset V_d \subset V_{d'} \subset V$ with $V_{d'}$ being $d'$-dimensional. We list the possible conjugacy classes of $Q$ and the corresponding $H$-representation $N^{\psi}\cap N_{Q'}$. 

\subsection{Type $A$}
\begin{enumerate}
\item[(1)] $H=T$ (in this case there is no need to describe $N^{\psi}\cap N_{Q'}$).
\item[(2a)] $G=\SL_{2m}$,  $(P_{\psi},Q)\sim (P_{m}, P_{m,m+1})$ or $ (P_{m}, P_{m,2m-1})$ or $(P_{m}, P_{1,m})$ or $(P_{m}, P_{m-1,m})$. Then $H=S(\GL(V_{m})\times \GL(V_{m-1})\times \Gm(V_{1}))$, $N^{\psi}\cap N_{Q'}=\Hom(V_{m-1},V_{m})\op\Hom(V_{1},V_{m})$ or its dual.

\item[(2b)] $G=\SL_{2m}$, $(P_{\psi},Q)\sim(P_{m}, P_{1,m})$ or $(P_{m}, P_{m-1,2m-1})$. Then $H=S(\GL(V_{m-1})\times \GL(V'_{m-1})\times \Gm(V_{1})\times\Gm(V'_{1}))$, $N^{\psi}\cap N_{Q'}=\Hom(V_{m-1}, V'_{m-1})\op \Hom(V_{1},V'_{1})\op\Hom(V_{m-1}, V_{1}')$ or its dual.

\item[(3a)] $G=\SL_{2m+1}$, $(P_{\psi},Q)\sim(P_{m}, P_{m,m+1})$ or $(P_{m}, P_{1,m+1})$ or $(P_{m+1},P_{m,2m})$ or $(P_{m+1}, P_{m,m+1})$. Then $H=S(\GL(V_{m})\times \GL(V'_{m})\times\Gm(V_{1}))$, $N^{\psi}\cap N_{Q'}=\Hom(V_{m},V'_{m})\op\Hom(V_{1},V'_{m})$ or its dual.

\item[(3b)] $G=\SL_{2m+1}$, $(P_{\psi},Q)\sim (P_{m}, P_{m,2m})$ or $(P_{m+1}, P_{1,m+1})$. Then $H=S(\GL(V_{m-1})\times \GL(V_{m})\times \Gm(V_{1})\times\Gm(V'_{1}))$, $N^{\psi}\cap N_{Q'}=\Hom(V_{m},V_{m-1})\oplus \Hom(V_m,V_1) \oplus \Hom(V_1,V_1')$ or its dual.

\item[(4)] $H=S(\GL(V_{2})\times\GL(V'_{2})\times\GL(V''_{2}))$, $N^{\psi}\cap N_{Q'}=\Hom(V_{2},V_{2}')\op \Hom(V_{2}, V_{2}'')$ or its dual.
\end{enumerate}

Case (1) is clear for dimension reasons: any non-open orbit in $N^{\psi}\cap N_{Q'}$ has dimension less than that of $T$, hence has positive dimensional subtorus in the stabilizers. In addition it's easy to check that the stabilizer on the open orbit is just the center of $\SL_{n}$ and that outside the open orbit the stabilizer always contains the center of a maximal Levi subgroup.

We reduce the cases (2a)(2b)(3a)(3b) to Case (1). We give the argument only for (2a). Let $a:V_{m-1}\to V_{m}$, $b: V_{1}\to V_{m}$, and $a\op b: V_{m-1}\op V_{1}\to V_{m}$.  The image $\Stab_{B_{H}}(a,b)\to \GL(\coker(a\op b))$ contains a Borel subgroup of the target.  Therefore $\Stab_{B_{H}}(a,b)$ contains a nontrivial torus if $a\op b$ is not surjective (equivalently not an isomorphism). In the remaining case, we may assume $a\op b: V_{m-1}\op V_{1}\isom V_{m}$ is an isomorphism. Such $(a,b)$ form a single orbit $O$ under $H$. We have an isomorphism of stacks $B_{H}\bs O\cong B(V_{m-1})\bs \GL(V_{m})/B(V_{m})$, where $B_{H}=B(V_{m-1})\times B(V_{m})$, and $B(V_{m-1})$ is embedded into $\GL(V_{m})$ via $a$.  This is exactly the situation of case (1) for $\GL(V_{m})$ where $B_{\psi}$ is the Borel subgroup of a Levi subgroup of of type $(m-1,1)$. The computation of the generic stabilizer and stabilizer outside the open orbit is also reduced to Case (1).

For Case (4)  we can use similar argument as above to reduce to the open $H$-orbit $O$ where $a: V_{2}\to V'_{2}$ and $b:V_{2}\to V'_{2}$ are both isomorphisms. For this $H$-orbit $O$, we use $a$ and $b$ to identify $V_{2}'$ and $V''_{2}$ with $V_{2}$, and write $B_{H}=S(B_{2}\times B'_{2}\times B''_{2})$ (where $B_{2}\subset \GL(V_{2})$ is a Borel etc),  then $\Stab_{B_{H}}(a,b)$ surjects onto $(B_{2}\cap B_{2}'\cap B_{2}'')/\Gm$ (modulo scalar matrices).  If $(a,b)$ is in a non-open $B_{H}$-orbit, then two of the Borels $B_{2}, B_{2}', B_{2}''$ are the same and $(B_{2}\cap B_{2}'\cap B_{2}'')/\Gm$ contains the center of some maximal Levi subgroup. On the open orbit it's clear that the intersection $B_{2}\cap B_{2}'\cap B_{2}''$ is the center of $\GL_2$ and hence the generic stabilizer is the center of $\SL_6$.

\subsection{Type $B$}\label{openB}
\begin{enumerate}
\item $H=\GL(V_{n})$, $N^{\psi}\cap N_{Q'}=\wedge^{2}(V_{n})\op V_{n}$.
\item $H=T=\Gm(V_{1})\times \Gm(V_{1}')$, $N^{\psi}\cap N_{Q'}=V_{1}\ot V_{1}'\op V_{1} \ot (V_{1}')^\vee$
\item $H=\Gm(V_{1})\times \Gm(V'_{1})\times \SO(V_{3})$, $N^{\psi}\cap N_{Q'}=V_{1}\ot V_{3}\op V_{1}\ot V_{1}'$.
\end{enumerate}

Case (1) is not strict as it may have several relevant orbits.
In Case (2) the stabilizer on the open orbit is $\mu_2=\{ \pm \id \}$, so it is not strict. \par
We consider Case (3). We identify $N^\psi\cap N_{Q'}$ with $V_3\ot V_1$ on which $(\lambda, \mu, X)\in \Gm\times \Gm \times \SO(V_3)$ acts as $(v,x)\mapsto (\lambda X v, \lambda \mu x)$.  Let $B_3\in \SO(V_3)$ be the stabilizer of an isotropic line $\ell$. Fix a vector $(v,x)$. If $v$ is isotropic or $x=0$ it's easy to see that the stabilizer of $(v,x)$ in $B_H=B_3\times \Gm\times \Gm$ contains a torus. Therefore assume that $v$ is anisotropic and $x\neq 0$. Denote by $\langle, \rangle$ the symmetric bilinear form on $V_3$. Then if $(\lambda, \mu, X) \in \Stab_{B_H}(v,x)$ we find that $\mu=\lambda^{-1}$ and $\langle v,v\rangle = \langle \lambda Xv, \lambda Xv\rangle = \lambda^2 \langle v,v\rangle$, implying that $\lambda=\pm 1$. The vector $v$ determines a splitting $V=\langle v \rangle^\perp \op \langle v \rangle$ and $\Stab_{\SO(V_3)}(v) \cong \SO(\langle v \rangle^\perp)\cong \Gm$. If $v\in \ell^\perp$ the stabilizer $\Stab_{B_H}(v,x)$ contains this $\Gm$. The open orbit is given by those $(v,x)$ where $v\notin \ell^\perp$ and $x\neq 0$. It is easily verified that in this case the stabilizer is $\mu_2=\{ \pm \id \}$ and hence this case is also not strict. 
\subsection{Type $C$}
\begin{enumerate}
\item $H=\GL(V_{n})$,  $N^{\psi}\cap N_{Q'}=\Sym^{2}(V_{n})$.
\item $H=\Gm(V_{1})\times \Gm(V'_{1})\times \SL(V_{2})$, $N^{\psi}\cap N_{Q'}=V_{1}\ot V_{2}\op V_{1}\ot V_{1}\op V_{1}\ot V_{1'}$.
\end{enumerate}

We consider Case (1). We may identify $\Sym^{2}(V_{n})$ with the space of quadratic forms on $V_{n}^{*}$. Let $L=\ker(q)\subset V_{n}^{*}$, and $P_{L}\subset \GL(V_{n})$ be the parabolic subgroup stabilizing $L$. Then we have a natural map $\Stab_{B_{H}}(q)\subset B_{H}\cap P_{L}\to \GL(L)$ whose image is a Borel subgroup of $\GL(L)$. Therefore if $L\ne0$, $\Stab_{B_{H}}(q)$ contains a nontrivial torus. Now suppose $L=0$, i.e., $q$ is nondegenerate. We equip $V_{n}$ with the quadratic form induced from $q$, and still denote it by $q$. Let $B_{H}$ be the stabilizer of a complete flag $F=(0\subset V_{1}\subset \cdots\subset V_{n-1}\subset V_{n})$. Consider the relative position  of the flag  $F^{\bot}=(0\subset V_{n-1}^{\bot}\subset \cdots V_{1}^{\bot}\subset V_{n})$ ($(-)^{\bot}$ is taken under the quadratic form $q$). If $F$ and $F^{\bot}$ are not opposite,  consider the first $i\ge1$ such that $V_{i}\cap V_{i}^{\bot}\ne0$ (i.e., the first $i$ such that $q|_{V_{i}}$ is degenerate), in which case $\ker(q|V_{i})$ is $1$-dimensional. Then $\Stab_{B_{H}}(q)\to \Gm(\ker(q|V_{i}))$ is surjective, hence $\Stab_{B_{H}}(q)$ contains a nontrivial torus. If $F$ and $F^{\bot}$ are opposite, then $q$ is the in the open $B_{H}$-orbit of $\Sym^{2}(V_{n})$. The intersection of the stabilizers of $F$ and $F^\perp$ is a maximal torus $T$ of $\GL(V_n)$ and the stabilizer on the open orbit is the $2$-torsion $T[2]\cong \mu_2^n$. \par
We consider Case (2). An element $(\lambda, \mu, X)\in H$ acts on a vector $(v,x,y)$ as $(\lambda Xv, \lambda^2 x, \lambda \mu y)$. Let $B_2$ be the stabilizer of a line $\ell \subset V_2$ and $B_H=\Gm \times \Gm \times B_2$. If $v=0$ then $\Stab_{B_H}(v,x,y)$ contains $B_2$. If $x=0$ then the stabilizer contains the intersection of $B_2$ and the stabilizer of $\langle v \rangle$ and if $y=0$ the stabilizer contains a torus coming from the $\Gm$-factors. We therefore assume that $v\neq 0, x\neq 0$ and $y\neq 0$. The open orbit is given by those $(v,x,y)$ for which in addition $v\notin \ell$. The stabilizer on the open orbit is $\mu_2=\{\pm \id \}$. If $v\in \ell$ the stabilizer does not contain a torus but the unipotent radical of $B_2$. 

\subsection{Type $D$}
\begin{enumerate}
\item[(1)] $(P_{\psi},Q)\sim (P_{4}, P_{1,2})$. Then $H=\Gm(V_{1})\times \Gm(V_{1}')\times \GL(V_{2})$, $N^{\psi}\cap N_{Q'}=V_{1}\ot V_{2}\op V_{1}'\ot V_{2}\op V_{1}\ot V_{1}'$. 


\item[(2)] $H=\GL(V_{2})\times \GL(V_{3})$, $N^{\psi}\cap N_{Q'}=V_{3}\ot V_{2}\op \wedge^{2}(V_{3})$.
\item[(3)] $H=\GL(V_{3})\times \GL(V'_{3})$, $N^{\psi}\cap N_{Q'}=V_{3}\ot V_{3}'\op \wedge^{2}(V_{3})$.
\end{enumerate}
We consider Case (1). An element $(\lambda, \mu, X)\in H$ acts on a vector $(u,v,x)$ via $(u,v,x)\mapsto (\lambda X u, \mu Xv, \lambda \mu x)$. Let $B_2 \in GL(V_2)$ be a Borel subgroup stabilizing a line $\ell \subset V_2$ and $B_H=\Gm\times \Gm \times B_2$. If $x=0$ the stabilizer of $(u,v,x)$ contains a torus coming from the $\Gm$-factors. If $u=0$ then the condition on $X$ is equivalent to asking that it stabilizes the line spanned by $v$ and $\Stab_{B_H}(0,v,w)$ contains the intersection of $B_2$ with the Borel subgroup stabilizing $\langle v \rangle$. The same happens for $v=0$ and we may therefore assume that $u\neq 0, v\neq 0$ and $x\neq 0$. The open orbit is given by those vectors $(u,v,x)$ for which $x\neq 0$, $u\notin \ell$, $v\notin \ell$ and $u$ and $v$ are linearly independent (i.e. $X$ is contained in the intersection of three pairwise different Borel subgroups). It's easy to check that the stabilizer on the open orbit is $\mu_2\cong \{\pm \id\}$. One may check explicitly that outside the open orbit the stabilizers contain tori which are contained in centers of maximal Levi subgroups. For $x=0$ one finds a torus of the form $Z(\GL_2)$ in a Levi subgroup isomorphic to $\GL_2\times \SO(4)$ and for $x\neq 0$ one obtains the center of a Levi subgroup isomorphic to $\Gm \times \SO(6)$ (as long as we're outside the open orbit).

\par
For Case (2) we identify $V_3 \ot V_2$ with $\Hom(V_3^*,V_2)$. An element $(Y,X)\in H$ acts on $f\in \Hom(V_3^*,V_2)$ via $f \mapsto YfX^{-1}$. Let $(\omega,f)\in\Hom(V_3^*,V_2)\op \wedge^{2}(V_3)$ and let $B_H = \Stab(\ell) \times \Stab(F)$ for a line $\ell \subset V_2$ and a full flag $F$ given by $0 \subset F_1 \subset F_2 \subset V_3$. If $f$ is not surjective clearly $\Stab_{B_H}(f,\omega)$ contains a torus. We therefore assume $f$ is surjective and hence has a one-dimensional kernel $L\subset V_3^*$. If $(Y,X)\in \Stab_{B_H}(f,\omega)$ then $X$ stabilizes the flag 
\[0\subset L \subset f^{-1}(\ell) \subset V_3^*.\]
Denote by $F^\perp$ the flag in $V_3^*$ orthogonal to $F$ with respect to the canonical pairing for $V_3$ and $V_3^*$. The open orbit is given by those $(f,\omega)$ for which $F^\perp$ and the above flag are opposite (i.e. $X$ lies in a maximal torus of $\GL(V_3)$) and for which $\omega$ (considered as a $2$-form on $V_3^*$) is not contained in any of the duals of the planes $\wedge^2(F_2^\perp \op L) \op \wedge^2(f^{-1}(\ell))$, $\wedge^2(F_2^\perp \op L) \op \wedge^2(F_1^\perp)$ or $\wedge^2(F_1^\perp )\op\wedge^2(f^{-1}(\ell))$. It is easy to verify that the stabilizer on the open orbit is $\mu_2=\{\pm \id\}$ and that it contains a torus outside this orbit. More precisely if $f$ is not surjective, the stabilizer contains the center of a Levi subgroup isomorphic to $\Gm\times \SO(8)$. If $f$ is surjective then outside the open orbit the stabilizers contain the center of a Levi subgroup isomorphic to $\GL_5$. Case $(3)$ works the same. This concludes the proof of strictness in type $A$ and $D$.

\section{Potential examples}\label{s:pot expl}
In this section we give a list of triples $(\bP_{\infty}, \psi, \bQ_{0})$ in type $A$ and exceptional types as potential examples of euphotic automorphic data with a generic choice of $\chi$.  In these examples, we only check that $B_\psi$ acts on $L/Q$ with an open orbit with finite stabilizers (part of condition (1)  in Definition \ref{d:qep data}).

\subsection{Type $A$}

\sss{Setup} Let  $\GG=\PGL(V)$ for some  vector space $V$ over $k$ of dimension $n$, and let $G$ be the split form of $\GG$ over $F$. Since all parahoric subgroups of $G(F_{\infty})$ can be conjugated to be contained in $G(\cO_{\infty})$ in this case, we may assume $\bP_{\infty}\subset G(\cO_{\infty})$. For such $\cP_{\infty}$ the corresponding $\Zm$-grading on $\frg=\op_{i\in\Zm}\frg(i)$ is induced from a $\Zm$-grading on the vector space
\begin{equation*}
V=\op_{i\in\Zm}V_{i}
\end{equation*}
such that 
\begin{eqnarray*}
\frg(i)=\op_{j\in\Zm}\Hom_{k}(V_{j}, V_{i+j}), \quad \forall i\in \Zm-\{0\},\\
\frg(0)=\Lie L=(\op_{j\in \Zm}\End_{k}(V_{j}))/k\cdot \id_{V}.
\end{eqnarray*}
Conversely, any $\Zm$-grading on $V$ with $V_{i}\ne 0$ for all $i\in\Zm$ arises from a parahoric subgroup $\bP_{\infty}\subset G(\cO_{\infty})$. Note that
\begin{eqnarray*}
L\cong \left(\prod_{j\in\Zm}\GL(V_{j})\right)/\D\Gm,\\
V^{*}_{\bP}=\op_{j\in\Zm}\Hom_{k}(V_{j}, V_{j-1}).
\end{eqnarray*}

We give two classes of potential examples.

\sss{Case (1)}
Assume the dimensions of $V_{i}$ satisfy
\begin{equation*}
\dim V_{0}=d_{0}, \quad \dim V_{i}=d \textup{ for }i\ne0, \quad \textup{ and }d_{0}>d>0.
\end{equation*}
Let $Q_{0}\subset \GL(V_{0})$ be a parabolic subgroup; let $\psi_{0}\in \End(V_{0})$ be a semisimple element. We assume
\begin{itemize}
\item The pair $(Q_{0},\psi_{0})$ appears in the list of hyperspecial euphotic data of type $A$ in \S\ref{orbitA}.
\item Let $V^{0}_{0}$ be the zero eigenspace of $\psi_{0}$; let $V'_{0}$ be the sum of nonzero eigenspaces.  Then $\dim V^{0}_{0}=d_{0}-d$ and $\dim V_{0}'=d$.
\end{itemize}

With these preliminary data, we construct  $Q$ and $\psi$ as follows. Let $Q\subset L$ be the parabolic subgroup
\begin{equation*}
Q=\left(Q_{0}\times \prod_{i\in\Zm-\{0\}} \GL(V_{i}) \right)/\D\Gm.
\end{equation*}
Viewing $\psi$ as a collection of maps $V_{i}\to V_{i-1}$ for $i\in \Zm$, we then require it to restrict to  isomorphisms $V'_{0}\isom V_{-1}\isom \cdots\isom V_{1}\isom V'_{0}\subset V_{0}$, and to restrict to zero on $V^{0}_{0}$. Moreover we require that $\psi^{m}|V_{0}=\psi_{0}$.

\sss{} We check that in the above situation, $B_{\psi}$ acts on $L/Q$ with an open orbit with finite stabilizers. Indeed, $L/Q\cong \GL(V_{0})/Q_{0}$, and $L_{\psi}\cong \PGL(V_{0})_{\psi_{0}}$ (the centralizer of $\psi_{0}$ in $\PGL(V_{0})$). Therefore we reduce to the case discussed in \S\ref{orbitA} for the group $\PGL(V_{0})$.

\sss{Case (2)}
Take $\bQ_{0}$ to be the standard Iwahori subgroup of $G(F_{0})$ (i.e., $Q\subset L$ is a Borel subgroup). Fix a decomposition
\begin{equation*}
V_{i}=\ell_{i}\op V^{0}_{i} ,\quad i\in\Zm
\end{equation*}
where $\dim\ell_{i}=1$. Viewing $\psi\in V_{\bP}^{*}$ as a collection of maps $V_{i}\to V_{i-1}$, let it restrict to an isomorphism $\ell_{i}\isom \ell_{i-1}$ and be zero on $V^{0}_{i}$.

\sss{} We check that in the above situation, $B_{\psi}$ acts on $L/Q$ with an open orbit with finite stabilizers. 
We have $L/Q=\prod_{i\in\Zm}\Fl(V_{i})$. We also have $L_{\psi}\cong \prod_{i\in\Zm}\GL(V^{0}_{i})$ (an extra factor of $\Gm$ acting on all the lines $\ell_{i}$ gets cancelled after dividing by scalar matrices). Therefore $B_{\psi}\cong\prod_{i\in\Zm}B_{\psi,i}$ where $B_{\psi, i}\subset \GL(V^{0}_{i})$ is a Borel subgroup. The required property of the $B_{\psi}$-action on $L/Q$ follows from the same property for the $B_{\psi,i}$-action on $\Fl(V_{i})$, which is checked in case (2) of \S\ref{orbitA}.

\begin{remark} We expect case (2) to correspond to hypergeometric local systems with slope $1/m$ at $\infty$ and unipotent monodromy at $0$. Rigid automorphic data corresponding to hypergeometric local systems are constructed in the work of Kamgarpour and Yi \cite{KY20}.
\end{remark}

\subsection{Convention} In the exceptional cases, we always assume $G$ is of {\em adjoint type}. We will indicate the type of $\bP_{\infty}$ by coloring the affine Dynkin diagram of $G(F_{\infty})$: the white nodes are simple roots of $L=L_{\bP}$, and the black nodes are simple roots not contained in $L$. 

When we describe $L$ and $V_{\bP}^{*}$, we will use $V_{i},V'_{i}, W_{i}, F_{i}$, etc. to indicate vector spaces of dimension $i$ over $k$.

\subsection{Type ${}^{3}D_{4}$} 

Type of $\bP_{\infty}$:
 \[\xymatrix{\circ \ar@{-}[r] & \circ \ar@3{<-}[r] & \bu}\]

In this case $m=3$. We have $L\cong \PGL(V_{3})$ acting on $V^{*}_{\bP}\cong\Sym^{3}(V_{3})\ot \det(V_{3})^{-1}$. 

Choose a basis $\{e_{1},e_{2},e_{3}\}$ for $V_{3}$. Take $\psi=(e_{1}e_{2}e_{3})\ot (e_{1}\wedge e_{2}\wedge e_{3})^{-1}\in V^{*}_{\bP}$. Then $L_{\psi}$ is the normalizer of the diagonal torus in $L$ with respect to the basis $\{e_{i}\}$. 

Potential choices of $Q$:   take $Q$ to be a maximal proper parabolic subgroup of $L$ so that $L/Q=\PP(V_{3})$ or $\PP^{\vee}(V_{3})$. Then $B_{\psi}=L_{\psi}^{\circ}$ acts on $L/Q$ with an open  free orbit.


\subsection{Type $F_{4}$}

\subsubsection{} Type of $\bP_{\infty}$:
\[\xymatrix{\circ \ar@{-}[r] &\circ \ar@{-}[r] & \bu \ar@{=>}[r] & \circ \ar@{-}[r] &\circ}\]


In this case $m=3$. We have $L\cong (\SL(V_{3})\times \SL(V'_{3}))/\D\mu_{3}$ (modulo diagonal center) acting on $V^{*}_{\bP}=\Sym^{2}(V_{3})\ot V'_{3}$. The factor $\SL(V_{3})$ has short roots of $G$ and $\SL(V'_{3})$  has long roots.   

Choose a basis $\{x_{1},x_{2},x_{3}\}$ of $V_{3}$, and a basis $\{e_{1},e_{2},e_{3}\}$ of $V'_{3}$. Take $\psi=x_{1}^{2}\ot e_{1}+x_{2}^{2}\ot e_{2}+x_{3}^{2}\ot e_{3}$. Then $L_{\psi}^{\circ}$ is a $2$-dimensional torus. The projection $L_{\psi}^{\circ}\to \PGL(V_{3})$ is an isomorphism onto the diagonal torus of $\PGL(V_{3})$ with respect to the basis $\{x_{i}\}$. The other projection $L_{\psi}^{\circ}\to \PGL(V'_{3})$ is a $\mu_{2}^{2}$-cover of the maximal torus of $\PGL(V'_{3})$ with respect to the basis $\{e_{i}\}$.

Potential choices of $Q$: 
\begin{enumerate}
\item $L/Q=\PP(V_{3})$ or $\PP^{\vee}(V_{3})$. In these cases $B_{\psi}=L^{\circ}_{\psi}$ acts on $L/Q$ with an  open free orbit.
\item $L/Q=\PP(V'_{3})$ or $\PP^{\vee}(V'_{3})$. In these cases $B_{\psi}=L^{\circ}_{\psi}$ acts on $L/Q$ with an  open orbit with stabilizer $\mu^{2}_{2}$.
\end{enumerate}

\subsubsection{} Type of $\bP_{\infty}$:
\[\xymatrix{\circ \ar@{-}[r] &\bu \ar@{-}[r] & \circ \ar@{=>}[r] & \circ \ar@{-}[r] &\circ}\]

In this case $m=2$. We have $L\cong (\SL(V_{2})\times \Sp(V_{6}))/\D\mu_{2}$ ($V_{6}$ is a symplectic space of dimension $6$,  $\mu_{2}$ embeds diagonally  into the center of each factor), and $V^{*}_{\bP}\cong V_{2}\ot \ov\wedge^{3}(V_{6})$. Here $\ov\wedge^{3}(V_{6})$ is the cokernel $V_{6}\to \wedge^{3}V_{6}$ given by wedging with the symplectic form on $V_{6}$, so $\dim \ov\wedge^{3}(V_{6})=14$.

Choose a basis $\{e_{1},e_{2}\}$ for $V_{2}$; choose a Lagrangian splitting $V_{6}=W_{3}\op W^{*}_{3}$. Let $\th$ be a volume form on $W_{3}$ (i.e., $\th\in \wedge^{3}W_{3}, \th\ne0$) and $\th^{*}$ be the dual volume form on $W^{*}_{3}$. Consider $\psi=e_{1}\ot \th+e_{2}\ot\th^{*}$. Then $L^{\circ}_{\psi}\cong \GL(W_{3})/\mu_{2}$, where $\GL(W_{3})\incl \SL(V_{2})\times \Sp(V_{6})$ is the following embedding. The projection $\GL(W_{3})\to \Sp(V_{6})$ identifies $\GL(W_{3})$ as the Siegel Levi preserving the splitting $W_{3}\op W^{*}_{3}$. The projection $\GL(W_{3})\to \SL(V_{2})$ is the composition of $\det: \GL(W_{3})\to \Gm$ and $\Gm\to \SL(V_{2})$ given by $t\mapsto \textup{diag}(t^{-1},t)$ in the basis $\{e_{1},e_{2}\}$.

Potential choices of $Q$:
\begin{enumerate}
\item $L/Q\cong \PP(V_{2})\times \PP(V_{6})$. Then $B_{\psi}$ acts on $L/Q$ with an open free orbit.
\item $L/Q$ is the space of Lagrangians in $V_{6}$. The open $L_{\psi}^{\circ}$-orbit of $L/Q$ is isomorphic to the space of non-degenerate quadratic forms on $W_{3}$, which has an open $B_{\psi}$-orbit with stabilizer $\mu_{2}^{2}$.
\end{enumerate}

\subsection{Type $E_{6}$}

\subsubsection{} Type of $\bP_{\infty}$:
\[\xymatrix{\circ\ar@{-}[r] & \circ\ar@{-}[r]& \bu\ar@{-}[r]\ar@{-}[d]& \circ\ar@{-}[r]&  \circ\\
& &  \circ\ar@{-}[d] \\ && \circ}\]


In this case $m=3$. We have $L\cong \SL(V_{3})\times \SL(V'_{3})\times \SL(V''_{3})/(\mu_{3}\times\mu_{3}\times\mu_{3})^{\prod=1}$ acting on $V^{*}_{\bP}\cong V_{3}\ot V'_{3}\ot V''_{3}$.  Here $(\mu_{3}\times\mu_{3}\times\mu_{3})^{\prod=1}$ is the subgroup of the central $\mu_{3}$'s with product $1$.

Choose bases $\{x_{i}\}_{1\le i\le 3}$ for $V_{3}$,  $\{x'_{i}\}_{1\le i\le 3}$ for $V'_{3}$ and $\{x''_{i}\}_{1\le i\le 3}$ for $V''_{3}$. Take $\psi=x_{1}\ot x'_{1}\ot x''_{1}+x_{2}\ot x'_{2}\ot x''_{2}+x_{3}\ot x'_{3}\ot x''_{3}\in V^{*}_{\bP}$. Then $L_{\psi}^{\circ}$ is a $4$-dimensional torus. The projection $L^{\circ}_{\psi}\to\PGL(V_{3})\times \PGL(V'_{3})$ is an isomorphism onto the diagonal torus in the target with respect to the chosen bases. Same for the other two projections.

Potential choices of $Q$: $L/Q$ can be $\PP(V_{3})\times \PP(V'_{3})$, or changing $\PP$ to $\PP^{\vee}$, and changing $(V_{3}, V'_{3})$ to other pairs $(V'_{3}, V''_{3})$ or $(V''_{3}, V_{3})$. The action of $B_{\psi}=L_{\psi}^{\circ}$ on  $L/Q$ has an open free orbit in all cases.

\subsubsection{} Type of $\bP_{\infty}$:
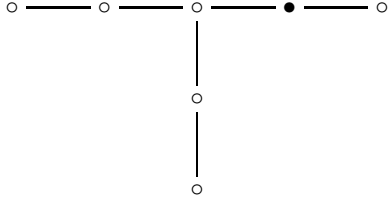
\[\xymatrix{\circ\ar@{-}[r] & \circ\ar@{-}[r]& \circ\ar@{-}[r]\ar@{-}[d]& \bu\ar@{-}[r]&  \circ\\
& &  \circ\ar@{-}[d] \\ && \circ}\]
and two other cases by symmetry.

In this case $m=2$. We have $L\cong \SL(V_{6})\times\SL(V_{2})/(\mu_{3}\times1)\Delta\mu_{2}$
 acting on $V^{*}_{\bP}\cong \wedge^{3}(V_{6})\ot V_{2}$.
 
Choose a basis $\{e_{1},e_{2}\}$ for $V_{2}$; choose a splitting $V_{6}=W_{3}\op W'_{3}$ into two $3$-dimensional spaces. Let $\th$ (resp. $\th'$) be a volume form on $W_{3}$ (resp. $W'_{3}$). Consider $\psi=\th\ot e_{1}+\th'\ot e_{2}$. The projection $L^{\circ}_{\psi}\to\PGL(V_{6})$ identifies $L^{\circ}_{\psi}$ with the Levi subgroup of $\PGL(V_{6})$ preserving the splitting  $W_{3}\op W'_{3}$. 

Potential choices of $Q$: $L/Q$ is a partial flag variety of $\PGL(V_{6})$ of type $(3,2,1)$ (dimensions of associated graded of the partial flag, in any order). Now $B_{\psi}$ is a Borel of $L^{\circ}_{\psi}$, which projects isomorphically to a Levi of $\PGL(V_{6})$ of type $(3,3)$.  The situation $B_{\psi}\bs L/Q$ appears as a special case of \S\ref{orbitA}(4),  from which we know that $B_{\psi}$ acts on $L/Q$ with an open free orbit.


\subsection{Type ${}^{2}E_{6}$}

\subsubsection{} Type of $\bP_{\infty}$:
\[\xymatrix{\circ \ar@{-}[r] &\circ \ar@{-}[r] &\circ \ar@{<=}[r] &\bu  \ar@{-}[r] &\circ}\]

In this case $m=4$. Then $L=(\SL(V_{4})\times \SL(V_{2}))/\G$ acting on $V^{*}_{\bP}\cong \Sym^{2}(V_{4})\ot V_{2}$, where $\G$ is the central subgroup $\{(x,x^{2})\in \mu_{4}\times\mu_{2}\}$.

Choose a basis $\{e_{1},e_{2}\}$ for $V_{2}$ and a basis $\{x_{1},y_{1},x_{2},y_{2}\}$ for $V_{4}$. Take $\psi=x_{1}y_{1}\ot e_{1}+x_{2}y_{2}\ot e_{2}$. Then the projection $L^{\circ}_{\psi}\to \PGL(V_{4})$ is an isomorphism onto the diagonal torus with respect to the basis $\{x_{1},y_{1},x_{2},y_{2}\}$. 

Potential choices of $Q$: $L/Q=\PP(V_{4})$ or $\PP^{\vee}(V_{4})$. It is clear that in both cases $B_{\psi}=L^{\circ}_{\psi}$ acts on $L/Q$ with an open free orbit.

\subsubsection{} Type of $\bP_{\infty}$:
\[\xymatrix{\circ \ar@{-}[r] &\bu \ar@{-}[r] &\circ \ar@{<=}[r] &\circ  \ar@{-}[r] &\circ}\]

In this case $m=4$. We have $L=(\Spin(V_{7})\times \SL(V_{2}))/\D\mu_{2}$ acting on $V^{*}_{\bP}=\D_{8}\ot V_{2}$ where $\D_{8}$ is the $8$-dimensional spin representation of $\Spin(V_{7})$. 

We have an embedding $\ph: \Spin(V_{7})/P_{3}\incl \PP\D_{8}$, where $\Spin(V_{7})/P_{3}$ classifies maximal isotropic subspaces in $V_{7}$.  Choose a splitting $V_{7}=W_{3}\op W'_{3}\op \j{x_{0}}$, where $W_{3}$ and $W'_{3}$ are maximal isotropic and paired perfectly to each other and both orthogonal to $x_{0}$. Let $\{e_{1},e_{2}\}$ be a basis for $V_{2}$. Take $\psi=\wt\ph([W_{3}])\ot e_{1}+\wt\ph([W'_{3}])\ot e_{2}$ (here $\wt\ph([W_{3}])\in \D_{8}$ is a lifting of $\ph([W_{3}])$, unique up to a scalar;  same for $\wt\ph([W'_{3}])$).  The projection $L^{\circ}_{\psi}\to \SO(V_{7})$ is an isomorphism onto the Levi subgroup isomorphic to $\GL(W_{3})$ that stabilizes the splitting $V_{7}=W_{3}\op W'_{3}\op \j{x_{0}}$. The  projection $L^{\circ}_{\psi}\cong\GL(W_{3})\to \PGL(V_{2})$ factors through the determinant and maps onto the diagonal torus with respect to the basis $\{e_{1},e_{2}\}$.

Potential choices of $Q$: 
\begin{enumerate}
\item $L/Q$ is the partial flag variety of $\Spin(V_{7})$ classifying maximal isotropic subspaces of $V_{7}$. The situation $B_{\psi}\bs L/Q$ appears as a special case of \S\ref{orbitB}(1), and we have checked that $B_{\psi}$ acts on $L/Q$ with an open orbit with finite stabilizers.
\item $L/Q=\QQ(V_{7})\times \PP(V_{2})$, where $\QQ(V_{7}) \subset \PP(V_{7})$ is the quadric. We identify $B_{\psi}\bs L/Q$ with $B(W_{3})\bs \QQ(V_{7})$, where $B(W_{3})\subset \SL(W_{3})$ is a Borel subgroup. It is then easy to check that $B(W_{3})$ acts on $\QQ(V_{7})$ with an open free orbit.
\end{enumerate}

\subsection{Type $E_{7}$}

\subsubsection{} 
Type of $\bP_{\infty}$:
\[\xymatrix{\circ\ar@{-}[r] & \circ\ar@{-}[r]& \bu\ar@{-}[r] & \circ\ar@{-}[r]\ar@{-}[d]&  \circ\ar@{-}[r]& \circ\ar@{-}[r] & \circ\\
& & &  \circ }\]
and another case by symmetry.

In this case $m=3$. We have $L\cong \SL(V_{6})\times \SL(V_{3})/(\mu_{2}\times 1)\Delta\mu_{3}$ acting on $V^{*}_{\bP}\cong \wedge^{2}(V_{6})\ot V_{3}$. 

Choose a basis $\{e_{1},e_{2},e_{3}\}$ for $V_{3}$; choose a splitting $V_{6}=W_{1}\op W_{2}\op W_{3}$ into three $2$-dimensional subspaces. Let $\th_{i}$ be a volume form on $W_{i}$. Consider the element $\psi=\th_{1}\ot e_{1}+\th_{2}\ot e_{2}+\th_{3}\ot e_{3}$. The projection $L_{\psi}^{\circ}\to\PGL(V_{6})$ is an isomorphism onto the Levi of $\SL(V_{6})$ stabilizing the splitting $V_{6}=W_{1}\op W_{2}\op W_{3}$, and the projection $L^{\circ}_{\psi}\to \PGL(V_{3})$ has image equal to the diagonal torus with respect to the basis $\{e_{i}\}$. 

Potential choices of $Q$: $L/Q$ is the partial flag variety of $\PGL(V_{6})$ with associated graded dimensions $(4,2)$ (in any order). The situation $B_{\psi}\bs L/Q$ appears in the example \S\ref{orbitA}(8),  from which we know that $B_{\psi}$ acts on $L/Q$ with an open free orbit.

\subsubsection{} Type of $\bP_{\infty}$:
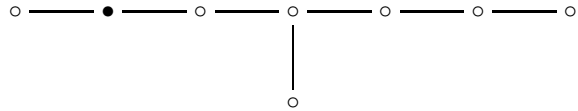
\[\xymatrix{\circ\ar@{-}[r] & \bu\ar@{-}[r]& \circ\ar@{-}[r] & \circ\ar@{-}[r]\ar@{-}[d]&  \circ\ar@{-}[r]& \circ\ar@{-}[r] & \circ\\
& & &  \circ}\]
and another case by symmetry.

In this case $m=2$. Then $L$ is isogenous to  $(\Spin^{+}(V_{12})\times \SL(V_{2}))/\D\mu_{2}$ acting on $V^{*}_{\bP}\cong \D^{+}_{32}\ot V_{2}$, where $\Spin^{+}(V_{12})$ is one of the half-spin quotient of $\Spin(V_{12})$ acting on its half-spin representation $\D^{+}_{32}$.

Choose a basis $\{e_{1},e_{2}\}$ for $V_{2}$. There is an embedding $\ph: \Spin(V_{12})/P_{6}\incl \PP\D^{+}_{32}$, where $\Spin(V_{12})/P_{6}$ is the partial flag variety of one of the two families of Lagrangian subspaces in $V_{12}$. Fix a splitting $V_{12}=W_{6}\op W'_{6}$ into Lagrangians. Take $\psi=\wt\ph([W_{6}])\ot e_{1}+\wt\ph([W'_{6}])\ot e_{2}$ (here $\wt\ph([W_{6}])$ is a lifting of $\ph([W_{6}])$ to $\D_{32}$, up to scalar; same for $\wt\ph([W_{6}])$). The projection $L_{\psi}^{\circ}\to \PSO(V_{12})$ is an isomorphism onto the Siegel Levi stabilizing the splitting $V_{12}=W_{6}\op W'_{6}$, so  $L_{\psi}^{\circ}\cong \GL(W_{6})/\mu_{2}$. The projection $L_{\psi}^{\circ}\cong\GL(W_{6})/\mu_{2}\to \PGL(V_{2})$ factors through the determinant and maps onto the diagonal torus of $\PGL(V_{2})$ with respect to the basis $\{e_{1},e_{2}\}$. Note that $\dim B_{\psi}=21$.

Potential choices of $Q$:
\begin{enumerate}
\item $L/Q=\SO(V_{12})/P_{3}$  is the partial flag variety classifying isotropic $F_{3}\subset V_{12}$. The situation $B_{\psi}\bs L/Q$  has been analyzed in \S\ref{orbitD}(3).

\item $L/Q=\SO(V_{12})/P_{5}\times \PP(V_{2})$, where the first factor classifies isotropic $F_{5}\subset V_{12}$. We check the open orbit condition as follows. We first reduce to study the action of $B_{\psi}^{1}$ (a Borel subgroup of $\SL(W_{6})$) on $Y=\SO(V_{12})/P_{5}$. Note that $Y$ classifies a pair of Lagrangians $U_{6}, U'_{6}\subset V_{12}$ such that $\dim(U_{6}\cap U'_{6})=5$. We may assume $U_{6}$ is conjugate to $W_{6}$. There is an open subset $Y'\subset Y$ classifying those $(U_{6},U'_{6})$ such that $U_{6}$ is the graph of a skew-symmetric map $a:W'_{6}\to W_{6}$. We may identify $Y'$ with $\wedge^{2}(W_{6})\times \PP(W_{6})$ (the choice of $U'_{6}$ is the same as choosing a hyperplane in $U_{6}$, or in $W'_{6}$). The situation of $B^{1}_{\psi}$ acting on $Y'$ is essentially the same as case (1) of \S\ref{openB}.

\item $L/Q=\SO(V_{12})/P_{1,6}\times\PP(V_{2})$, where the first factor classifies isotropic $F_{1}\subset F_{6}\subset V_{12}$ for $F_{6}$ a Lagrangian in the same connected component of $W_{6}$. The same argument as in the previous case reduces to the action of $B^{1}_{\psi}$ on $\wedge^{2}W_{6}\times\PP^{\vee}(W_{6})$, which is essentially the same as case (1) of \S\ref{openB}.
\end{enumerate}

\subsubsection{} 
Type of $\bP_{\infty}$:
\[\xymatrix{\circ\ar@{-}[r] & \circ\ar@{-}[r]& \circ\ar@{-}[r] & \bu\ar@{-}[r]\ar@{-}[d]&  \circ\ar@{-}[r]& \circ\ar@{-}[r] & \circ\\
& & &  \circ }\]


In this case $m=4$. We have $L$ isogenous to $(\SL(V_{4})\times \SL(V'_{4})\times \SL(V_{2}))/\G$ acting on $V^{*}_{\bP}\cong V^{*}_{4}\ot V'_{4}\ot V_{2}$, where $\G$ is the central subgroup $\{(x,y,z)\in \mu_{4}\times\mu_{4}\times\mu_{2}|x^{-1}yz=1\}$ of $L$.

Fix splittings $V_{4}=X_{2}\op Y_{2}$ and $V'_{4}=X'_{2}\op Y'_{2}$ into planes. Let $\ph_{X}: X_{2}\isom X'_{2}$ and $\ph_{Y}: Y_{2}\isom Y'_{2}$ be isomorphisms. Let $\wt\ph_{X}$ be the composition $V_{4}\surj X_{2}\xr{\ph_{X}} X'_{2}\incl V'_{4}$, viewed as an element in $V^{*}_{4}\ot V'_{4}$. Similarly define $\wt\ph_{Y}$.  Let $\{e_{1},e_{2}\}$ be a basis of $V_{2}$. Take $\psi=\wt\ph_{X}\ot e_{1}+\wt\ph_{Y}\ot e_{2}$. Then $L^{\circ}_{\psi}$ is the image of $i:S(\GL(X)\times \GL(Y))\times\Gm\to \SL(V_{4})\times \SL(V'_{4})\times \SL(V_{2})\surj L$. Here $i$ sends $(g_{X},g_{Y}, \l)$ (with $\det(g_{X})\det(g_{Y})=1$) to the triple $(g_{X}\op g_{Y}, \l^{-1}\ph_{X}g_{X}\ph_{X}^{-1}\op \l \ph_{Y}g_{Y}\ph_{Y}^{-1}, \textup{diag}(\l,\l^{-1}))$.  We have $\dim B_{\psi}=6$.

Potential choices of $Q$:
\begin{enumerate}
\item $L/Q=\SL(V_{4})/Q_{1}\times \PP(V_{2})$, where $\SL(V_{4})/Q_{1}$ is a partial flag variety of $\SL(V_{4})$ with associated graded dimensions $(2,1,1)$ (in any order). The situation $B_{\psi}\bs L/Q$ appears as a special case of \S\ref{orbitA}(4), and we have checked that $B_{\psi}$ acts on $L/Q$ with an open orbit with finite stabilizers.
\item $L/Q=\SL(V'_{4})/Q'_{1}\times \PP(V_{2})$, where $\SL(V'_{4})/Q'_{1}$ is the partial flag variety of $\SL(V'_{4})$ with associated graded dimensions $(2,1,1)$ (in any order). Again the situation $B_{\psi}\bs L/Q$ appears as a special case of \S\ref{orbitA}(4),  from which we know that $B_{\psi}$ acts on $L/Q$ with an open orbit with finite stabilizers.
\end{enumerate}

\subsection{Type $E_{8}$}    
Type of $\bP_{\infty}$:
\[\xymatrix{\circ\ar@{-}[r] & \circ\ar@{-}[r]& \circ\ar@{-}[r]\ar@{-}[d]& \bullet\ar@{-}[r]& \circ\ar@{-}[r]& \circ\ar@{-}[r]& \circ\ar@{-}[r]& \circ\\
& &  \circ}\]

In this case $m=5$. We have $L\cong (\SL(V_{5})\times \SL(V'_{5}))/\mu_{5}$ acting on $V^{*}_{\bP}=\wedge^{2}(V_{5})\ot V'_{5}$, here the embedding $\mu_{5}\incl \SL(V_{5})\times \SL(V'_{5})$ is $z\mapsto(z^{2}\id_{V_{5}},z\id_{V_{5}'})$.  

Choose a basis $\{x_{i}\}_{1\le i\le 5}$ of $V_{5}$, and  a basis $\{e_{i}\}_{1\le i\le 5}$ of $V'_{5}$.  Take $\psi=\sum_{i\in\ZZ/5\ZZ}x_{i-1}\wedge x_{i+1}\ot e_{i}$. Then the projection $L_{\psi}^{\circ}\to \PGL(V_{5})$ is an isomorphism onto the diagonal torus of $\PGL(V_{5})$ with respect to the basis $\{x_{i}\}$. Same for the other projection $L_{\psi}^{\circ}\to \PGL(V'_{5})$.

Potential choices of $Q$: $L/Q=\PP(V_{5}), \PP^{\vee}(V_{5}), \PP(V'_{5})$ or $\PP^{\vee}(V'_{5})$. In all these cases  $B_{\psi}=L^{\circ}_{\psi}$ acts on $L/Q$ with an open free orbit.

\appendix

\section{Factorizable module categories}

In this appendix we define and classify semisimple factorizable module categories over a neutral Tannakian category with coefficients. We will apply the classification result here to the category of semisimple perverse sheaves in the automorphic category $\cD(\psi,\chi)$ in \S\ref{s:sh}.  The materials presented here are an elementary case of the theory of chiral homology that does not involve the language of $\infty$-categories, so that we give self-contained proofs. 

\subsection{Notations} The notations used in the appendix differ from the ones in the main body of the paper. 

Let $L$ be an algebraically closed field of characteristic zero. All abelian categories in this subsection will be $L$-linear.  Let $\Vect$ denote the category of finite-dimensional vector spaces over $L$.

Let $\cP$ be a semisimple $L$-linear abelian category such that $\End_{\cP}(X)=L$ for each simple object $X\in \cP$. Let $\Irr(\cP)$ denote the set of isomorphism classes of simple objects in $\cP$. Objects in $\cP$ will be denoted $X,Y,\cdots$.

Let $(\cR,\ot)$ and $(\cC,\ot)$ be semisimple rigid tensor category over $L$.  Objects in $\cR$ and $\cC$ will be denoted $V,W,\cdots$.

\subsection{Factorizable module categories with coefficients}\label{as:define fact mod} We say that $\cP$ is a factorizable $\cR$-module category with coefficients in $\cC$, if for every finite set $I$ there is a bi-exact functor $\cR^{\boxtimes I}\times \cP\to \cC^{\boxtimes I}\boxtimes \cP$ (Deligne's tensor product)
\begin{equation*}
(V, X)\mapsto V\star_{I}X, \textup{ for }V\in \cR^{\boxtimes I}, X\in \cP
\end{equation*}
with the following extra structures:
\begin{enumerate}
\item When $I=\vn$, we understand that $\cR^{\bt \vn}\cong \Vect$, and the action of $\cR^{\bt \vn}$ is the usual action of $\Vect$ on $\cP$ by tensoring.
\item Any map of finite sets $\ph: I\to J$ induces $\ph_{\cR}: \cR^{\bt I}\to \cR^{\bt J}$ sending $\bt_{i\in I}V_{i}$ to $\bt_{j\in J}(\ot_{i\mapsto j}V_{i})$. Similarly it induces $\ph_{\cC}: \cC^{\bt I}\to \cC^{\bt J}$. Then there is a functorial isomorphism $\ph_{\cR}(V)\star_{J}X\cong (\ph_{\cC}\bt\id_{\cP})(V\star_{I}X)\in \cC^{\bt J}\bt \cP$, for $V\in \cR^{\bt I}$ and $X\in \cP$.
\item If $I=I'\sqcup I''$ is a partition of $I$ then there is a functorial isomorphism $V'\star_{I'}(V''\star_{I''}X)\cong (V'\bt V'')\star_{I}X$ for $V'\in \cR^{\bt I'}, V''\in \cR^{\bt I''}$ and $X\in \cP$. Here on the left side, when $V'$ acts on $V''\star_{I''}X\in \cC^{\bt I''}\bt \cP$, it only acts on the $\cP$-factor.
\end{enumerate}
These structures have to satisfy the usual compatibilities: composition of maps in (1), refinement of partitions in (2), and the compatibility of (1) and (2) for maps $\ph'\sqcup\ph'': I'\sqcup I''\to J'\sqcup J''$. We do not spell out the details. 

For $X'\in \cC^{\bt I}\boxtimes \cP$ and $Y\in \cP$,  let $\uHom(X',Y)$ and $\uHom(Y,X')$ denote the inner homs taking values in $\cC^{\bt I}$. For example, $\uHom(Y,X')$ is characterized by having an isomorphism $\Hom_{\cC^{\bt I}}(C, \uHom(Y,X'))\cong \Hom_{\cC^{\bt I}\bt \cP}(C\bt Y, X')$ functorial in $C\in \cC^{\bt I}$. Then the axioms imply that for $X,Y\in \cP$, $V\in \cR^{\bt I}$, there is a functorial isomorphism
\begin{equation}\label{adj Vdual}
\uHom(V\star_{I} X,Y)\cong \uHom(X, V^{\vee}\star_{I} Y) \in \cC^{\bt I}.
\end{equation}

For $I$ equal to a singleton set, we denote $V\star_{I}X$ simply by $V\star X$. The factorizable $\cR$-module structure on $\cP$ in particular gives an $E_{2}$-action $V\mapsto V\star(-)$ of $\cR$ on $\cP$.
 
\begin{exam}\label{ex:fact mod}
\begin{enumerate}
\item Let $r: \cR\to \cC$ be a tensor functor, which extends to $r^{I}: \cR^{\boxtimes I}\to \cC^{\boxtimes I}$. Then for any semisimple abelian category $\cP$, $V\star_{I}X:=r^{I}(V)\boxtimes X$ gives $\cP$ the structure of a factorizable $\cR$-module category with coefficients in $\cC$. Such $\cP$ are a categorical analogue of an eigenspace under a commutative algebra action, therefore we say that $\cP$ is eigen with eigenvalue $r:\cR\to \cC$.

\item A factorizable $\cR$-module category with coefficients in $\cC=\Vect$ is the same as an $E_{2}$-module category for $\cR$. 

\item We may combine the above two examples. Let $\cC'$ be another semisimple rigid tensor category and $r:\cR\to \cC\boxtimes\cC'$ be a tensor functor. Let $\cP$ be an $E_{2}$-module category for $\cC'$. Then $\cP$ also carries the structure of a factorizable $\cR$-module category with coefficients in $\cC$ as follows. For $V\in \cR^{\boxtimes I}$ and $X\in\cP$, let $V\star_{I}X$ be the image of $r^{I}(V)\boxtimes X$ under the functor $\id\boxtimes a'_{I}: \cC^{\boxtimes I}\boxtimes \cC'^{\boxtimes I}\boxtimes \cP\to \cC^{\boxtimes I}\boxtimes \cP$, where $a'_{I}$ is the action map of $\cC'^{\boxtimes I}$ on $\cP$. We say that $\cP$ is {\em inflated} from the $E_{2}$-action of $\cC'$ on $\cP$.

\item\label{inflate} As a special case of the above example, consider the case $\cR=\Rep(H)$ and  $\cC=\Rep(M)$ for reductive groups $H$ and $M$ over $L$. Let $\r: M\to H$ be a homomorphism and $H_{\r}$ be the centralizer of $\r(M)$, and let $\cC'=\Rep(H_{\r})$. Then we have the restriction functor $r: \cR=\Rep(H)\to \Rep(M\times H_{\r})\cong \cC\boxtimes\cC'$. For any $E_{2}$-module category $\cP$ under $\Rep(H_{\r})$, the construction in (3) gives a factorizable $\Rep(H)$-module category with coefficients in $\cC=\Rep(M)$, {\em inflated} from the $E_{2}$-action of $\Rep(H_{\r})$ on $\cP$.
\end{enumerate}
\end{exam}

\subsection{Indecomposable module categories} 
We call a factorizable $\cR$-module category $\cP$ with coefficients in $\cC$ {\em indecomposable} if $\cP$ is not the direct sum of two nonzero factorizable $\cR$-module categories with coefficients in $\cC$.

\begin{lemma}\label{l:indecomp}
Let $\cP$ be a semisimple abelian category over $L$ with finite set $\Irr(\cP)$ of simple objects up to isomorphism. Suppose $\cP$ is equipped with the structure of a factorizable $\cR$-module categories with coefficients in $\cC$.
\begin{enumerate}
\item For $X,Y\in \Irr(\cP)$, define $X\sim Y$ if for some $V\in \cR$, $V\star X$ contains $C\boxtimes Y$ as a direct summand, for some nonzero object $C\in \cC$. Then $\sim$ is an equivalence relation.
\item\label{P decomp} For each equivalence class $s\in \Irr(\cP)/\sim$, let $\cP_{s}$ be the full subcategory whose objects are direct sums of objects in $s$. Then $\cP_{s}$ is an indecomposable  factorizable $\cR$-module categories with coefficients in $\cC$, and $\cP\cong\op_{s\in \Irr(\cP)/\sim}\cP_{s}$ as factorizable $\cR$-module categories.
\end{enumerate} 
\end{lemma}
\begin{proof}
(1) The transitivity of $\sim$ is clear from the definition of the $\cR$-action. Taking $V=\one_{\cR}$ the unit in $\cR$, we see that $X\sim X$. To show that $\sim$ is reflexive, suppose $Y$ shows up in $V\star X$, then we have a nonzero map $f: V\star X\to C\bt Y$ for some $C\in \cC$. Rewrite $f$ as a map $h: C^{\vee}\to \uHom(V\star X,Y)$ in $\cC$. Using the adjunction \eqref{adj Vdual}, $h$ corresponds to a nonzero map $h': C^{\vee}\to \uHom(X,V^{\vee}\star Y)$, which corresponds to a nonzero $f': C^{\vee}\bt X\to V^{\vee}\star Y$, showing that $Y\sim X$.

(2) is clear.
%
\end{proof}

\subsection{Classification} 
Now assume $\cR=\Rep(H), \cC=\Rep(M)$ for reductive groups $H$ and $M$ over $L$.   The next result shows that any indecomposable factorizable $\Rep(H)$-module category with coefficients in $\Rep(M)$ must take the form of Example \ref{ex:fact mod}\eqref{inflate}. 

\begin{theorem}\label{th:fact mod} Let $\cP$ be a semisimple abelian category over $L$ with finitely many simple objects. Suppose $\cP$ is equipped with the structure of an indecomposable factorizable $\Rep(H)$-module category with coefficients in $\cC=\Rep(M)$. Then there is a homomorphism $\r: M\to H$, unique up to $H$-conjugation,  such that the factorizable  $\Rep(H)$-module structure on $\cP$ with coefficients in $\Rep(M)$ is inflated from an $E_{2}$-module structure of $\cP$ under $\Rep(H_{\r})$, where $H_{\r}$ is the centralizer of $\Im(\r)$ in $H$ (see Example \ref{ex:fact mod}\eqref{inflate}).
\end{theorem}

Combine this theorem with Lemma \ref{l:indecomp}, we can speak about the eigen-decomposition of a decomposable $\cP$. We give a statement that does not mention the fiber functors of $\cR$ and $\cC$ explicitly.

\begin{cor}\label{c:fact mod} Let $\cP$ be a semisimple abelian category over $L$ with finitely many simple objects. Let $\cR$ and $\cC$ be semisimple neutral Tannakian categories over $L$. Suppose $\cP$ is equipped with the structure of a factorizable $\cR$-module category with coefficients in $\cC$. Then there is a well-defined finite set of (isomorphism classes of) tensor functors $\{\s: \cR\to \cC\}_{\s\in \Sigma}$ and a unique decomposition
\begin{equation*}
\cP=\bigoplus_{\s\in \Sigma} \cP_{\s}
\end{equation*}
such that the factorizable  $\cR$-module structure on $\cP_{\s}$ with coefficients in $\cC$ is inflated from an $E_{2}$-action of  $\Rep(\Aut^{\ot}(\s))$ on $\cP_{\s}$.
\end{cor}
\begin{proof}
Choose fiber functors of $\cR$ and $\cC$ to identify them with $\Rep(H)$ and $\Rep(M)$. In the decomposition of $\cP$ into indecomposables (see Lemma \ref{l:indecomp}\eqref{P decomp}), we apply Theorem \ref{th:fact mod} to each $\cP_{s}$ to get a homomorphism $\r_{s}: M\to H$, such that $\cP_{s}$ is inflated from an $E_{2}$-action of $\Rep(H_{\r_{s}})$ on $\cP_{s}$. Now let $\Sigma$ be the set of $H$-conjugacy classes of $\{\r_{s}\}_{s\in \Irr(\cP)/\sim}$. A homomorphism $\r: M\to H$ up to $H$-conjugacy is the same datum as a tensor functor $\s:\cR\to \cC$, so we may identify $\Sigma$ with a set of tensor functors $\{\s: \cR\to\cC\}$. For $\s\in \Sigma$ with the corresponding $\r:M\to H$, Let $\cP_{\s}$ be the direct sum of $\cP_{s}$ for those $\r_{s}$ conjugate to $\r$ under $H$.  Note that $\Aut^{\ot}(\s)\cong H_{\r}$, so $\cP_{\s}$ is inflated from an $E_{2}$-action of $\Rep(\Aut^{\ot}(\s))$. 
\end{proof}

\begin{remark}\label{r:fact mod equiv} We state an equivariant version of Corollary \ref{c:fact mod}. Suppose both $\cR$ and $\cC$ are equipped with actions of a group $\G$. The action of $\g\in \G$ on $V\in \cR^{\bt I}$ and $W\in \cC^{\bt I}$ are denoted $V^{\g}$ and $W^{\g}$.  Suppose further that the action of $\cR$ on $\cP$ is equipped with functorial isomorphisms
\begin{equation*}
V^{\g}\star_{I}X\cong (V\star_{I}X)^{\g}, \quad \forall \g\in \G
\end{equation*}
compatible with the group structure on $\G$ and the factorization structure. Here the action of $\g$ on the right side is only on the $\cC^{\bt I}$-factor.  Under these assumptions, each functor $\s: \cR\to \cC$ constructed in Corollary \ref{c:fact mod} is equipped with a $\G$-equivariant structure. Therefore $\Aut^{\ot}(\s)$ also carries an action of $\G$. Moreover, the $E_{2}$-action of $\Rep(\Aut^{\ot}(\s))$ on $\cP_{\s}$ (denoted $\bu$) is equipped with a $\G$-invariant structure, i.e., functorial isomorphisms
\begin{equation*}
U\bu_{I} X\cong U^{\g}\bu_{I} X,\quad \forall \g\in \G, U\in \Rep(\Aut^{\ot}(\s))^{\bt I}, X\in \cP_{\s}
\end{equation*}
compatible with the group structure on $\G$ and the factorization structure. 
\end{remark}

The rest of the appendix is devoted to the proof of Theorem \ref{th:fact mod}.

\subsection{Proof of Theorem \ref{th:fact mod}}

First some notations. Let $\Ind \cP$ be the category of ind-objects in $\cP$: it is equivalent to $\Irr(\cP)$-graded vector spaces of possibly infinite dimension.  
We denote by $\om: \cC^{\bt I}=\Rep(M^{I})\to \Vect$ the forgetful functor for various $I$. We also denote the forgetful functor $\cC^{\bt I}\bt \cP\to \cP$ by $\om$. Let $\Irr(H)$ and $\Irr(M)$ denote the set of (isomorphism classes of) irreducible representations of $H$ and $M$.

For any $V\in \Irr(H)$ we have an embedding $m_{V}: V\ot V^{\vee}\cong\End(V)\subset \cO_{H}$ as matrix coefficients. Same for $M$.

The proof goes in several steps.

\sss{The affine scheme $\cS$}\label{sss:OS}
For each $V\in \Irr(H)$, the action $V\star X\in \Rep(M)\bt \cP$ for various $X\in \Irr(\cP)$ only involves finitely many irreducible representations $W\in \Irr(M)$. We denote this finite set by $\G_{V}\subset \Irr(M)$. Note that $\G_{V}$ always contains the trivial representation $\one_{\cR}$ of $M$, for $V\star(V^{\vee}\star X)$ contains $\one_{\cR}\bt X$ as a direct summand.

We define a moduli problem as follows. For any $L$-algebra $R$, let $\cS(R)$ be the set of Hopf algebra homomorphisms $\ph: \cO_{H}\to \cO_{M}\ot R$ such that,  for any $V\in \Irr(H)$, $\ph(m_{V}(\End(V)))$ lies in the span of $m_{W}(\End(W))\ot R$ for $W\in \G_{V}$. Then $\cS(R)$ is a subset of homomorphisms of algebraic groups $M_{R}\to H_{R}$. Note that $\cS(R)\ne\vn$ since it contains the trivial homomorphism $M_{R}\to H_{R}$ (because $\G_{V}$ contains $\one_{\cR}$).

We claim that $\cS$ is representable by an affine scheme of finite type over $L$. Indeed, choose a faithful $V_{0}\in \Irr(H)$, then $\ph\in \cS(R)$ is determined by the restriction $\ph|_{m_{V_{0}}(\End(V_{0}))}: \End(V_{0})\to \op_{W\in \G_{V_{0}}}\End(W)\ot R$, which is representable by an affine space of finite dimension. This realizes $\cS$ as a closed subscheme of an affine space. Moreover, $\cS$ carries an action of $M\times H$ by conjugation on $\cO_{M}$ and $\cO_{H}$.

We give generators and relations for the ring of regular functions $\cO_{\cS}$. This part is inspired by \cite[\S6]{LZ}. For any $f\in \cO_{H}$ and any $g\in M(L)$, define a function $\Phi_{f,g}\in \cO_{\cS}$ that assigns to each $R$-point $\ph:\cO_{H}\to \cO_{M}\ot R$ the value $\Phi_{f,g}(\ph)=\ev_{g}\ph(f)\in R$, where $\ev_{g}: \cO_{M}\ot R\to R$ is the evaluation at $g$.  

The functions $\{\Phi_{f,g}\}_{f\in \cO_{H}, g\in M(L)}$ generate $\cO_{\cS}$ as an $L$-algebra. Indeed it suffices to run $f$ through a basis of the matrix coefficients for a faithful $V_{0}\in \Irr(H)$, and take a finite set of $g_{i}$ such that their images in $\prod_{W\in \G_{V_{0}}}\End(W)$ span.

We now give the relations among $\{\Phi_{f,g}\}_{f\in \cO_{H}, g\in M(L)}$. We claim that the relations are generated by the following:
\begin{enumerate}
\item For any $g\in M(L)$, the assignment $f\mapsto \Phi_{f,g}$  is $L$-linear.
\item Let $V\in \Irr(H)$. Then for any finite $L$-linear combination $\sum_{i}c_{i}g_{i}$ of elements in $M(L)$ such that $\sum_{i}c_{i}g_{i}|_{W}=0$ for all $W\in \G_{V}$, then $\sum_{i}c_{i}\Phi_{f,g_{i}}=0$ for all $f\in m_{V}(\End(V))$.
\item For any $f, f'\in \cO_{H}$ and $g\in M(L)$, we have $\Phi_{ff',g}=\Phi_{f,g}\Phi_{f',g}$.
\item For any $f\in \cO_{H}$ and $g,g'\in M(L)$, we have $\Phi_{f,gg'}=\sum_{i}\Phi_{f_{i},g}\Phi_{f'_{i}, g'}$ if $\D(f)=\sum_{i}f_{i}\ot f'_{i}$ for the comultiplication $\D$ on $\cO_{H}$.
\end{enumerate}
It is easy to see that these relations indeed hold in $\cO_{\cS}$. To show they are all the relations, suppose we are given an assignment $\Phi_{f,g}\to \ph_{f,g}\in R$ satisfying the above relations, we show how to construct a Hopf algebra map $\ph: \cO_{H}\to \cO_{M}\ot R$ such that $\Phi_{f,g}$ evaluated at $\ph\in \cS(R)$ is $\ph_{f,g}$. For any $f\in \cO_{H}$,  relation (2) ensures that there exists a unique element $\ph(f)\in \op_{W\in \G_{V}}m_{W}(\End(W))\ot R\subset \cO_{M}\ot R$ such that $\ev_{g}\ph(f)=\ph_{f,g}$ for any $g\in M(L)$. Relation (1) says that the assignment $f\mapsto \ph(f)$ gives a linear map $\ph: \cO_{H}\to \cO_{M}\ot R$. Relation (3) shows that $\ph$ is an algebra homomorphism. Relation (4) shows that $\ph$ is a coalgebra homomorphism. This proves the claim.

\sss{Construction of an action of $\cO_{\cS}$ on  each object in $\cP$}We view $\cO_{\cS}$ as an algebra ind-object in $\Rep(H)$ using the conjugation action of $H$. For each $X\in\cP$ we will construct a map $\a_{X}: \om(\cO_{\cS}\star X)\to X$ in $\Ind\cP$,  compatible with the algebra structure on $\cO_{\cS}$, such that the following diagram is commutative for any $V\in \Rep(H)$ and $X\in \cP$ 
\begin{equation*}
\xymatrix{ V\star \om(\cO_{\cS}\star X) \ar[r]^-{\sim}\ar[d]^{\id_{V}\star\a_{X}} & 
(\om\bt \id_{\cC})(\cO_{\cS}\star (V\star X))\ar[d]^{\a_{V\star X}} \\
V\star X \ar@{=}[r] & V\star X}
\end{equation*}
Here the top row is induced by the commutativity constraint of the action of $\cR$ on $\cP$; in $(\om\bt \id_{\cC})(\cO_{\cS}\star (V\star X))$ we emphasize that $\om$ is applied to the first factor of $\cC$, so the result is still an object in $\cC\bt \cP$. Moreover, all maps in $\Ind\cP$ are compatible with the $\cO_{\cS}$-actions. See \cite[\S22]{G-Cat}.

In other words, if we define an internal $\uHom(X,Y)\in \Ind\Rep(H)$ for $X,Y\in \cP$ by the adjunction $\Hom_{\cP}(\om(V\star X), Y)=\Hom_{H}(V,\uHom(X,Y))$ for all $V\in \Rep(H)$, then we need to construct an algebra homomorphism $\a'_{X}: \cO_{\cS}\to \uEnd(X)$ in $\Ind\Rep(H)$, that is compatible with morphisms in $\cP$. 

The construction of $\a_{X}$ is analogous to V. Lafforgue's excursion operators \cite{La18}. 
%
For any $g\in M(L)$ we first define a map $\a_{g,X}: \om(\cO_{H}\star X)\to X$ as the composition
\begin{equation}\label{agX}
\xymatrix{\om(\cO_{H}\star X)\ar@{=}[r] & \bigoplus_{V\in \Irr(H)}\om((V\bt V^{\vee})\star_{\{1,2\}}X)\ar[d]_{(g,1)}\\
& \bigoplus_{V\in \Irr(H)}\om((V\bt V^{\vee})\star_{\{1,2\}}X)\ar@{=}[r] & \om(\cO_{H}\star X)\ar[r]^{\ev_{1}} & \one_{\cR}\star X=X.
}
\end{equation}
Here the second step uses that $(V\bt V^{\vee})\star_{\{1,2\}}X\in \Rep(M^{2}) \bt \cP$, hence $(g,1)\in M^{2}$ acts on $\om((V\bt V^{\vee})\star_{\{1,2\}}X)$. The last map $\ev_{1}:\cO_{H}\to \one_{\cR}$ is evaluation at $1\in H$. 

The map $\a_{g,X}$ then gives a map $\a'_{g,X}: \cO_{H}\to \uEnd(X)$ in $\Ind\Rep(H)$. It is supposed to be the composition
\begin{equation*}
\cO_{H}\xr{\Phi_{\bu, g}}\cO_{\cS}\xr{\a'_{X}}\uEnd(X)
\end{equation*}
where $\Phi_{\bu, g}$ denotes the $H$-equivariant map $\cO_{H}\to \cO_{\cS}$ that sends $f\in \cO_{H}$ to $\Phi_{f,g}$. We need to check the ring relations (2)-(4) in \S\ref{sss:OS} hold for $\a'_{g,X}$ to ensure that the maps $\{\a'_{g,X}\}_{g\in M(L)}$ together  give an algebra map $\cO_{\cS}\to \uEnd(X)$.
\begin{enumerate}
\item[(2)] This follows from the definition of $\G_{V}$: in $V\bt V^{\vee}\star_{\{1,2\}} X= V^{\vee}\star_{\{2\}}(V\star_{\{1\}}X)\in \cC^{\bt 2}\bt \cP$, the first factor only involves $W\in \G_{V}\subset \Irr(M)$, and $\sum c_{i}g_{i}$ acts on these $W$ by zero.

\item[(3)] We need to show that the following diagram is commutative
\begin{equation*}
\xymatrix{\om(\cO_{H}\bt \cO_{H}\star_{\{1,2\}} X) \ar[d]^{\mult}\ar[rr]^-{\id_{\cO_{H}}\bt\a_{g,X}} && \om(\cO_{H}\star X) \ar[r]^-{\a_{g,X}} & X\ar@{=}[d]\\
\om(\cO_{H}\star X) \ar[rrr]^{\a_{g,X}} & && X}
\end{equation*}
Here the map ``mult'' is induced by the multiplication $\cO_{H}\ot \cO_{H}\to \cO_{H}$.Let $V,V',W\in \Irr(H)$. Let $m^{V,V'}_{W}: \End(V)\ot \End(V')\to \End(W)$ be the composition of the multiplication of matrix coefficients in $\cO_{H}$ with the projection to $\End(W)$. Since $m^{V,V'}_{W}$ is equivariant under left and right $H$-action, it is induced from a pair of maps $\mu^{V,V'}_{W}: V\ot V'\to W$ and $\nu^{V,V'}: V^{\vee}\ot V'^{\vee}\to W^{\vee}$ in $\Rep(H)$ such that $m^{V,V'}_{W}=\mu^{V,V'}_{W}\bt \nu^{V,V'}_{W}$ in $\Rep(H^{2})$. Then the required commutativity restricted to $m_{V}(\End(V))\bt m_{V'}(\End(V'))$ follows from the following commutative diagram in $\Rep(H^{2})$
\begin{equation*}
\xymatrix{\om(V\bt V^{\vee}\bt V'\bt V'^{\vee}\star_{\{1,2,1',2'\}}X) \ar[r]^{(g,1,g,1)}\ar@{=}[d] &   
\om(V\bt V^{\vee}\bt V'\bt V'^{\vee}\star_{\{1,2,1',2'\}}X)\ar@{=}[d]\\
\om((V\ot V')\bt(V^{\vee}\ot V'^{\vee})\star_{\{1,2\}} X) \ar[r]^{(g,1)}\ar[d] & \om((V\ot V')\bt(V^{\vee}\ot V'^{\vee})\star_{\{1,2\}} X) \ar[d]\\
\om(W\bt W^{\vee}\star_{\{1,2\}}X) \ar[r]^{(g,1)} &  \om(W\bt W^{\vee}\star_{\{1,2\}}X) }
\end{equation*}

\item[(4)] We need to show, for any $g,g'\in M(L)$, the following diagram is commutative in $\Ind\Rep(H)$
\begin{equation*}
\xymatrix{\cO_{H}\ar[d]^{\D}\ar[rr]^-{\a'_{gg',X}} && \uEnd(X)\\
\cO_{H}\ot\cO_{H}\ar[rr]^-{\a'_{g,X}\ot \a'_{g',X}} && \uEnd(X)\ot \uEnd(X)\ar[u]}
\end{equation*}
Here the right vertical map is given by composition of internal Hom. Restricting to the sub-coalgebra $\End(V)\subset \cO_{H}$ for any $V\in \Irr(H)$, we need to show the following diagram is commutative
\begin{equation*}
\xymatrix{\om(V\bt V^{\vee}\star X)\ar[r]^{(gg',1)} \ar[d]^{\id_{V}\ot\textup{coev}_{V}\ot\id_{V^{\vee}}} & \om(V\bt V^{\vee}\star X) \ar[d]^{\ev_{V}}\\
\om(V\bt (V^{\vee}\ot V)\bt V^{\vee} \star X)\ar@{=}[d]  & X\\
\om(V\bt V^{\vee}\bt V\bt V^{\vee} \star X) \ar[r]^{(g,1,g',1)} &  \om(V\bt V^{\vee}\bt V\bt V^{\vee} \star X) \ar[u]_{\ev_{V}\bt\ev_{V}}
}
\end{equation*}
Both compositions $\om(V\bt V^{\vee}\star X)\to X$ are adjoint to the following map 
\begin{equation*}
\om(V\star X)\xr{g'}\om(V\star X)\xr{g}\om(V\star X).
\end{equation*}
This finishes the proof of relation (4).
\end{enumerate}

\sss{Construction of $\r$} The action of $\cO_{\cS}$ on $X\in \cP$ gives the following map in $\cP$ (here we are viewing $\cO^{H}_{\cS}$ as a trivial $H$-submodule of $\cO_{\cS}$)
\begin{equation*}
\cO^{H}_{\cS}\ot X \to \om(\cO^{H}_{\cS}\star X)\to \om(\cO_{\cS}\star X)\xr{\a_{X}} X. 
\end{equation*}
This gives an action of $\cO^{H}_{\cS}$ (as a plain $L$-algebra) on $X\in \Ind\cP$, commuting with all morphisms in $\Ind\cP$ (i.e., $\cO^{H}_{\cS}$ acts on $\id_{\cP}$). For each $X\in \Irr(\cP)$, Since $\End(X)=L$, this action factors through a homomorphism $\th_{X}: \cO^{H}_{\cS}\to L$. Since all morphisms in $\cP$ commute with the $\cO_{\cS}^{H}$-action, all simple objects $X$ with the same $\th_{X}$ form a union of equivalence classes in the sense of Lemma \ref{l:indecomp}. Since $\cP$ is indecomposable, there is only one equivalence class on $\Irr(\cP)$, hence all simple objects $X$ have the same $\th_{X}$, which we denote by $\th$.

Let $\cI\subset\cO_{\cS}$ be the ideal generated by $\ker(\th)$. Let $\wt \cZ\subset \cS$ be the closed subscheme defined by $\cI$, and  $\cZ:=\wt \cZ_{\red}\subset \cS$ the reduced scheme. Since $M$ is reductive, conjugacy classes of homomorphisms $M\to H$ form a discrete set, hence $\cS_{\red}$ is a disjoint union of $H$-orbits.  Since the only $H$-invariant functions on $\cZ$ are the scalars, it is a single $H$-orbit, i.e., the $H$-orbit of some homomorphism $\r:M\to H$.  Therefore $\cZ\cong H/H_{\r}$.

\sss{The action of $\cO_{\cS}$ on $X\in \cP$ factors through $\cO_{\cZ}$}  By the previous step, the action of $\cO_{\cS}$ on any $X\in \cP$ factors through the quotient $\cO_{\wt \cZ}$, i.e., a map  $\ov\a_{X}: \om(\cO_{\wt \cZ}\star X)\to X$. We show that $\ov\a_{X}$ further factors through a map $\b_{X}: \om(\cO_{\cZ}\star X)\to X$. 

Let $\cJ\subset \cO_{\wt \cZ}$ be the nilpotent radical. Since $\cS$ is of finite type over $L$, $\cJ^{n}=0$ for some $n>0$. Let $X\in \Irr(\cP)$, we show that the restriction of $\ov\a_{X}$ to $\om(\cJ\star X)$ is zero. If not, let $\ell$ be the smallest positive integer such that $\ov\a_{X}|\om(\cJ^{\ell}\star X)=0$. Since $\ov\a_{X}: \om(\cJ^{\ell-1}\star X)\to X$ is nonzero, it is surjective since $X$ is simple. Applying $\cJ$-action we still get a surjection $\om(\cJ\star\om(\cJ^{\ell-1}\star X))\surj\om(\cJ\star X)$. We have a commutative diagram
\begin{equation*}
\xymatrix{  \om(\cJ\star\om(\cJ^{\ell-1}\star X))\ar[d] \ar[r] & \om(\cJ\star X)\ar[d]         \\
\om(\cJ^{\ell}\star X)\ar[r] & X}
\end{equation*}
where the left arrow is given by the multiplication $\cJ\star \cJ^{\ell-1}\to \cJ^{\ell}$, and all other maps are the action maps. Now the top arrow is surjective and the bottom one is zero by assumption, which implies the right arrow is zero.  This shows that $\ov\a_{X}: \om(\cJ\star X)\to X$ is zero, i.e., the action of $\cO_{\wt \cZ}$ on $X$ factors through $\cO_{\cZ}$.


In the sequel we denote the action map of $\cO_{\cZ}$ on $X$ by
\begin{equation*}
\b_{X}: \om(\cO_{\cZ}\star X)\to X.
\end{equation*}

\sss{Construction of an $E_{2}$-action of $\Rep(H_{\r})$ on $\cP$} We will denote this action by $\bu$.  Since the image of the restriction functor $\Rep(H)\to \Rep(H_{\r})$ generate $\Rep(H_{\r})$ under taking direct summands, it suffices to define $V\bu X:=\om(V\star X)$ for any $V\in \Rep(H)$,  and check that any $H_{\r}$-equivariant map $V\to V'$ induces a map $\om(V\star X)\to \om(V'\star X)$ in a way functorial in $V,V'$ and $X$ and compatible with compositions.    

In the previous step we have shown that  $\cO_{\cZ}=L[H]^{H_{\r}}$ (right translation invariants) acts  on $X$. For any $W\in \Rep(H)$, we have a map $m_{W,\cZ}: W^{H_{\r}}\ot W^{\vee}\to \cO_{\cZ}$ in $\Rep(H)$ given by the taking the matrix coefficient $w\ot \xi\mapsto (h\mapsto \j{\xi, hw})$. In particular, taking $W=V^{\vee}\ot V'$ for $V,V'\in \Rep(H)$, we get  a map $m_{W,\cZ}:\Hom_{H_{\r}}(V,V')\ot (V\ot V'^{\vee})\to \cO_{\cZ}$ in $\Rep(H)$. Composed with the action map $\b_{X}$ we get
\begin{equation}\label{pre gam}
\wt\g_{V,V',X}: \Hom_{H_{\r}}(V,V')\ot \om( (V\ot V'^{\vee})\star X)\xr{m_{W,\cZ}\star\id_{X}} \om(\cO_{\cZ}\star X)\xr{\b_{X}} X
\end{equation}
which is the same as a map
\begin{equation*}
\g_{V,V',X}: \Hom_{H_{\r}}(V,V')\to \Hom_{\cP}(\om((V\ot V'^{\vee})\star X), X)=\Hom_{\cP}(\om(V\star X), \om(V'\star X)).
\end{equation*}
One checks these maps are compatible with compositions $V\to V'\to V''$ using the fact that $\b_{X}$ is compatible with the ring structure on $\cO_{\cZ}$. This finishes the construction of $V\bu X$.

\sss{The action of $\Rep(H)$ on $\cP$ is inflated from the $\Rep(H_{\r})$-action}  To show this, we need to check that for any $g\in M(L), V\in \Rep(H)$ and $X\in \cP$, the following two endomorphisms of $\om(V\star X)$ are the same: the first one, which we denote by $a_{g}$,  is obtained by the action of $g$ on the forgetful functor $\om$; the second one, which we denote by $b_{g}$, comes by evaluating $\g_{V,V,X}$ at $\r(g)\in \Hom_{H_{\r}}(V,V)$.

Consider the map $c_{V,g}: V\ot V^{\vee}\to \cO_{\cZ}$ given by $\xi\ot v\mapsto c_{V,g}(h)=\j{\xi, h\r(g)h^{-1}v}=\j{h^{-1}\xi, \r(g)h^{-1}v}$. We are using that $\r(g)$ commutes with $H_{\r}$ to conclude that $c_{V,g}$ is right invariant under $H_{\r}$ hence giving a function on $\cZ$. The map $c_{V,g}$ is $H$-equivariant for the diagonal action of $H$ on $V\ot V^{\vee}$. Consider the following map
\begin{equation*}
\d_{V,g,X}: \om((V\ot V^{\vee})\star X)\xr{c_{V,g}\star\id_{X}} \om(\cO_{\cZ}\star X)\xr{\b_{X}} X. 
\end{equation*}
By adjunction it gives an endomorphism $d_{g}$ of $\om(V\star X)$. On the one hand, the definition of the $\cO_{\cS}$-action, see \eqref{agX},  implies that $d_{g}=a_{g}$; on the other hand, comparing $\d_{V,g,X}$ to the map $\wt\g_{V,V,X}$ in \eqref{pre gam}, we see that $\d_{V,g,X}$ is the restriction of $\wt\g_{V,V,X}$ to $\r(g)\ot \om((V\ot V^{\vee})\star X)$, hence $d_{g}=b_{g}$. This shows $a_{g}=b_{g}$ and finishes the proof of Theorem \ref{th:fact mod}. \qed

\bibliographystyle{alpha}
\bibliography{mybib}

\end{document}